\documentclass[11pt,letterpaper]{article}
\usepackage[margin=1in]{geometry}
\usepackage[T1]{fontenc}
\usepackage{microtype}
\usepackage{amsmath,amssymb,amsthm,mathtools}
\usepackage[sc,osf]{mathpazo}
\usepackage[round,authoryear]{natbib}
\usepackage{tikz}
\usepackage{float}
\usepackage[hidelinks]{hyperref}

\newtheorem{theorem}{Theorem}[section]
\newtheorem{proposition}[theorem]{Proposition}
\newtheorem{corollary}[theorem]{Corollary}
\newtheorem{lemma}[theorem]{Lemma}
\numberwithin{equation}{section}
\newcommand{\E}{\mathbb E}
\newcommand{\Prob}{\mathbb P}
\newcommand{\R}{\mathbb R}
\newcommand{\N}{\mathbb N}
\newcommand{\one}{\mathbf 1}
\newcommand{\Ber}{\operatorname{Bernoulli}}
\newcommand{\Cov}{\operatorname{Cov}}
\newcommand{\cum}{\operatorname{cum}}
\newcommand{\risk}{\mathfrak P}
\newcommand{\aff}{\operatorname{aff}}
\newcommand{\Var}{\operatorname{Var}}
\newcommand{\TV}{\operatorname{TV}}

\title{Identifiable link probabilities from aggregated relational data:\\
when does one network suffice?}
\author{Yen-hsuan Tseng\\[2pt]\small\href{mailto:r00323024@ntu.edu.tw}{\texttt{r00323024@ntu.edu.tw}}}
\date{November 2026}
\hypersetup{
  pdftitle={Identifiable link probabilities from aggregated relational data: when does one network suffice?},
  pdfauthor={Yen-hsuan Tseng},
  pdfsubject={Working paper},
  pdfkeywords={community recovery, aggregated relational data, minimax risk, learning from counts}
}

\begin{document}
\maketitle
\begin{abstract}
We prove that even when every node reports its contact counts and their
joint distribution identifies every pair's link probability, the minimax
error from one network can stay bounded away from zero as the network
grows. In a logistic model with two equally sized communities, the
baseline link probabilities are $1/2\pm\kappa$ for fixed known weak
signal $\kappa>0$. The questionnaire partitions the roster into $K$
groups of size $m$. Let $b_k$ be the difference between the two community
counts in group $k$. Uniformly over assignments with these imbalances
and additive node effects of bounded total energy, consistent estimation
in average squared error over pairs is possible if and only if
$\sum_{k=1}^K b_k^2/m+K/m^2\to\infty$.
At exact group balance and zero activity, the ratio of the minimax risk
to its value without counts tends to one if $K/m^2\to0$. A within-group
exchange of opposite labels preserves both marginal count laws of an
affected cross-group block. The squared Hellinger distance between its
joint laws is of order $m^{-2}$. Without knowing the imbalances, a polynomial-time estimator
matches the likelihood's exponential dependence on the two design terms,
with an additive $O((Km)^{-2})$ risk term for node activity.
\end{abstract}

\section{Introduction}
Suppose every resident reports how many people they know in each of two
equally large districts, east and west. Such counts are aggregated
relational data (ARD) \citep{breza2023consistently}. \citet{breza2020using}
use household ARD to fit link probabilities and simulate village networks.
How accurately can one such survey estimate pairwise link probabilities?
Their sum over a list gives each resident's expected number of contacts
on it. Consider two lists, each naming half the city, and divide expected
contacts by the city population. For the eastern district, the normalized
eastern count is consistent. For a fixed half of each district, every
estimator can incur nonvanishing worst-case squared error averaged over
residents and graphs. This lower bound requires changes in pairwise
probabilities that survive summation over the list.

Suppose each district is evenly divided between two latent communities.
Edges are independent, with probabilities $1/2+\kappa$ within communities
and $1/2-\kappa$ between them for fixed small $\kappa>0$. As the districts
grow, the minimax mean squared error over pairs is asymptotically the same
with all the reports as with none. Swapping two opposite labels in the
east preserves the distribution of each district's reports about the other.
Their joint law distinguishes the swap through the shared edges. Observing
the edges themselves permits consistent estimation of the pairwise
probabilities and of normalized contacts with either list.

\begin{samepage}
How finely must a questionnaire divide the roster to recover those
probabilities from one network? With $r=Km$ residents in $K$ groups of
size $m$, uniform consistency over all globally balanced community
assignments holds exactly when $K/r^{2/3}\to\infty$. At zero activity,
if the labels are drawn uniformly over those assignments independently
of the partition, the risk averaged over assignments and graphs vanishes
exactly when $K\to\infty$.
\par\end{samepage}

The difference comes from the community composition of the survey
groups. Write $b_k$ for the difference between the two community sizes
in group $k$. Imbalances separate the mean response profiles and
contribute $\Lambda(b)=\|b\|_2^2/m$ to the threshold. At exact balance, the joint
fluctuations of the counts on the two sides of each cross-group block
supply the remaining term $K/m^2$. A polynomial-time estimator attains
the consistency threshold without knowing the composition. At exact
group balance and zero activity, it also attains the minimax rate on
the logarithmic scale. As $K/m^2\to\infty$, the negative logarithm of
its worst-case risk is of order $K/m^2$ at fixed signal.%
\footnote{The first version of this paper
(\href{https://arxiv.org/abs/2609.07077v1}{arXiv:2609.07077v1})
established the two-group impossibility result.
\href{https://arxiv.org/abs/2609.07077v2}{Version 2} gave the
composition criterion, the exponential risk bound for the full
likelihood and the fixed-list contact criterion (Supplement, Theorem~S7.1).
The present work attains the full likelihood's exponential risk scale
with an estimator that adapts to unknown composition in $O(r^4)$ arithmetic
operations (Theorem~\ref{thm:poly-recovery}). It also gives a polynomial-time
estimator attaining the fixed-list risk bound (Theorem~\ref{thm:exposure-recovery}).}

\paragraph{What a population law tells us.}
The full count law identifies every off-diagonal edge probability.
Let $A$ be the adjacency matrix, $T$ the group indicators, and $Y=AT$
the observed counts. For any partition the vector of total degrees
$D=Y\one_K$ is observable. For independent Bernoulli edges with
$0<p_{ij}<1$ and $i\ne j$, $\Cov(D_i,D_j)=p_{ij}(1-p_{ij})$ and
$\cum(D_i,D_i,D_j)=p_{ij}(1-p_{ij})(1-2p_{ij})$.
Only the shared edge contributes to these moments. Dividing the third
cumulant by the covariance gives $1-2p_{ij}$, so even one group suffices
for population identification
\citep{chen2010tomography}. Each roster supplies a single realization
of $Y$ from which to estimate all its named pairs.
An identification--estimation gap also arises for peer-effect coefficients
on an observed network. Under independently randomized treatments,
\citet{hayes2025gap} derive a minimax lower bound from near-collinearity
of the peer regressors. Here the unknowns are the link probabilities
themselves. Comparing counts with full adjacency on the same roster
isolates the effect of aggregation on their estimation.

\paragraph{Related work.}
\citet{breza2023consistently} give conditions for consistent estimation
of network statistics from ARD. Their community classification procedure
uses differences in the mean rates of contact with the observed traits.
\citet{alidaee2020recovering} bound the estimation error for a matrix
of link probabilities under a random trait design independent of the
network. We fix the survey partition and characterize recovery even
when both communities have the same mean count profile.

In simulations of peer-effect estimation, \citet{hayes2026smoothing} use
the leading left singular vectors of $AW$, where the observed traits
$W$ are correlated with latent network positions. In their conference
presentation on trait selection, \citet{hayes2026traits} give spectral
bounds for estimating a low-rank expected adjacency matrix when the
smallest canonical correlation between its signal subspace and the
trait subspace is positive, under degree and conditioning assumptions.
For our partition, let $s_i\in\{-1,1\}$ denote node $i$'s community.
The columns of $T/\sqrt m$ are orthonormal and $T^\top s=b$, so the
squared length of the projection of $s/\sqrt r$ onto the trait subspace
is $\|b\|_2^2/(mr)=\Lambda(b)/r$. At zero activity and exact balance, the community
contrast has zero projection for every $K$. The term $K/m^2$ measures
how much information accumulates in the joint fluctuations of the counts.

\citet{davison2026community} study the transpose observation $OA$,
which equals $Y^\top$ when $O=T^\top$. Their degree example has
within- and between-community probabilities $p$ and $q$. Even when the
communities have the same mean profile, the covariance between two
degrees distinguishes the two types of pair if $p\ne q$ and $p+q\ne1$
(Remark~4.7). In our homogeneous model $p+q=1$, so that covariance
has the same positive value for both types. At group balance a mixed
third moment still distinguishes them. We quantify the resulting
separation for the entire joint law of the two margins. With separated
mean profiles, their Lemma~D.5, equation~(78), gives a spectral recovery
bound under its signal and degree conditions.
They also study surveys with additional questions that count length-two
paths to each trait group, excluding returns to the respondent
(Section~5.3). Our mixed third moment uses the original direct-contact
counts reported by the two endpoint groups.

In a bibliographic network, \citet{jones2025scalable} aggregate
citations by journal and rank sampled pairs of articles by their fitted
citation probabilities. The loss studied here evaluates estimated
probabilities for named pairs directly.

Learning from network projections also includes exact community recovery
from hyperedge co-occurrence counts \citep{gaudio2023similarity} and
efficient reconstruction of three-uniform hypergraphs up to the
information threshold \citep{bresler2024projections}. In ARD the reports
are linear functions of graph edges, and our target is their probability
matrix. With the edges observed, \citet{amini2013pseudo} use block
counts to refine an initial grouping by pseudo-likelihood.
\citet{zhang2016minimax} establish exponential community misclassification
rates and \citet{gao2017achieving} give a polynomial-time refinement
procedure. In our experiment the edges incident to a held-out node enter
through group counts. We use both margins of the held-out blocks to
compare that node with a reference node.

\section{Observation, loss and recovery thresholds}\label{sec:model}
There are $r=Km$ nodes and $K$ known disjoint trait groups
$G_1,\ldots,G_K$, each of size $m\in4\N$. Let $T_{ik}=\one_{\{i\in
G_k\}}$. The graph is undirected and hollow. Conditional on fixed
parameters, the distinct unordered edges are independent and
\begin{equation}\label{eq:model}
 A_{ij}\sim\Ber(p_{ij}),\qquad p_{ij}=\sigma(\alpha_i+\alpha_j+\tau s_i s_j),
 \qquad \sigma(x)=(1+e^{-x})^{-1},\quad i<j .
\end{equation}
Here $s_i\in\{-1,1\}$ is the unknown community label, $\alpha_i$ is
additive activity, and $P_{\alpha,s}$ has off-diagonal entries $p_{ij}$
and zero diagonal. This is the rank-one case of the logistic inner-product
model of \citet{ma2020universal}. Fix the known signal $\tau$ and an
activity radius $M\ge0$, with
\begin{equation}\label{eq:signal}
 0<\tau\le\operatorname{atanh}(1/16),\qquad \kappa=\sigma(\tau)-1/2 .
\end{equation}
The parameter class is
\begin{equation}\label{eq:class}
 \Theta_r(M)=\left\{(\alpha,s):\one^\top s=0,\quad
                  \tfrac14\sum_{i<j}(\alpha_i+\alpha_j)^2\le M^2\right\}.
\end{equation}
At $M=0$ this gives $\alpha=0$. For within- and between-community
probabilities $p,q$, the difference of the two shared-edge covariances
is $p(1-p)-q(1-q)=(p-q)(1-p-q)$. Our homogeneous model has $p+q=1$.
This cancels the covariance contrast. After centering each edge at
$1/2$, the two characteristic functions are complex conjugates, as used by the
two-frequency estimator in Section~\ref{sec:poly-main}.

At fixed $M$, the identity
$\sum_{i<j}(\alpha_i+\alpha_j)^2=(r-2)\|\alpha\|_2^2+(\sum_i\alpha_i)^2$
bounds the root mean square activity by $2M/\sqrt{r(r-2)}$.
This shrinking neighborhood gives a bounded second moment for the
likelihood ratio relative to zero activity. We use it to transfer risk
bounds without estimating the individual activities. The small-signal
range supplies uniform constants in the full-law comparison.
Divisibility of $m$ by four allows the within-group pair and anchor
constructions in Appendix~\ref{sec:lower}.

Every node reports its counts $Y_{ik}=\sum_{j\in G_k} A_{ij}$.
For $g\ne h$, the reports of $G_g$ towards $G_h$ and those of $G_h$
towards $G_g$ are the two margins of one $m\times m$ block. Within a
group the reports are degrees of its induced subgraph. Different
unordered group pairs and the within-group subgraphs use disjoint sets
of edges. Our loss on named pairs is
\[
 L_r(\widehat P,P)=\frac1{r(r-1)}\sum_{i\ne j}
                   (\widehat p_{ij}-p_{ij})^2 .
\]
Global reversal of $s$ preserves $P$. Exchanging two opposite labels
at specified nodes generally changes it.

Set $b=T^\top s$. A composition $b$ is feasible when every $b_k$ is
even, $|b_k|\le m$ and $\sum_k b_k=0$. For such $b$, let
$\Theta_r(M;b)=\{(\alpha,s)\in\Theta_r(M):T^\top s=b\}$. The
fixed-composition minimax risk is
\[
 \risk_{r,b}^Y(T)=\inf_{\widehat P}\ \sup_{(\alpha,s)\in\Theta_r(M;b)}
                     \E_{\alpha,s}L_r(\widehat P,P_{\alpha,s}).
\]
For each parameter value the loss averages over named pairs and the
expectation averages over the graph and any estimator randomization.
The supremum ranges over label assignments and activities at the
specified composition. The estimator in the infimum may know $b$.
Define
\[
 \Lambda(b)=\|b\|_2^2/m,\qquad \rho=K/m^2 .
\]
We call $\Lambda(b)$ the alignment of the trait groups with the latent
communities and $\rho$ the resolution of the questionnaire.

\begin{theorem}[Recovery boundary]\label{thm:composition-boundary}
Fix $M,\tau$ as above and let $r=Km\to\infty$. For every sequence of
feasible compositions, $\risk_{r,b}^Y(T)\to0$ if and only if
$\Lambda(b)+\rho\to\infty$. On sequences with bounded $\Lambda(b)+\rho$,
the risk has positive lower limit.
\end{theorem}

\begin{theorem}[Recovery in polynomial time]\label{thm:poly-recovery}
A polynomial-time estimator using $Y,T,\tau$ satisfies, simultaneously
for every feasible $b$,
\begin{equation}\label{eq:poly-adaptive-risk}
 \sup_{\Theta_r(M;b)}\E_{\alpha,s} L_r(\widehat P,P_{\alpha,s})
 \le C_{\tau,M}\kappa^2 e^{-c_\tau(\Lambda(b)+\rho)}
                                +\frac{M^2}{r(r-1)} .
\end{equation}
At exact group balance $b=0$, a polynomial-time estimator using cubic
scores and majority votes satisfies
\begin{equation}\label{eq:panel-risk}
 \sup_{\Theta_r(M;0)}\E_{\alpha,s}L_r(\widehat P^{\mathrm{vote}},P_{\alpha,s})
 \le C_M\kappa^2 e^{-c\kappa^2\rho}+\frac{M^2}{r(r-1)},
\end{equation}
where $c>0$ is absolute.
\end{theorem}

Alignment and resolution arise from different observable features. Subtract $m/2$ from
each count and correct for the missing diagonal. At zero activity the
resulting matrix has mean $\kappa s b^\top$. Each row has signal
$\kappa s_i b$ and count variances of order $m$, giving mean information
on the scale $\kappa^2\Lambda(b)$. At $b=0$ this signal vanishes.
For a within-group swap at exact balance
there are $K-1$ cross-group blocks, each with squared Hellinger distance
of order $m^{-2}$. The swap's within-group squared Hellinger contribution is
$O_\tau(m^{-1})$. When $\Lambda+\rho$ is bounded, $r=\rho m^3\to\infty$
forces $m\to\infty$, so this last term vanishes. The cross-group
information accumulates to the resolution scale.

\begin{corollary}[Likelihood and balanced minimax risk]\label{cor:exponential-rates}
Maximizing the zero-activity likelihood of the full array over globally
balanced signs gives the bound in \eqref{eq:poly-adaptive-risk}, with
$C_M$ in the prefactor and $c_\tau=\kappa^2/1600$ admissible. At $b=0$,
write $R_r(M)=\risk_{r,0}^Y(T)$. As $\rho\to\infty$,
\[
 -\log R_r(0)\asymp_\tau\rho,\qquad
 -\log R_r(M)\asymp_{\tau,M}\min\{\rho,\log r\}\quad(M>0\ \text{fixed}).
\]
\end{corollary}

The likelihood and polynomial-time upper bounds have the same
dependence on $\Lambda+\rho$ in the exponent for fixed signal. Matching
minimax bounds on the logarithmic scale hold at exact balance.
At positive activity radius an
$r^{-2}$ risk floor persists even if the true labels and all edges are
given to the estimator. A common activity shift of order $r^{-1}$
changes every edge probability by that order while keeping the
full-adjacency divergence bounded (Appendix~\ref{sec:exponential-rates}).

\begin{corollary}[Risk relative to observing no counts]\label{cor:no-count-risk}
Suppose $b=0$ and $M=0$. Let $B_{K,m}$ be the minimax risk under $L_r$
when the estimator knows the partition, group balance and signal, and
receives no counts. If $r=Km\to\infty$ and $K/m^2\to0$, then
\[
 \frac{R_r(0)}{B_{K,m}}\longrightarrow1,\qquad
 B_{K,m}=\kappa^2\left(1-\frac{1}{(r-1)(m-1)}\right).
\]
\end{corollary}

The proof compares the full count laws conditional on a named pair
belonging to the same or to different communities
(Appendix~\ref{sec:no-count-risk}).

\section{Information in the joint count law}\label{sec:information}

Consider the edges between two trait groups at zero activity. Write $R_i$
and $C_j$ for the row and column totals of their $m\times m$ edge
block. With row signs $a$ and column signs $c$, the entries are independent,
with
\begin{equation}\label{eq:block-model}
 B_{ij}\sim\Ber(1/2+\kappa S_{ij}),\qquad S=ac^\top,\qquad
 v=\Var(B_{ij})=1/4-\kappa^2.
\end{equation}
When $a$ and $c$ are balanced, each row total is the sum of two
independent binomials with $m/2$ trials and success probabilities
$1/2+\kappa$ and $1/2-\kappa$. Different rows use disjoint edges.
Thus the entire row vector has the same law for every balanced pattern,
as does the column vector. Their joint first and second moments satisfy
\begin{equation}\label{eq:block-means}
 \E R_i=\E C_j=m/2,\quad
 \Cov(R)=\Cov(C)=mvI_m,\quad \Cov(R,C)=v\one\one^\top.
\end{equation}
Center the margins as $U_i=R_i-m/2$ and $V_j=C_j-m/2$.
Balance gives $U_i=\sum_k(B_{ik}-\E_SB_{ik})$, and similarly for $V_j$.
In the expansion of $U_i^2V_j$,
every term except the cube of the shared centered edge has a centered
factor independent of the rest. Hence
\begin{equation}\label{eq:block-cubic}
 \E_S(U_i^2V_j)=\E_S(B_{ij}-\E_SB_{ij})^3=-2\kappa vS_{ij}.
\end{equation}
This identity distinguishes the two orientations of a pair of opposite
row signs. The following bounds quantify the evidence in one draw from
the joint law.

\begin{proposition}[Separation of joint margins]\label{prop:block}
Let $Q_S$ be the joint law of $(R,C)$, with
$h^2(Q,Q')=\sum_y(\sqrt{q(y)}-\sqrt{q'(y)})^2$ and
$\aff(Q,Q')=1-h^2(Q,Q')/2$. For any sign-product patterns $S,S'$
differing at $\ell$ entries,
\begin{equation}\label{eq:all-candidate}
 h^2(Q_S,Q_{S'})\ge\kappa^2\ell/(100m^3).
\end{equation}
If both groups are balanced and the patterns differ by exchanging two
opposite row signs, then, uniformly over the other row and column signs,
\begin{equation}\label{eq:swap}
 \kappa^2/(50m^2)\le h^2(Q_S,Q_{S'})\le C_\tau/m^2.
\end{equation}
\end{proposition}

\noindent
The separation lower bounds give the likelihood risk upper bound.
The swap upper bound in \eqref{eq:swap}, extended below to general
compositions, gives necessity in Theorem~\ref{thm:composition-boundary}
through Assouad's argument on within-group swaps.

\paragraph{Separation for the likelihood risk bound.}
First consider the balanced swap in \eqref{eq:swap}. Give the two focal
rows $u_1,u_2$ signs $(\omega,-\omega)$, where $\omega=\pm1$, and use
the common column signs $c_j$ of the two candidates to form
\[
 G=(U_{u_1}^2-U_{u_2}^2)\sum_{j=1}^m c_j V_j .
\]
By \eqref{eq:block-cubic}, each summand in $\E_\omega G$ equals
$-4\omega\kappa v$, so the two means are $-4\kappa vm$ and $4\kappa vm$.
The standard deviation under either orientation is at most $2m^2$, as
the variance bound below shows. Thus the squared mean gap relative to
the variance gives a lower bound of order $m^{-2}$ on the squared
Hellinger distance at fixed signal. Here the $c_j$ are specified to
compare two laws. The estimator in Section~\ref{sec:poly-main} has to
learn its comparison directions from the counts.

Separation follows from a changed margin mean or the shared-edge third moment.
If a margin mean changes, its gap is at least $2\kappa$ and its variance
is $mv$. Otherwise center the margins at the means common to $S,S'$ and
write $D=S-S'$, so $D\one=D^\top\one=0$ and
$N_D=\sum_{ij}D_{ij}^2=4\ell$. The observable contrast
\begin{equation}\label{eq:contrast}
 F_D=\sum_{ij}D_{ij}(U_i^2-mv)V_j
\end{equation}
has mean gap $-2\kappa vN_D$ by the shared-edge calculation.
Its variance under either law is at most $2m^3 N_D$, by reducing the
contrast to orthogonal products of distinct centered edges
(Appendix~\ref{sec:recovery}). For the focal swap, taking
$D=S_+-S_-$ gives $F_D=2G$ and $N_D=8m$, hence
$\Var_\omega G\le 4m^4$. For any statistic $F$ with mean gap
$\Delta$, Cauchy--Schwarz about the midpoint of the two means gives
\begin{equation}\label{eq:moment-hellinger}
 h^2(Q,Q')\ge\frac{\Delta^2}
 {2\Var_Q(F)+2\Var_{Q'}(F)+\Delta^2}.
\end{equation}
Applying this to the changed margin or to $F_D$ proves
\eqref{eq:all-candidate}. A swap changes $2m$ entries, giving the lower
bound in \eqref{eq:swap}.

\paragraph{A swap bound for impossibility.}
Could another statistic separate the balanced swap more sharply?
We construct an auxiliary observation from which all $2m$ margins can be
recovered under either orientation. Bounding the distance between these
two auxiliary laws from above then bounds the distance for every
statistic of the margins by data processing.
Keep the two focal rows and call the other $m-2$ rows background.
The two candidate assignments agree on all column signs. Using their
common signs, pair opposite columns, positive first, to form $m/2$
focal $2\times2$ arrays. This pairing defines the auxiliary observation
for the comparison.
In array $l$,
write $W_l$ for its total, $\Delta_l$ for its row difference, and $T_l$
for its column difference. At even totals their joint law is the same
under both orientations. Let $I$ index the odd totals and $N=|I|$.
Conditional on $l\in I$, $t_l=2-W_l$ and $\varepsilon_l=\Delta_l$
are independent fair signs. The third sign $T_l$ satisfies
$\E_\omega(T_l\mid t_l,\varepsilon_l)=\omega\lambda t_l\varepsilon_l$,
where $\lambda=\tanh\tau$.

The three random signs $T_l,\varepsilon_l,t_l$ enter the auxiliary
observation through different sums. For a typical odd pair, the background noise in its
column difference contributes a factor of order $m^{-1}$ to the squared
contrast. Observing the row differences only through their total
contributes $N^{-1}$; observing the pair total through the node sums
below supplies another factor of order $N^{-1}$. There are $N$ such
terms. Their combined scale is $1/(mN)$, of order $m^{-2}$ when
$N$ is proportional to $m$. We now account for these factors in the
likelihood of the full observation.

In each background row, flag every column pair with exactly one edge.
Its column difference is then a sign of mean $\pm\lambda$.
An unflagged pair has difference zero and total $1+\eta$, where
$\eta$ is fair. Reveal the flags, $I$ and the even focal triples,
collectively $\mathcal H$. Form a bipartite graph on the background rows and
column pairs, with one edge of sign $\eta$ per unflagged pair.
Join an extra row $0$ to each odd pair $l$ by an edge of sign $-t_l$.
Conditional on $\mathcal H$, all signs on this graph are independent and fair.
The auxiliary observation retains its node sums $\mathcal M$, the sum
$S_0=\sum_{l\in I}\varepsilon_l$, and the actual column differences
$D_l=Z_l+T_l$, where $Z_l$ sums the flagged background differences.

Given $\mathcal H$, the node sums determine the background row totals and each column-pair
total. The differences $D_l$ then recover the individual columns.
The extra row sum determines the sum of the two focal totals, and
$S_0$ supplies their difference after adding the known even pairs.
The same reconstruction \eqref{low:cross-reconstruction} works for both orientations.

The $F_l$ flagged signs in $Z_l$ smooth the column difference.
For odd $l$, let $q_l$ be the conditional law given $\mathcal H$ of $Z_l$ plus
an independent fair sign. Given $\mathcal H,t_l,\varepsilon_l$, the density
of $D_l$ relative to $q_l$ is $1+\omega t_l\varepsilon_l\delta_l(D_l)$,
where $\E_{q_l}\delta_l=0$ and
$\nu_l=\E_{q_l}\delta_l^2\le C_\tau/(F_l+1)$.
The bound follows by expressing each flagged sign as a mixture with a
fair component (Lemma~\ref{low:scalar}). Thus the background reduces
the squared contrast in each $D_l$ to $O_\tau(m^{-1})$ on average.

Let $Q$ give $\mathcal H$ its distribution under either orientation
and, given $\mathcal H$, take $\mathcal M$, $S_0$, and the $D_l$
independently with their common marginal laws. Let $P_\omega^{\mathcal O}$ be the
auxiliary law under orientation $\omega$. Its likelihood expansion gives
\[
 \chi^2(P_\omega^{\mathcal O}\Vert Q)
 =\E\sum_{\varnothing\ne A\subseteq I}
   \binom N{|A|}^{-1} a_A^2\prod_{l\in A}\nu_l,\qquad
 a_A=\E\!\left(\prod_{l\in A}t_l\,\middle|\,\mathcal M,\mathcal H\right).
\]
Projecting a product of $k$ distinct fair $\varepsilon_l$ onto their sum
averages all $\binom Nk$ orthogonal products of that order, explaining
the reciprocal coefficient. Centered column contrasts make different
subsets orthogonal under $Q$. Given $I$, each $F_l$ is binomial on $m-2$
trials with fixed positive success probability, independently over $l$.
Using $|a_A|\le1$, all terms of orders $k\ge2$ together contribute
$O_\tau(m^{-2})$ (Appendix~\ref{sec:lower}).

For a singleton $A=\{l\}$, the coefficient $1/N$ comes from observing
only the sum of the $\varepsilon_l$. One more factor of order $1/N$
comes from observing $-t_l$ in the node sums. To see it, fix $l$ and
match the other odd pairs $k$ to distinct background rows $i_k$.
The target edge $(0,l)$ is retained.

Keep the path $0$--$k$--$i_k$--$l$ whenever both background pairs are
unflagged, and reveal the fair signs on all remaining edges outside
the target edge and retained paths. Each choice of $l$ gives a
separate refinement for bounding that coordinate's projection.
Figure~\ref{fig:swap-fair-path} shows the signs retained on one path.

\begin{figure}[H]
\centering
\begin{tikzpicture}[
 vertex/.style={circle,draw=black!80,fill=white,minimum size=6.6mm,inner sep=1pt},
 fair path/.style={line width=.85pt},
 every node/.style={font=\small}]
 \node[vertex] (origin) at (0,0) {$0$};
 \node[vertex] (pair) at (1.45,1.35) {$k$};
 \node[vertex] (row) at (4.05,1.35) {$i_k$};
 \node[vertex] (target) at (5.50,0) {$l$};
 \draw[fair path] (origin) -- node[above left] {$x$} (pair);
 \draw[fair path] (pair) -- node[below] {$y$} (row);
 \draw[fair path] (row) -- node[above right] {$z$} (target);
 \draw[line width=.85pt,dashed] (origin) --
   node[below=2pt] {target sign $\xi_l=-t_l$} (target);
 \node[above=3pt] at (pair.north) {$x+y$};
 \node[above=3pt] at (row.north) {$y+z$};
 \node[anchor=west,align=left] at (6.2,.70) {
   When both sums vanish,\\[3pt]
   $(x,y,z)=(+,-,+)$\\[1pt]
   $\phantom{(x,y,z)}\text{ or }(-,+,-)$.};
\end{tikzpicture}
\caption{One surviving path for a fixed odd pair $l$. The sums above
the internal vertices are their residual observations after the other
incident signs are revealed. Each of the two configurations shown
contributes its first sign to both endpoint sums. The signs are equally likely.}
\label{fig:swap-fair-path}
\end{figure}

When the two residual internal sums vanish, the path contributes the
same fair sign to both endpoints. Every other pair of residual sums
determines all three path signs. After subtracting the determined
contributions, each endpoint records $\xi_l+\zeta_1+\cdots+\zeta_{L_l}$,
with one independent fair sign $\zeta_j$ per ambiguous path.

By exchangeability, the posterior mean of $\xi_l$ is the endpoint sum
divided by $L_l+1$. Its mean square is $1/(L_l+1)$ conditional on $L_l$.
Refining the observation can only increase the mean square of the projection.
Given $I$, $L_l$ is binomial on $N-1$ trials with a fixed positive
success probability. The counts $L_l$ and $F_l$ share flags, so
Lemma~\ref{low:fan} uses Cauchy--Schwarz to give
\[
 \E[\nu_l a_{\{l\}}^2\mid I]
 \le C_\tau\E\!\left[\frac{1}{(F_l+1)(L_l+1)}\,\middle|\,I\right]
 \le \frac{C_\tau}{mN}.
\]
Summing with weight $1/N$ and averaging the binomial $N$ on $m/2$
trials gives $O_\tau(m^{-2})$ for all singletons together. The
Hellinger triangle inequality through $Q$ and the margin reconstruction
give the upper bound in \eqref{eq:swap}. The likelihood expansion has
accounted for every order, so this bound applies to the entire joint law.

\begin{corollary}[One swap at a given column composition]\label{cor:swap-composition}
In the independent-edge block \eqref{eq:block-model} at zero activity,
fix $\tau$ as in \eqref{eq:signal} and let $m\in4\N$. Let $Q_+$ and
$Q_-$ be the joint margin laws before and after exchanging two opposite
row signs, with all other signs fixed. If $\beta=\sum_{j=1}^m c_j$,
then, uniformly in those signs,
\[
 h^2(Q_+,Q_-)\asymp_\tau
       \min\{1,\,\beta^2/m+m^{-2}\}.
\]
\end{corollary}

\begin{proof}
The focal row difference $R_{u_1}-R_{u_2}$ has mean $2\omega\kappa\beta$
and variance $2mv$ under either orientation, since the rows use
independent edges. Inequality~\eqref{eq:moment-hellinger} gives
\[
 h^2(Q_+,Q_-)\ge
 \frac{16\kappa^2\beta^2}{8mv+16\kappa^2\beta^2}
 \ge c_\tau\min\{1,\beta^2/m\}.
\]
The $2m$ changed entries also give $h^2\ge\kappa^2/(50m^2)$ by
\eqref{eq:all-candidate}. Taking the larger of these bounds proves
the lower estimate. For $\beta^2/m\le1$, the single-block affinity
bound \eqref{alow:cross-affinity} in Appendix~\ref{sec:alignment-lower}
with $B=1$ gives, for $x=\beta^2/m+m^{-2}$,
$h^2=2(1-\aff)\le2(1-e^{-C_\tau x})\le2C_\tau x$.
Together with $h^2\le2$ this proves the upper estimate for every $\beta$.
\end{proof}

Thus the same swap has squared distance of order $m^{-2}$ towards
a balanced group and of order $m^{-1}$ towards a group whose community
sizes differ by two. These are comparisons for a single cross-group
block. A single swap changes the probabilities of the two focal nodes'
links to the other $r-2$ nodes. Its contribution to a loss averaged
over all pairs is only of order $1/r$. To obtain a positive risk floor
we first bound the information in the whole survey about one swap, then
vary a number of disjoint pairs proportional to $r$.

\paragraph{Necessity of the recovery criterion.}
A swap in group $g$ changes its $K-1$ cross-group blocks and its
within-group degree vector. Assigning background edge blocks to receiver
node pairs gives conditionally independent noises; scalar smoothing
and fair-sum projection then give within-group squared Hellinger distance
$O_\tau(m^{-1})$ (Appendix~\ref{sec:lower}). The affected group experiments
use disjoint edges. Multiplying their affinities at exact balance gives
at least $\exp\{-C_\tau[(K-1)/m^2+1/m]\}$ for large $m$.

For bounded $\Lambda(b)+\rho$, the full laws also retain positive
affinity at general composition. Towards group $k$ the two focal row
means have gap $2\kappa|b_k|$. Pair the columns in each cross-block $(g,k)$ as
before, leaving $d=|b_k|$ surplus columns.
Each opposite pair contributes the same mean to both focal rows, so
the whole gap lies in the surplus columns. Reveal their background edges
and their two-focal-edge totals. Each total of one contributes a hidden
sign of mean $\omega\operatorname{sign}(b_k)\lambda$. The oracle
observes a single sum $S$ combining these signs with the paired focal
differences in $S_0$, as in \eqref{alow:combined-sum}.
In the augmented likelihood expansion, surplus subsets project
onto the same fair sum, so their coefficients add before squaring.
Summing the orders gives a chi-square bound
$C_{B,\tau}(d^2/m+m^{-2})$ relative to the fair reference when
$d^2/m\le B$. To cover bounded values for which this bound is not
small, a fair walk with $N\asymp m$ steps and a shift of at most $d$
has uniformly positive overlap by its central binomial masses.
Appendix~\ref{sec:alignment-lower} proves both comparisons.
Together these bounds give cross-group affinity at least
$\exp\{-C_{B,\tau}(b_k^2/m+m^{-2})\}$. The within-group exponent is
$C_{B,\tau}(1+b_g^2)/m$. Thus when $\Lambda(b)\le B$ the full laws
$P_+^Y,P_-^Y$ have affinity at least
$A_r=\exp\{-C_{B,\tau}(\Lambda(b)+\rho+1/m)\}$
(Appendix~\ref{sec:alignment-lower}).

If $\Lambda(b)+\rho$ stays bounded, $r=\rho m^3\to\infty$ implies
$m\to\infty$, and eventually
$|b_k|\le m/2$. In each group fix $m/2$ anchors with sign sum $b_k$.
The remaining nodes form $m/4$ opposite-sign pairs within each group,
$r/4$ in all, with freely chosen orientations. Flipping one orientation changes $2r$
ordered probabilities to anchors by $2\kappa$, on coordinates disjoint
across bits. Their squared separation is $8r\kappa^2$ and the count laws
have overlap $1-\TV\ge 1-\sqrt{1-A_r^2}\ge A_r^2/2$. Assouad's
argument sums the conditional two-point risks over uniform bits
\citep{yu1997assouad}. After division by $r(r-1)$ the minimax risk is
at least $r\kappa^2 A_r^2/[4(r-1)]$, even with $b$ given.

\paragraph{The likelihood risk upper bound.}
For Corollary~\ref{cor:exponential-rates}, the bound
\eqref{eq:all-candidate}, valid for every pair of patterns, controls
errors far from the truth as well as
individual swaps. Write $q_s$ for the mass function of the full count
array under $(0,s)$, let $t$ be a competing balanced labeling, and
orient it so its Hamming distance $d_{\mathrm H}$ from $s$ is at most $r/2$.
Exactly $d_{\mathrm H}(r-d_{\mathrm H})$ unordered pairs change sign product. At most $d_{\mathrm H}(m-1)$
of those pairs can lie within groups. For $K\ge4$, at least $d_{\mathrm H}r/4$
changed pairs therefore lie in cross-group blocks. Multiplying their
block affinities and using $1-x\le e^{-x}$ gives
$\aff(q_s,q_t)\le\exp(-\kappa^2 d_{\mathrm H}\rho/800)$.

The mean counts supply the other part of the exponent. Put
$b_s=T^\top s$, $b_t=T^\top t$ and
$\Delta_P=\kappa(sb_s^\top-tb_t^\top)$.
Projecting its columns onto $t^\perp$ and using $s^\top t=r-2d_{\mathrm H}$ gives
\[
 \|\Delta_P\|_F^2\ \ge\ 4\kappa^2 d_{\mathrm H}(1-d_{\mathrm H}/r)\|b_s\|_2^2
       \ \ge\ 2\kappa^2 d_{\mathrm H}m\Lambda(b_s).
\]
The contrast $\langle\Delta_P,Y\rangle_F$ has mean gap $\|\Delta_P\|_F^2$
and squared edge coefficients summing to at most $2m\|\Delta_P\|_F^2$.
Hoeffding's lemma and Cauchy--Schwarz give affinity at most
$\exp\{-\|\Delta_P\|_F^2/(4m)\}$. Combining the two affinity bounds gives
$\aff(q_s,q_t)\le e^{-d_{\mathrm H}u}$ with $u=\kappa^2(\Lambda(b_s)+\rho)/1600$.
For $K<4$, the mean bound suffices: $\rho\le3/16$, so its missing
resolution term changes only the final prefactor.

Under $s$, a competitor attains the higher likelihood with probability
at most this affinity. Summing over the at most $2\binom r{d_{\mathrm H}}$
competitors in each Hamming shell gives
$\E[\widehat d_{\mathrm H}/r]\le C e^{-u}$ for its Hamming error
$\widehat d_{\mathrm H}$ up to reversal.
The sum splits at $2er e^{-u}$; beyond this radius its contributions
decrease geometrically. Thus the bound controls the fraction of errors
without requiring exact label recovery.
Appendix~\ref{sec:exponential-rates} gives the transfer to nonzero
activity and the balanced risk lower bounds.
For each comparison with a fixed reference node, the estimator in the
next section obtains independent votes from disjoint panels of groups
at zero activity.

\section{Learning labels from the reported counts}\label{sec:poly-main}

Theorem~\ref{thm:poly-recovery} gives two estimators, both using the
reported counts $Y$, the group memberships $T$ and the known signal
$\tau$. The table gives their arithmetic costs for the entire labeling
and the exponents in their label-comparison error bounds.
\begin{center}
\small
\begin{tabular}{@{}llcl@{}}
\hline
Estimator & Composition & Operations & Error exponent\\
\hline
Cubic panel majority & $b=0$ & $O(r^3)$ & $c\kappa^2\rho$\\
Two-frequency Fourier votes and held-out means & All feasible $b$ & $O(r^4)$ & $c_\tau(\Lambda+\rho)$\\
\hline
\end{tabular}
\end{center}
The cubic estimator can be used when each group is known to be
balanced. Its exponent has an absolute constant $c>0$. The second
estimator adapts to unknown $b$ at fixed $\tau$.

Fix a reference node $a$ and consider one other node $i$. We seek a
vote of $+1$ if they share a label and $-1$ otherwise. Partition the
groups into panels and fit labels within each panel using only its
internal cross-group counts. For this comparison retain the panels
that contain neither endpoint group. Each retained panel scores $i$
and $a$ against its fitted labels using the counts between that panel
and the endpoint groups. Reversing all fitted labels reverses both
scores and preserves their product. Each panel therefore votes by
the sign of this product. At exact balance the cubic estimator
returns the majority of the cubic-score votes. The adaptive
estimator computes a majority using two-frequency scores and a
second comparison using held-out mean scores. The resolution $\rho$ and
the largest singular value of the centered count matrix determine which
comparison is used for every $i$.
Both estimators keep $a$ fixed as $i$ ranges over the other nodes.
We analyze them first at zero activity.

\paragraph{Finding labels inside a panel.}
Consider two distinct groups $g,h$. Write $U=Y_{G_g,h}-(m/2)\one_m$
and $V=Y_{G_h,g}-(m/2)\one_m$ for their two vectors of counts.
The shared-edge identity \eqref{eq:block-cubic} gives
$\E U_i^2V_j=-2\kappa v s_i s_j$ at balance, with
$v=1/4-\kappa^2$.
For general group compositions the means are
$\mu_i=\kappa s_i b_h$ and $\nu_j=\kappa s_j b_g$, and expansion gives
$\E U_i^2V_j=(mv+\mu_i^2)\nu_j+2v\mu_i-2\kappa v s_i s_j$.
Since $\mu_i^2=\kappa^2b_h^2$, the first term is constant across rows and
the second is constant across columns.
With $H=I_m-\one_m\one_m^\top/m$, form the symmetric matrix with
zero diagonal blocks and upper blocks
\[
 C_{gh}=-\frac{H\{U^{\circ2}V^\top+U(V^{\circ2})^\top\}H}{4\kappa v}.
\]
The squares act coordinatewise. Its cross-group expectations are
$t_gt_h^\top$, where $t_g=Hs_{G_g}$. Thus the top eigenvector of a
panel submatrix has a population direction close to the signs when
the groups in that panel are nearly balanced.

Let each panel $B$ contain $q=2\lceil D_*m^2/(2\kappa^2)\rceil$
groups, an even number for the pairing below, where $D_*$ is a
sufficiently large numerical constant. Write $n=qm$ and
$\Lambda_B=\sum_{h\in B} b_h^2/m$.
The population matrix differs from $s_Bs_B^\top$ in operator norm
by at most $m+2\sqrt{n\Lambda_B}$, accounting for the omitted
diagonal blocks and $\|s_B-t_B\|_2^2=\Lambda_B$.
The unordered group pairs give independent matrix summands.
After truncation for the concentration argument, Corollary~2.17 of
\citet{brailovskaya2024universality} bounds the
fluctuation by $C\kappa^{-1}(1+\Lambda_B)m^2\sqrt q$ with probability
$1-Cn^{-8}$. Appendix~\ref{sec:cubic-initialization} supplies the
block moments, absorbs the logarithmic remainder terms for $q\ge m^2$,
and transfers this bound to the observed, untruncated matrix.
For sufficiently large $D_*$, spectral rounding gives the following
initializations.
For the balanced votes below, use nearest balanced rounding of the
leading eigenvector to obtain $z_B$ with
$|\langle z_B,s_B\rangle|\ge3n/4$ with probability at least $15/16$.
The adaptive estimator takes coordinate signs in every panel.
Whenever the panel's true composition satisfies $\Lambda_B\le1$, this
rounding gives $|\langle z_B,t_B\rangle|\ge n/2$ with probability at least $31/32$.
Write $z_h$ for the restriction of $z_B$ to $G_h$.

\paragraph{Voting at balance.}
Partition the groups into consecutive panels of size $q$, dropping
the remainder. For a comparison of $i$ and $a$, omit the panels
containing either endpoint. For $i\in G_g$ and $h$ in a remaining
panel, use $U_i=Y_{ih}-m/2$ and $V=Y_{G_h,g}-(m/2)\one_m$ as above.
Compute
\[
 S_{iB}=-\frac1{2\kappa v}\sum_{h\in B}
                  (U_i^2-mv)\langle z_h,V\rangle .
\]
Conditional on the panel fit, its mean is $s_i\langle z_B,s_B\rangle$
by the shared-edge identity. Its variance is at most
$Cq m^4/\kappa^2$. Indeed, the fourth norms of the two factors of each
summand are both $O(m)$, and distinct summands use independent edge
blocks. On $|\langle z_B,s_B\rangle|\ge3n/4$, Chebyshev's inequality bounds the chance
of an incorrect or zero score by $Cm^2/(\kappa^2 q)\le1/16$.
The panel votes by the sign of $S_{iB}S_{aB}$, taking $+1$ at zero.
Both scores use the same fit, and endpoints in the same group also
share counts. A union bound on a poor fit and the two endpoint sign errors
still gives a correct vote with probability at least $1-3/16$.
For fixed $(i,a)$, each panel uses only its own internal edges and
the edges joining it to the endpoint groups. These sets are disjoint
across retained panels, so their votes are independent. Hoeffding's
inequality for the majority of the $\lfloor K/q\rfloor-O(1)$ ballots gives comparison
error $C\exp(-c\kappa^2\rho)$ with numerical constants. A fixed tie
rule covers empty majorities, since $K<6q$ implies bounded
$\kappa^2\rho$.

\paragraph{Reading a score when the endpoint group is unbalanced.}
A direct extension of the balanced score would use
$(U_i^2-mv)\langle w,V\rangle$ with $w=Hz_h$ to remove the
column-constant term $2v\mu_i$. The endpoint group $g$ lies outside
every retained panel, so $\Lambda_B\le1$ does not control $b_g$.
The mean
$\E[\langle w,V\rangle\mid z_B]=\kappa b_g\langle z_h,t_h\rangle$
can inflate the variance of this cubic score. For example, with
$b_h=0$, $z_h=s_h$ and $|b_g|=m$, the summand has mean of magnitude
$\Theta_\kappa(m)$ and variance $\Theta_\kappa(m^6)$.
At fixed $\kappa$, the sum over $q$ such independent blocks has
variance-to-squared-mean ratio $\Theta_\kappa(m^4/q)$.
With $q\asymp_\kappa m^2$, this grows like $m^2$, so Chebyshev's bound
does not ensure reliable votes.
We construct a Fourier score using bounded exponentials, whose
expectations factor over independent edges.

For $x>0$, let $\theta=2\arctan x$. The edge $i\ell$, centered at
$1/2$, has characteristic function
$d_x\exp(\mathrm{i}s_i s_\ell\phi)$ at $\theta$, where
$\phi=\arctan(2\kappa x)$ and
$d_x^2=(1+4\kappa^2x^2)/(1+x^2)$ with $d_x>0$.
Thus $\E e^{\mathrm{i}\theta V_\ell}=d_x^m e^{\mathrm{i}s_\ell b_g\phi}$.
Isolating the shared edge in $\E[U_i e^{\mathrm{i}\theta V_\ell}]$
gives this column characteristic function times
$\mu_i+\mathrm{i}a_x+c_xs_i s_\ell$, where
\[
 a_x=\frac{2vx}{1+4\kappa^2x^2},\qquad c_x=2\kappa x a_x .
\]
Remove the imaginary shared-edge term and record the phase through
\[
 F_{ih}(x)=m\operatorname{Re}\!\left[(U_i-\mathrm{i}a_x)
                    \sum_{\ell\in G_h} w_\ell e^{\mathrm{i}\theta V_\ell}\right],
 \qquad J_{gh}(x)=m^{-1}\sum_{\ell\in G_h}\cos(\theta V_\ell).
\]
The phase statistic has mean $\E J_{gh}=d_x^m\cos(b_g\phi)$.
Since $\sum_\ell w_\ell=0$, the constant real part of the $\mu_i$
term vanishes. Hence the conditional mean of
$F_{ih}$ is $m c_x d_x^m\cos(b_g\phi)s_i\langle z_h,t_h\rangle$.

Pair the panel groups. For a pair $h,k$, multiply $F_{ih}$ by
$J_{gk}$ from the independent block, then add the product with $h,k$
exchanged. This squares the potentially signed cosine in the mean.
One frequency could still give zero. Use
$x_1=m^{-1/2}$ and $x_2=2x_1/(1-4\kappa^2x_1^2)$, so
$\phi_2=2\phi_1$. For both, $d_x^{2m}$ is bounded below and
$m c_x\asymp\kappa$. Since $\cos^2 u+\cos^2(2u)\ge7/16$, the sum
over both frequencies and all pairs has conditional mean
$\gamma_g s_i\langle z_B,t_B\rangle$ with $\gamma_g\asymp\kappa$
uniformly in $b_g$. Its conditional variance is $O(m^4(q+\Lambda_B))$.
Appendix~\ref{sec:fourier-panels} gives the variance calculation,
including the sine term in $F/m$, bounded by $|a_x|\,\|w\|_1\le C\sqrt m$,
and the dependence of the two margins.

For $\Lambda_B\le1$ and $|\langle z_B,t_B\rangle|\ge n/2$, the squared mean dominates
this variance by an arbitrarily large constant through the choice of
$D_*$. A union bound on a poor fit and the two endpoint sign errors
again gives a correct ballot with probability at least $1-3/32$.
The estimator votes with every
retained panel. To analyze the majority, observe that at most $\Lambda$
panels have $\Lambda_B>1$, since the sum of their energies is bounded
by $\Lambda$. Set $\delta_\tau=1/[32(1+D_*/\kappa^2)]$.
When $\Lambda\le\delta_\tau\rho$, these panels occupy at most a sixteenth
of the retained panels for sufficiently large $\rho$. Even allowing all
their votes to be wrong, the expected correct fraction exceeds $3/4$.
Independence across panels again gives exponentially small error in
$\rho$ at fixed signal.

\paragraph{Using the means and selecting the comparison.}
For larger alignment the diagonal-corrected counts
$W=Y-(m/2)\one_r\one_K^\top+(1/2+\kappa)T$ have mean
$\kappa s b^\top$. Hold out the endpoint groups $U_0=\{g(i),g(a)\}$.
Fit a leading right singular vector using rows and columns in the
remaining groups, extending it by zero on $U_0$. Include this direction
when the complement is nonempty and also include $e_u$ for each
$u\in U_0$. The coordinate directions retain signal that can be
concentrated entirely in the held-out groups. From these at most three
directions choose $\xi$ maximizing
$\min_{k\in\{i,a\}}|\langle W_{k,*},\xi\rangle|$, and multiply the two
score signs.

If the endpoint groups contain half the squared norm of $b$ or more,
one coordinate direction has $|\langle b,\xi\rangle|\ge\|b\|_2/2$.
Otherwise the remaining $n'$ nodes satisfy $n'\ge\Lambda/2$. An
operator-norm noise bound of $4\sqrt{n'm}$ on their count matrix
gives $|\langle b,\xi\rangle|\ge\|b\|_2/4$ for the fitted direction
when $\kappa\sqrt\Lambda\ge32$, except with probability $2e^{-c n'}$.
Given the fit, the two endpoint projections onto that direction use
held-out edges. For each coordinate direction the same concentration
follows directly from independent edges and the diagonal correction.
Hoeffding's inequality bounds all six score deviations from
$\kappa s_k\langle b,\xi\rangle$ by $\kappa\|b\|_2/16$ except with
probability $C e^{-c\kappa^2\Lambda}$. A candidate with large
population projection therefore gives two large observed scores. The
maximizing candidate has both scores large enough to have the signs
of their own means, whose product is $s_i s_a$.
Thus the mean comparison error is at most $C e^{-c\kappa^2\Lambda}$.
Enlarging $C$ covers $\kappa\sqrt\Lambda<32$.

Finally, write $D^{\mathrm{mean}}_{ia}$ for the product of mean-score
signs and $D^{\mathrm{panel}}_{ia}$ for the two-frequency panel majority.
Put $\widehat\Lambda=\sigma_1(W)^2/(\kappa^2rm)$ and
$\rho_0=256/(\kappa^2\delta_\tau)$.
The output for this pair is
\[
 D_{ia}=
 \begin{cases}
 D^{\mathrm{mean}}_{ia}, & \rho<\rho_0
                  \ \text{or}\ \widehat\Lambda\ge\delta_\tau\rho/4,\\
 D^{\mathrm{panel}}_{ia}, & \text{otherwise}.
 \end{cases}
\]
The selector reuses the counts from both comparisons. Under zero
activity the event $\kappa|\sqrt{\widehat\Lambda}-\sqrt\Lambda|\le4$
has probability at least $1-2e^{-cr}$. On this event, for
$\rho\ge\rho_0$, choosing the mean implies $\Lambda\ge\delta_\tau\rho/16$,
and choosing the panel majority implies $\Lambda\le9\delta_\tau\rho/16$.
In each of these parameter ranges, the corresponding \emph{unconditional}
error probability is at most $C_\tau e^{-c_\tau(\Lambda+\rho)}$, since
$\rho\le16\Lambda/\delta_\tau$ in the mean range and $\Lambda\le\rho$
in the panel range. For a fixed composition, we sum these unconditional
bounds over the branches possible on the event and add $2e^{-cr}$ for
its complement. The result has the same rate because $\Lambda+\rho\le2r$.
For $\rho<\rho_0$, bounded $\rho$ changes only the prefactor.

The cutoff in this construction is
$\rho_0=8192\,\kappa^{-2}(1+D_*\kappa^{-2})$. Even at the largest
signal in \eqref{eq:signal}, the coefficient of $D_*$ exceeds
$1.36\times10^{11}$. The adaptive estimator therefore uses the mean
comparison throughout a large finite range. This cutoff is a sufficient
choice for the proof. It establishes polynomial-time attainment of
the asymptotic rate at fixed signal. How to tune the selector for
finite survey sizes remains open.

Set the reference sign to $+1$ and round the comparison vector to a
nearest globally balanced labeling. Rounding increases its Hamming
distance to $s_as$ by at most a factor of two. For activity in the
stated energy ball the full-edge likelihood ratio has second moment
at most $e^{M^2}$ under zero activity, by the logistic log-partition
bound $f''\le1/4$. Cauchy--Schwarz transfers each comparison error
by taking its square root and multiplying by $e^{M^2/2}$. Summing
over endpoints bounds the expected fraction of errors even though
comparisons to the common reference may be dependent. The returned
matrix uses $1/2+\kappa\widehat s_i\widehat s_j$ off the diagonal,
incurring the additional activity term $M^2/[r(r-1)]$ in squared loss.
The cubic estimator reuses the same panel fits for all endpoints,
with $O(r^3)$ total arithmetic operations. The adaptive estimator
fits the mean direction on at most $K$ distinct complements of
endpoint groups. Processing these successively uses $O(r^4)$
operations and $O(r^2)$ space for every feasible composition.
Appendix~\ref{sec:polynomial-rates} gives the complete proof of
Theorem~\ref{thm:poly-recovery}, including the numerical precision
required for the scores and selector.

\section{What can the questionnaire resolve?}\label{sec:design}

With baseline link probabilities $1/2\pm\kappa$, uniform recovery
over all globally balanced assignments holds exactly when $K/m^2\to\infty$
(Theorem~\ref{thm:boundary}). Every partition admits $b=0$, which
removes the mean channel. In terms of roster size, the criterion is
$K/r^{2/3}\to\infty$. Even $K$ of order $r^{2/3}$ leaves a positive
worst-case risk floor. Observing all edges on the same roster allows
uniformly consistent recovery in every regime (Appendix~\ref{sec:adjacency}).

Take even $K\asymp m^{3/2}$ and let $m$ grow. Compare two globally
balanced populations with the same roster and questionnaire. One has
$b_k=0$ throughout. The other has $b_k=2$ in half the groups and
$b_k=-2$ in the rest. Their resolution tends to zero, but
the recovery criterion gives
\[
\begin{array}{lccc}
 \text{Composition} & \Lambda(b) & \rho & \text{Uniform recovery on that class}\\
 \hline
 b_k=0\ \text{for all }k & 0 & K/m^2 & \text{impossible}\\
 |b_k|=2\ \text{for all }k & 4K/m & K/m^2 & \text{possible}.
\end{array}
\]
Each group's community share changes by only $1/m$. In the second
population the squared projection of the unit community contrast onto
the trait subspace is $\Lambda(b)/r=4/m^2$, which tends to zero.
The $K$ small differences in mean counts still give a diverging total
alignment $4K/m$. At zero activity,
Corollary~\ref{cor:swap-composition} gives the squared Hellinger distance
for an exchange of two opposite labels within one group. Its order in
each cross-group block is $m^{-2}$ towards a balanced group and $m^{-1}$
towards a group whose community sizes differ by two. Summed over the
$K-1$ other groups, the first scale vanishes and the second diverges.
Along $K=m\to\infty$ with the same $b_k=\pm2$ pattern,
$\Lambda(b)=4$ and $\rho=1/m$. The worst-case risk on that composition
class stays bounded away from zero although every group is imbalanced.

When labels are uniform subject to global balance and independent of
the questionnaire, the zero-activity risk averaged over assignments
and graphs tends to zero if and only if $K\to\infty$
(Theorem~\ref{thm:average-recovery}, Appendix~\ref{sec:average-recovery}).
Indeed, $\E\Lambda=r(K-1)/(r-1)$, and $\Lambda/(K-1)$ concentrates at
one as $K$ grows. Typical imbalance supplies mean information even
when resolution is inadequate. Along $K=m\to\infty$, for instance,
$\rho=1/m$. The averaged matrix risk vanishes along this sequence, even
though the exactly balanced assignments retain a worst-case risk floor.

Return to the eastern and western districts, each with $m$ residents. For a
fixed list with indicator $z$, the target $(Pz)_i/r$ is resident $i$'s
expected number of contacts on the list divided by the size of the
whole roster, $r=2m$. Contacts with a designated group also enter
network exposure measures, whose conditional averages support causal
inference with partial network data \citep{wilkinsreeves2026single}.
At zero activity, writing $\bar z=r^{-1}\sum_j z_j$, we have
by \eqref{tar:exact-target}
\[
 \frac{(P_{0,s}z)_i}{r}
   =\underbrace{\frac{\bar z}{2}
     -\left(\frac12+\kappa\right)\frac{z_i}{r}}_{\beta_i(z)}
     +\kappa\frac{s^\top z}{r}s_i .
\]
The list and the known signal determine the baseline $\beta_i(z)$,
including the correction for the absence of self-contacts.

Take the whole eastern district as the first list. At any composition
its questionnaire column divided by $r$ is unbiased for $Pz/r$ and has
mean squared error, averaged over residents, at most $1/(4r)$.
Summing the corresponding columns gives the same guarantee for any
union of queried groups after division by $r$. If each district
is exactly balanced ($b=0$) and activity is zero, then
$s^\top z=0$ and the baseline gives every coordinate exactly.

For the second list, fix half the residents of each district. Both
lists contain $m$ names. Even at zero activity and the same $b=0$, the
second list has a positive worst-case risk floor as $m$ grows. To see
the distinction, put $d(T,z)=r^{-1}\inf_{a\in\R^K}\|z-Ta\|_1$. For a
binary list, $d$ is the smallest fraction of residents whose list
membership must change to make the list a union of queried groups.
The eastern list has $d=0$. For the second list, either including or
excluding an entire district changes half its memberships, so $d=1/2$.
For any number of groups and fixed activity radius $M\ge0$, the
criterion for uniform consistency at a prescribed composition is
$d/(1+\Lambda(b)+\rho)\to0$ (Theorem~\ref{thm:exposure-recovery}).
Equivalently, along every subsequence on which $\Lambda(b)+\rho$ stays
bounded, the fraction of memberships needed to turn the list into a
union of queried groups must tend to zero. Here $\Lambda=0$ and $\rho=2/m^2$.

For a general binary list, summing the counts over the nearest union of
groups gives a direct estimate. Summing estimated pairwise probabilities
over the list itself gives a second estimate. Both sums are divided by $r$.
The polynomial-time procedure in Appendix~\ref{sec:exposure} uses $d$,
the survey resolution, the known signal and the largest singular value
of the centered counts to choose between them without knowing $b$.

Keep the second list, zero activity and exact balance in both districts,
and now \emph{average} over uniform label assignments within this same
class. The assignments in the two districts are independent. In each
district the label sum over its $m/2$ listed residents has mean zero and variance
$m^2/[4(m-1)]$ by sampling without replacement. Adding the two sums
and using $r=2m$ gives the baseline's mean squared error
\[
 \kappa^2\,\E_s\!\left[\left(\frac{s^\top z}{r}\right)^2\right]
   =\frac{\kappa^2}{8(m-1)}.
\]
Under the same distribution of assignments, the smallest average
squared error for the pairwise probabilities tends to $\kappa^2$,
even using all the counts. Within-district permutations equate this
average risk with the minimax risk in Corollary~\ref{cor:no-count-risk}
(Lemma~\ref{tar:conditional-symmetry}). Thus averaging over the
assignments makes the list baseline consistent while uncertainty about
the individual pairs persists. Under uniform labels subject only to
balance across the whole roster, the list baseline has risk
$\kappa^2/[4(r-1)]$ (Proposition~\ref{tar:random-exposure}).

With partial participation, exchanging the latent profiles of two
nonrespondents in the same group permutes the summands in each affected
count and preserves every respondent--respondent edge probability.
The joint law of all reports is unchanged. After accounting for this
symmetry, how does the respondent set determine which labeled edge
probabilities can be recovered from one count array?

\clearpage
\appendix
Throughout the appendices, a cell is a trait group. Write $\risk_r^Y(T)$
for the minimax risk over $\Theta_r(M)$ and use superscript $A$ for
the full-adjacency experiment.

\section{Basic bounds for recovery}\label{sec:recovery-tools}

We collect the activity change of measure, the conversion from sign error
to probability-matrix loss, and the operator bounds used by the estimators.
These bounds hold uniformly over the parameter class \eqref{eq:class}.

\subsection{Activity transfer and probability loss}

For sign vectors $s,t$, write
\[
 d_\pm(s,t)=\min\{\#\{i:s_i\ne t_i\},\#\{i:s_i\ne-t_i\}\}
\]
for their Hamming distance up to global reversal.

Let
$\beta_{ij}=\alpha_i+\alpha_j$ and let
$\Lambda_{\alpha,s}$ be the full-adjacency likelihood ratio of
$\Prob_{\alpha,s}$ against $\Prob_{0,s}$. For
$f(x)=\log(1+e^x)$, independence of the edges gives
\begin{align}
 \log\E_{0,s}\Lambda_{\alpha,s}^2
 &=\sum_{i<j}\{f(\tau s_is_j+2\beta_{ij})
       -2f(\tau s_is_j+\beta_{ij})+f(\tau s_is_j)\}\notag\\
 &\le\frac14\sum_{i<j}\beta_{ij}^2\le M^2.
 \label{up:second-moment}
\end{align}
The second-difference bound holds for either sign of $\beta_{ij}$:
it follows by integrating $f''\le1/4$ twice over an interval of
length $|\beta_{ij}|$. Every event $E$ measurable from the graph,
and hence every ARD-measurable event, therefore satisfies
\begin{equation}\label{up:change-measure}
 \Prob_{\alpha,s}(E)
 =\E_{0,s}(\Lambda_{\alpha,s}\one_E)
 \le e^{M^2/2}\sqrt{\Prob_{0,s}(E)}.
\end{equation}

For any sign estimate $\widehat s$, define its homogeneous off-diagonal probabilities by
\begin{equation}\label{up:probability-estimator}
 \widehat p_{ij}=1/2+\kappa\widehat s_i\widehat s_j.
\end{equation}
Since $\sigma$ is $1/4$-Lipschitz,
\begin{equation}\label{up:target-transfer}
 L_r(P_{\alpha,s},P_{0,s})
 \le\frac{2}{16r(r-1)}\sum_{i<j}\beta_{ij}^2
 \le\frac{M^2}{2r(r-1)}.
\end{equation}
For two zero-activity sign vectors at quotient Hamming distance $h$,
the exact loss is
\begin{equation}\label{up:sign-loss}
 L_r(P_{0,t},P_{0,s})
 =\frac{8\kappa^2h(r-h)}{r(r-1)}.
\end{equation}
For any sign estimate $\widehat s$ and its homogeneous probability
estimate $P_{0,\widehat s}$, equations \eqref{up:target-transfer}
and \eqref{up:sign-loss} give the pointwise bound
\begin{equation}\label{aup:sign-to-matrix}
 L_r(P_{0,\widehat s},P_{\alpha,s})
 \le\frac{M^2}{r(r-1)}+
       \frac{16\kappa^2r}{r-1}\frac{d_\pm(\widehat s,s)}r.
\end{equation}
This decomposition will transfer sign-error bounds to probability-matrix risk.

\subsection{The centered count matrix}

The hollow probability matrix at zero activity is
\begin{equation}\label{aup:homogeneous-matrix}
 P_{0,s}=\tfrac12\one_r\one_r^\top+\kappa ss^\top
                         -(\tfrac12+\kappa)I_r.
\end{equation}
Since $T^\top T=mI_K$, the observable matrix
\begin{equation}\label{aup:centered-counts}
 W=Y-\tfrac m2\one_r\one_K^\top+(\tfrac12+\kappa)T
   =\kappa sb^\top+Z,
 \qquad
 Z=(P_{\alpha,s}-P_{0,s})T+(A-P_{\alpha,s})T
\end{equation}
has a rank-one signal of singular value
$a=\kappa\sqrt{rm\Lambda}$. The $1/4$-Lipschitz property of the
logistic function and the activity constraint imply
\begin{equation}\label{aup:activity-perturbation}
 \|P_{\alpha,s}-P_{0,s}\|_F^2
 \le\frac18\sum_{i<j}(\alpha_i+\alpha_j)^2
 \le\frac{M^2}{2}.
\end{equation}

We also need a uniform bound for the random part of $Z$. Write
$E=A-P_{\alpha,s}$. For any unit vector $x\in\R^r$,
\[
 x^\top Ex=2\sum_{i<j}x_ix_jE_{ij}.
\]
The summands are independent and centered, and the sum of their squared
range lengths is
\[
 4\sum_{i<j}x_i^2x_j^2
   =2\left(1-\sum_i x_i^4\right)\le2.
\]
Hoeffding's inequality therefore gives
$\Prob_{\alpha,s}(|x^\top Ex|>t)\le2e^{-t^2}$. A $1/4$-net
$\mathcal N$ of the unit sphere can be chosen with
$|\mathcal N|\le9^r$. For a symmetric matrix,
$\|E\|_{\rm op}\le2\max_{x\in\mathcal N}|x^\top Ex|$:
approximating a maximizing unit vector by a point of the net changes
the quadratic form by at most $\|E\|_{\rm op}/2$. Hence
\begin{equation}\label{aup:noise-event}
 \Prob_{\alpha,s}(\mathcal G^c)\le2e^{-c_0r},
 \qquad
 \mathcal G=\{\|E\|_{\rm op}\le4\sqrt r\},
 \qquad c_0=4-\log9>0.
\end{equation}
This probability bound is uniform over the parameter class.

Put $D_0=4+M$ and $N_0=D_0\sqrt{rm}$. On $\mathcal G$,
\eqref{aup:activity-perturbation} and $\|T\|_{\rm op}=\sqrt m$ give
\[
 \|Z\|_{\rm op}
 \le4\sqrt{rm}+M\sqrt{m/2}\le N_0.
\]
The variational characterization of the largest singular value then
implies
\begin{equation}\label{aup:singular-value-control}
 \|Z\|_{\rm op}\le N_0,
 \qquad
 \left|\frac{\sigma_1(W)}{\sqrt{rm}}-\kappa\sqrt\Lambda\right|
 \le D_0
 \quad\hbox{on }\mathcal G.
\end{equation}

\subsection{Spectral recovery and balanced rounding}\label{sec:spectral-rounding}

Choose a leading left singular vector $\widehat u$ of $W$, using a
deterministic rule when needed. Give the largest $r/2$ coordinates of
$\widehat u$ sign $+1$ and the remaining coordinates sign $-1$, with
fixed tie-breaking, and denote the resulting vector by $\widehat s^S$.
This chooses a nearest point to $\widehat u$ among balanced sign vectors
divided by $\sqrt r$. Define the hollow estimator by
$\widehat P^S_{ij}=1/2+\kappa\widehat s_i^S\widehat s_j^S$ for
$i\ne j$.

To bound its error, first suppose $\Lambda>0$, and put
$u=s/\sqrt r$ and $e=\|Z\|_{\rm op}$. A leading right singular
vector $\widehat v$ satisfies
$W\widehat v=\sigma_1(W)\widehat u$. Projection onto the orthogonal
complement of $u$ gives
\[
 \|(I_r-uu^\top)\widehat u\|_2
 \le\frac{e}{\sigma_1(W)}\le\frac{e}{a-e}
 \quad\hbox{if }a>e.
\]
Orient $u$ so that its inner product with $\widehat u$ is nonnegative.
If $a\ge2e$, then
\[
 \|\widehat u-u\|_2^2
 \le2\|(I_r-uu^\top)\widehat u\|_2^2
 \le\frac{8e^2}{a^2}.
\]
Nearest-point rounding and the triangle inequality yield
\[
 \left\|\frac{\widehat s^S}{\sqrt r}-u\right\|_2
 \le2\|\widehat u-u\|_2.
\]
Because squared distance between normalized sign vectors is four times
their fractional Hamming distance, this proves
$d_\pm(\widehat s^S,s)/r\le8e^2/a^2$. When $a<2e$, the same
upper bound follows from $d_\pm/r\le1/2$. With the convention that
$1/\Lambda=\infty$ at $\Lambda=0$, we have on $\mathcal G$
\begin{equation}\label{aup:spectral-sign-error}
 \frac{d_\pm(\widehat s^S,s)}r
 \le\min\left\{\frac12,\frac{8D_0^2}{\kappa^2\Lambda}\right\}
 \le\frac{C_{\tau,M}}{1+\Lambda}.
\end{equation}

\section{Contraction and the recovery lower bound}\label{sec:lower}

We prove the upper bound in \eqref{eq:swap} by comparing the margin
experiment with an observation that retains randomness in both pair sums
and pair differences. We then combine this comparison with a within-group bound and
a collection of disjoint probability-matrix coordinates to obtain a
finite-sample lower risk bound at exact group balance.

Throughout this section the activity vector is zero. Write
\begin{equation}\label{low:parameters}
 \begin{gathered}
 w=2\kappa=\tanh(\tau/2),\qquad
 \lambda=\frac{2w}{1+w^2}=\tanh\tau,\qquad
 q_0=\frac{1-w^2}{4},\qquad \pi=\frac{1+w^2}{2},\\
 s_0=4q_0\pi=\frac{1-w^4}{2},\qquad
 p_f=\frac{(1-\pi)^2}{4}.
 \end{gathered}
\end{equation}
In particular, $0<\lambda\le1/16$. Define
\begin{equation}\label{low:constants}
 \begin{split}
 z_c&=\frac{\lambda^2}{(1-\lambda)\pi(m-1)},\qquad
 z_w=\frac{\lambda^2}{(1-\lambda)\pi(m/2-1)},\\
 D_c(m,w)&=4\left\{
 \frac{8\lambda^2}
 {(1-\lambda)\pi p_f s_0(m-1)(m+2)}
 +\frac{z_c^2}{1-z_c}\right\}.
 \end{split}
\end{equation}
All denominators are positive for $m\in4\N$. The comparisons below give
$h_c^2\le D_c(m,w)=O_\tau(m^{-2})$ for cross-group margins and
$h_w^2\le4z_w/(1-z_w)=O_\tau(m^{-1})$ for a within-group degree vector.

\subsection{Scalar smoothing and a conditional projection identity}

For probability mass functions $p,q$, put
\[
 J(P,Q)=\frac12\sum_x\frac{\{p(x)-q(x)\}^2}{p(x)+q(x)},
 \qquad
 \chi^2(P\Vert Q)=\sum_{x:q(x)>0}\frac{\{p(x)-q(x)\}^2}{q(x)}.
\]
Terms with zero numerator and denominator are zero, and the second
quantity is infinite unless $P$ is absolutely continuous with respect
to $Q$.

\begin{lemma}[Smoothing a single sign]\label{low:scalar}
Let $Z$ be the sum of $F\ge0$ independent signs, each with mean
$\lambda$ or $-\lambda$, and let $f$ be its mass function. Set
\begin{equation}\label{low:channel-definition}
 \begin{gathered}
 q(d)=\frac{f(d-1)+f(d+1)}2,\qquad
 \delta(d)=\lambda\frac{f(d-1)-f(d+1)}{f(d-1)+f(d+1)},\\
 \nu=\sum_dq(d)\delta(d)^2,
 \end{gathered}
\end{equation}
with $\delta=0$ when the denominator is zero. Then
\begin{equation}\label{low:scalar-bound}
 \E_q\delta=0,\qquad
 \nu\le\lambda^2,\qquad
 \nu\le\frac{\lambda^2}{(1-\lambda)(F+1)}.
\end{equation}
For a random revealed collection of summands, interpret $\nu$
conditional on that collection. If its size is a binomial random
variable $F\sim\operatorname{Bin}(b,\pi)$, then
\begin{equation}\label{low:averaged-smoothing}
 \E\nu\le\frac{\lambda^2}{(1-\lambda)\pi(b+1)}.
\end{equation}
The individual means of the signs may be fixed arbitrarily among
$\{\lambda,-\lambda\}$.
\end{lemma}

\begin{proof}
The zero-mean assertion follows by summing the difference in
\eqref{low:channel-definition}; $|\delta|\le\lambda$ gives the second
assertion. With $B=(Z+F)/2$, a change of lattice gives
\[
 \frac{\nu}{\lambda^2}
 =J\{\mathcal L(B),\mathcal L(B+1)\}.
\]
If $B\sim\operatorname{Bin}(L,1/2)$, the equal mixture of $B$ and
$B+1$ is $\operatorname{Bin}(L+1,1/2)$. Relative to this mixture, the
contrast $(p-q)/(p+q)$ is $(L+1-2x)/(L+1)$. Its mean square is
therefore
\begin{equation}\label{low:fair-translation}
 J\{\operatorname{Bin}(L,1/2),1+\operatorname{Bin}(L,1/2)\}
 =\frac1{L+1}.
\end{equation}
Each Bernoulli summand of $B$, whose success probability is
$(1\pm\lambda)/2$, can be generated by using a fair Bernoulli with
probability $1-\lambda$ and the corresponding deterministic value
with probability $\lambda$. Conditional on these independent selectors,
$B$ is a translate of $\operatorname{Bin}(L,1/2)$, where
$L\sim\operatorname{Bin}(F,1-\lambda)$.

The function $(x-y)^2/(x+y)$ is jointly convex on the nonnegative
quadrant: for nonnegative weights, the required inequality follows
from Cauchy--Schwarz with denominator equal to the weighted sum of
$x+y$. Applying this inequality to the common selector mixture,
using translation invariance and \eqref{low:fair-translation}, gives
\[
 \nu\le\lambda^2\E\frac1{L+1}
 =\lambda^2\frac{1-\lambda^{F+1}}{(F+1)(1-\lambda)}.
\]
This includes $F=0$ and proves \eqref{low:scalar-bound}. For
$B_0\sim\operatorname{Bin}(b,p)$, integration of its probability
generating function gives
\begin{equation}\label{low:reciprocal-first}
 \E\frac1{B_0+1}
 =\int_0^1(1-p+px)^b\,dx
 =\frac{1-(1-p)^{b+1}}{p(b+1)}.
\end{equation}
Averaging \eqref{low:scalar-bound} and taking $p=\pi$ proves
\eqref{low:averaged-smoothing}.
\end{proof}

We shall also use the reciprocal second-moment bound
\begin{equation}\label{low:reciprocal-second}
 \E(B_0+1)^{-2}\le\frac{2}{p^2(b+1)(b+2)},
 \qquad B_0\sim\operatorname{Bin}(b,p),\quad p>0.
\end{equation}
Indeed, $(k+1)^{-2}\le2/\{(k+1)(k+2)\}$ and
\[
 \E\frac1{(B_0+1)(B_0+2)}
 =\int_0^1(1-x)(1-p+px)^b\,dx
 \le\frac1{p^2}\int_0^1(1-y)y^b\,dy.
\]
The final integral is $1/\{(b+1)(b+2)\}$.

The next identity records every order of a hidden-sign likelihood.
It will be used with either random or known signs $t_l$.

\begin{lemma}[Projection through a fair sum]\label{low:projection}
Let $\varepsilon_1,\ldots,\varepsilon_N$ be independent fair signs,
independent of $(t,\mathcal M)$, where $t\in\{-1,1\}^N$ and
the joint law of $(t,\mathcal M)$ is common to
$\omega\in\{-1,1\}$. Put $S=\sum_l\varepsilon_l$. Conditional on
$(t,\mathcal M,\varepsilon)$, suppose the $D_l$ are independent and
\begin{equation}\label{low:abstract-channel}
 \Prob_\omega(D_l=d\mid t,\mathcal M,\varepsilon)
 =q_l(d)\{1+\omega t_l\varepsilon_l\delta_l(d)\},
 \qquad \E_{q_l}\delta_l=0.
\end{equation}
Let $Q$ be the product of the marginal law of $\mathcal M$, the
fair-sum law of $S$, and the $q_l$ laws. For $B\subseteq[N]$, set
\[
 a_B(\mathcal M)=\E\left(\prod_{l\in B}t_l\mid\mathcal M\right),
 \qquad \nu_l=\E_{q_l}\delta_l^2.
\]
Then the law $P_\omega$ of $(\mathcal M,S,D)$ satisfies
\begin{equation}\label{low:likelihood-expansion}
 \frac{dP_\omega}{dQ}
 =\sum_{B\subseteq[N]}\omega^{|B|}a_B(\mathcal M)
       g_{|B|}(S)\prod_{l\in B}\delta_l(D_l),
\end{equation}
where $g_0=1$ and, for any fixed $B$ of size $k$,
$g_k(S)=\E(\prod_{l\in B}\varepsilon_l\mid S)$. Moreover,
\begin{equation}\label{low:projection-chi}
 \chi^2(P_\omega\Vert Q)
 =\sum_{\varnothing\ne B\subseteq[N]}
    \frac{\E a_B(\mathcal M)^2}{\binom N{|B|}}
       \prod_{l\in B}\nu_l.
\end{equation}
For $N=0$, the likelihood is one and the sum is zero.
\end{lemma}

\begin{proof}
For $|B|=k$, exchangeability conditional on $S$ gives
\[
 g_k(S)=\binom Nk^{-1}\sum_{|A|=k}\prod_{l\in A}\varepsilon_l.
\]
The sum on the right is a symmetric function of the signs, hence a
function of their sum. Products indexed by different subsets are
orthogonal under independent fair signs: their product contains a
sign to the first power. Consequently
\begin{equation}\label{low:walsh-norm}
 \E g_k(S)^2=\binom Nk^{-1}.
\end{equation}
Expand the product in \eqref{low:abstract-channel} and condition on
$(\mathcal M,S)$. Independence of $\varepsilon$ and
$(t,\mathcal M)$ factors each conditional expectation into $a_Bg_k$,
giving \eqref{low:likelihood-expansion}. Under $Q$, two distinct
subsets in that expansion are orthogonal because some independent
$D_l$ occurs in just one term and $\E_{q_l}\delta_l=0$. The empty
subset is the constant one. Squaring the remaining terms and applying
\eqref{low:walsh-norm} proves \eqref{low:projection-chi}.
\end{proof}

All these identities may be applied conditional on a common ancillary
$\mathcal H$. Giving the reference law the same $\mathcal H$ marginal then averages
the conditional chi-square divergences. Additional coordinates that,
conditional on $\mathcal H$, have a common law independent of the displayed
variables can be appended without changing the calculation.
We shall use
\begin{equation}\label{low:hellinger-reference}
 h^2(P_+,P_-)
 \le2h^2(P_+,Q)+2h^2(P_-,Q)
 \le2\chi^2(P_+\Vert Q)+2\chi^2(P_-\Vert Q).
\end{equation}
The first inequality is the squared triangle inequality in the
space of square-root densities; the second follows pointwise from
$(\sqrt{x}-1)^2\le(x-1)^2$.

\subsection{Opposite-sign pairs and the cross-group oracle}

Consider the cross-group block of Proposition~\ref{prop:block} with
balanced row and column signs. Fix two rows $u,v$ whose signs are
$(a_u,a_v)=(\omega,-\omega)$; every other sign is fixed. Pair the
columns into $n=m/2$ labeled opposite-sign pairs
$(j_l^+,j_l^-)$ and write
\begin{equation}\label{low:pair-variables}
 U_{il}=B_{i,j_l^+}+B_{i,j_l^-},\qquad
 V_{il}=B_{i,j_l^+}-B_{i,j_l^-}.
\end{equation}
The law of $U_{il}$ is $(q_0,\pi,q_0)$ on $(0,1,2)$ and does not
depend on $a_i$. If $U_{il}\ne1$, then $V_{il}=0$; conditional on
$U_{il}=1$, $V_{il}$ is a sign with mean $\lambda a_i$. For example,
\[
 \Prob(V_{il}=1\mid U_{il}=1)
 =\frac{(1+wa_i)^2}{2(1+w^2)}.
\]
All pairs in \eqref{low:pair-variables} use disjoint edges.

For the two target rows define
\[
 W_l=U_{ul}+U_{vl},\qquad
 \Delta_l=U_{ul}-U_{vl},\qquad T_l=V_{ul}+V_{vl}.
\]
The indicators of odd $W_l$ have probability $s_0$. At even positions
the joint law of $(W_l,\Delta_l,T_l)$ is common to both orientations.
Indeed, $W=0,4$ forces $\Delta=T=0$; for $W=2$, the possibilities
$(U_u,U_v)=(0,2),(2,0)$ force $T=0$, while $(1,1)$ gives
$\Delta=0$ and the sum of independent signs with means
$\omega\lambda,-\omega\lambda$.

Let $I=\{l:W_l\text{ is odd}\}$ and $N=|I|$. At odd positions put
$t_l=2-W_l$ and $\varepsilon_l=\Delta_l$. Direct multiplication
of the pair probabilities gives
\begin{equation}\label{low:odd-triple}
 \Prob_\omega(t_l=t,\varepsilon_l=e,T_l=z\mid l\in I)
 =\frac18(1+\omega\lambda tez),
 \qquad t,e,z\in\{-1,1\}.
\end{equation}
For example, each admissible ordered pair of $U$ values at an odd
position has probability $q_0\pi$, and its active $V$ sign has
probability $(1+\omega\lambda tez)/2$; division by
$s_0=4q_0\pi$ gives \eqref{low:odd-triple}. Thus $t_l$ and
$\varepsilon_l$ are independent fair signs, with
$\E_\omega(T_l\mid t_l,\varepsilon_l)=\omega\lambda t_l\varepsilon_l$.
The triples are independent across positions, and
$N\sim\operatorname{Bin}(n,s_0)$.

For every background row $i\notin\{u,v\}$, reveal only
$F_{il}=\one_{\{U_{il}=1\}}$. If $F_{il}=0$, write
$U_{il}=1+\eta_{il}$, where $\eta_{il}$ is fair, and $V_{il}=0$.
If $F_{il}=1$, then $U_{il}=1$ and $V_{il}$ has mean $\lambda a_i$.
Conditional on all flags these remaining signs are independent.
Let $\mathcal H$ consist of the background flags, $I$, and the full
$(W_l,\Delta_l,T_l)$ at every even position. Its joint law is common
under the two orientations. In particular, $\mathcal H$ does not reveal the
individual $t_l$ at odd positions.

The first part of the oracle is
\begin{equation}\label{low:first-layer}
 \begin{split}
 R_i^0&=\sum_{l:F_{il}=0}\eta_{il},\qquad
 C_l^0=\sum_{i:F_{il}=0}\eta_{il}-\one_{\{l\in I\}}t_l,\\
 T_0&=\sum_{l\in I}t_l,\qquad
 \mathcal M=(T_0,(R_i^0)_{i\notin\{u,v\}},(C_l^0)_{l=1}^n).
 \end{split}
\end{equation}
Conditional on $\mathcal H$, this is the full vector of node sums of a
bipartite graph with independent fair signs on its retained edges.
Background edge $(i,l)$ is present when $F_{il}=0$ and carries
$\eta_{il}$. An extra row, indexed by $0$, is adjacent to the columns
in $I$ and carries $\xi_l=-t_l$; its sum is $-T_0$.

The second part retains $S_0=\sum_{l\in I}\varepsilon_l$ and all actual
column-pair differences $D_l=C_{j_l^+}-C_{j_l^-}$. They satisfy
\begin{equation}\label{low:second-layer}
 D_l=Z_l+T_l,\qquad
 Z_l=\sum_{i:F_{il}=1}V_{il},\qquad
 F_l=\sum_{i\notin\{u,v\}}F_{il}.
\end{equation}
Conditional on $\mathcal H$, the $Z_l$ are independent, are independent of
$(\mathcal M,\varepsilon)$, and contain $F_l$ signs with means
in $\{\lambda,-\lambda\}$. Before conditioning on $\mathcal H$, the counts
$F_l$ are independent $\operatorname{Bin}(m-2,\pi)$ variables,
independent of $I$. At even positions $T_l$ is known, so $D_l$ is an
independent common coordinate. The full oracle is
\[
 \mathcal O=(\mathcal H,\mathcal M,S_0,(D_l)_{l=1}^n).
\]

Every actual margin is a common deterministic function of this oracle.
In detail, put
\begin{equation}\label{low:cross-reconstruction}
 \begin{aligned}
 R_i&=n+R_i^0 &&(i\notin\{u,v\}),\\
 W_{\mathrm{tot}}&=\sum_{l\notin I}W_l+2N-T_0,\qquad&
 \Delta_{\mathrm{tot}}&=\sum_{l\notin I}\Delta_l+S_0,\\
 R_u&=(W_{\mathrm{tot}}+\Delta_{\mathrm{tot}})/2,&
 R_v&=(W_{\mathrm{tot}}-\Delta_{\mathrm{tot}})/2,\\
 B_l^{\mathrm{tot}}&=
 \begin{cases}m+C_l^0,&l\in I,\\m-2+W_l+C_l^0,&l\notin I,
 \end{cases}\\
 C_{j_l^+}&=(B_l^{\mathrm{tot}}+D_l)/2,&
 C_{j_l^-}&=(B_l^{\mathrm{tot}}-D_l)/2.
 \end{aligned}
\end{equation}
The identities follow by summing the contributions of the background
rows and the target rows. They have the required parity on both
oracle-law supports and contain no $\omega$. Hence Hellinger data
processing bounds the actual margin distance by the oracle distance.

For each odd position, apply \eqref{low:channel-definition} to the
conditional mass function of $Z_l$, obtaining $q_l,\delta_l,\nu_l$.
By Lemma~\ref{low:scalar},
\begin{equation}\label{low:cross-noise}
 \nu_l\le\frac{\lambda^2}{(1-\lambda)(F_l+1)},\qquad
 \nu_l\le\lambda^2,\qquad \E\nu_l\le z_c.
\end{equation}
The weights $\nu_l$ are functions of the background flags in their own
column pair and the fixed background signs. Formula
\eqref{low:odd-triple}, after adding $Z_l$, gives the channel
\eqref{low:abstract-channel}. Conditional on $\mathcal H$, take the reference
law to be the product of the true $\mathcal M$ marginal, the fair-sum
law of $S_0$, the $q_l$ laws at odd positions, and the true common laws
of the even $D_l$. Lemma~\ref{low:projection} gives
\begin{equation}\label{low:cross-chi-identity}
 \chi^2(P_\omega^{\mathcal O}\Vert Q)
 =\E_{\mathcal H}\sum_{\varnothing\ne B\subseteq I}
 \frac{\E_{\mathcal M\mid \mathcal H}a_B(\mathcal M,\mathcal H)^2}{\binom N{|B|}}
       \prod_{l\in B}\nu_l,\qquad
 a_B=\E\left(\prod_{l\in B}t_l\mid\mathcal M,\mathcal H\right).
\end{equation}

\subsection{Fair paths and the complete likelihood bound}

The singleton terms in \eqref{low:cross-chi-identity} require one
more factor of smoothing. For each individual coordinate, a collection
of paths in the fair graph of \eqref{low:first-layer} supplies it.

\begin{lemma}[Weighted singleton projection]\label{low:fan}
Conditional on any nonempty target mask $I$, for each $l\in I$,
\begin{equation}\label{low:fan-bound}
 \E\{\nu_l a_{\{l\}}(\mathcal M,\mathcal H)^2\mid I\}
 \le\frac{2\lambda^2}
 {(1-\lambda)\pi p_f(m-1)N}.
\end{equation}
\end{lemma}

\begin{proof}
Match the $N-1$ columns $k\in I\setminus\{l\}$ to distinct
background rows $i_k$, using a fixed rule depending only on $I$.
There are enough rows because $N-1\le m/2-1\le m-2$. The path
\[
 0\;\text{--}\;k\;\text{--}\;i_k\;\text{--}\;l
\]
survives when $F_{i_k,k}=F_{i_k,l}=0$. These events use disjoint
flags, so they are independent with probability $(1-\pi)^2$.
Conditional on $\mathcal H$, refine $\mathcal M$ by revealing all fair edges
outside the target edge $\xi_l$ and these surviving paths. Subtracting
the revealed edges from the internal node sums of a surviving path
leaves $(a+b,b+c)$ for its three independent fair signs $(a,b,c)$.

The only ambiguous internal observation is $(0,0)$. It occurs with
probability $1/4$ and leaves
$(a,b,c)=(\zeta,-\zeta,\zeta)$ with $\zeta$ fair. Each of the
other six sign configurations is determined by its internal observation.
After subtracting those determined endpoint contributions, the two
endpoint sums give the same quantity
\[
 \xi_l+\zeta_1+\cdots+\zeta_{L_l},
\]
where all signs are independent and fair. Conditional on the internal
reveals, the posterior mean of $\xi_l$ is this sum divided by $L_l+1$,
and its mean square, averaged over the sum, is $1/(L_l+1)$.
The paths are edge-disjoint except at their endpoints. Averaging their
flags and internal reveals, conditional on $I$, therefore gives
\begin{equation}\label{low:fan-count}
 L_l\sim\operatorname{Bin}(N-1,p_f).
\end{equation}

Conditional expectation has nondecreasing mean square under refinement.
Since $\nu_l$ is $\mathcal H$-measurable and $t_l=-\xi_l$, this yields
\begin{equation}\label{low:weighted-refinement}
 \E(\nu_l a_{\{l\}}^2\mid I)
 \le\E\left(\frac{\nu_l}{L_l+1}\,\middle|\,I\right)
 \le\frac{\lambda^2}{1-\lambda}
 \E\left(\frac1{(F_l+1)(L_l+1)}\,\middle|\,I\right).
\end{equation}
The two counts in the last expectation need not be independent: a
surviving path forces one flag in column $l$ to be zero. Cauchy--Schwarz
and \eqref{low:reciprocal-second}, applied to their marginal laws
$\operatorname{Bin}(m-2,\pi)$ and \eqref{low:fan-count}, instead give
\[
 \E(\nu_l a_{\{l\}}^2\mid I)
 \le\frac{2\lambda^2}
 {(1-\lambda)\pi p_f\sqrt{(m-1)mN(N+1)}}
 \le\frac{2\lambda^2}{(1-\lambda)\pi p_f(m-1)N}.
\]
This also holds for $N=1$. Each $l$ uses its own refinement solely
to bound its projection; the different refinements are not jointly
appended to the oracle.
\end{proof}

Let $T_1$ denote the contribution of singleton subsets in
\eqref{low:cross-chi-identity}. Summing \eqref{low:fan-bound} over
$l\in I$ with its factor $1/N$, and then averaging
$N\sim\operatorname{Bin}(n,s_0)$, gives
\begin{equation}\label{low:singleton-contribution}
 T_1\le\frac{8\lambda^2}
 {(1-\lambda)\pi p_f s_0(m-1)(m+2)}.
\end{equation}
Here
\[
 \E\{N^{-1};N>0\}
 \le2\E(N+1)^{-1}
 \le\frac{2}{(n+1)s_0}
\]
by \eqref{low:reciprocal-first}.

For every subset with $|B|\ge2$, use $a_B^2\le1$ before averaging
the flags. Conditional on $I$, the background flag columns remain
independent, and \eqref{low:cross-noise} implies, for $2\le k\le N$,
\[
 \E\left\{\binom Nk^{-1}\sum_{|B|=k}\prod_{l\in B}\nu_l
          \,\middle|\,I\right\}\le z_c^k.
\]
Thus the full likelihood, including every higher order, satisfies
\begin{equation}\label{low:cross-chi-bound}
 \chi^2(P_\omega^{\mathcal O}\Vert Q)
 \le T_1+\frac{z_c^2}{1-z_c}.
\end{equation}
Combining \eqref{low:hellinger-reference},
\eqref{low:cross-reconstruction}, and
\eqref{low:singleton-contribution} proves
\begin{equation}\label{low:cross-hellinger}
 h^2(Q_{ab^\top},Q_{a'b^\top})\le D_c(m,w).
\end{equation}
The quantities $\pi,p_f,s_0$ are bounded away from zero throughout
\eqref{eq:signal}, while $z_c=O(\lambda^2/m)$. Consequently
$D_c(m,w)\le C\lambda^2/m^2$, with an absolute constant over that
signal interval. This proves the upper half of \eqref{eq:swap}.

\subsection{A within-group degree comparison}

We next compare the full degree vectors of a symmetric hollow
$m$-node graph when two opposite signs $u,v$ are swapped. Pair the
remaining nodes as $G_l=(j_l^+,j_l^-)$, $1\le l\le n_w=(m-2)/2$.
Since $m\in4\N$, $n_w$ is odd. Orient a regular tournament on these
pair indices, with every in-degree and out-degree equal to
$d_w=(n_w-1)/2$. For instance, with indices modulo $n_w$, direct
$k\to l$ when $1\le(l-k)\bmod n_w\le(n_w-1)/2$. When $m=4$,
there is one pair and no tournament edge.

For each $k\to l$, regard the four undirected graph edges between
$G_k$ and $G_l$ as a two-by-two block with source rows $G_k$ and
receiver columns $G_l$. Reveal the two source-row sums
$U_{k+,l},U_{k-,l}$, leaving the corresponding source-row differences
$V_{k+,l},V_{k-,l}$ unobserved. These differences are taken in the
order $(j_l^+,j_l^-)$. Also reveal
\[
 E_l=A_{j_l^+,j_l^-}\quad(1\le l\le n_w),\qquad E_0=A_{uv}.
\]
For the two target rows and the background column pairs, use the
variables $W_l,\Delta_l,T_l$ defined above. In this construction reveal
every $W_l$, and reveal $\Delta_l,T_l$ whenever $W_l$ is even.
Let $\mathcal H_w$ be this entire collection of reveals.

The source-row sums are independent, each with probabilities
$(q_0,\pi,q_0)$, because they involve disjoint edge pairs. The
$E_l,E_0$ are independent common factors, each joining opposite
signs. The target reveal also has a common law by the pair calculation
above. Thus the joint distribution of $\mathcal H_w$ is the same for both
orientations.

Define the known outgoing totals, the known incoming pair totals,
and the remaining incoming differences by
\begin{equation}\label{low:tournament-totals}
 \begin{split}
 O_l^+&=\sum_{h:l\to h}U_{l+,h},\qquad
 O_l^-=\sum_{h:l\to h}U_{l-,h},\\
 I_l&=\sum_{k:k\to l}(U_{k+,l}+U_{k-,l}),\qquad
 Z_l=\sum_{k:k\to l}(V_{k+,l}+V_{k-,l}).
 \end{split}
\end{equation}
Conditional on $\mathcal H_w$, $Z_l$ is a sum of independent signs with
means $\lambda$ or $-\lambda$, one for each incoming source-row sum
equal to one. Before the reveal its number of signs has distribution
\begin{equation}\label{low:within-flag-count}
 F_l\sim\operatorname{Bin}(2d_w,\pi)
       =\operatorname{Bin}(m/2-2,\pi).
\end{equation}
The $Z_l$ are conditionally independent across receivers. Each
unordered background block is assigned by the tournament to one
receiver, and its remaining differences occur only in that receiver's
$Z_l$. Its contribution to the source degrees is already fixed by
the revealed source-row sums. This gives a product of conditional
laws even though the original adjacency is symmetric.

Let $d_j$ denote actual within-group degrees. The oracle adds the
odd-position target sum $S=\sum_{l:W_l\text{ odd}}\Delta_l$ and
the residual degree differences
\begin{equation}\label{low:within-channel}
 D_l=(d_{j_l^+}-d_{j_l^-})-(O_l^+-O_l^-)=Z_l+T_l.
\end{equation}
At an odd position, $t_l=2-W_l$ is now known, and
$\varepsilon_l=\Delta_l$ is a fair sign. Conditioning
\eqref{low:odd-triple} on the revealed $W_l$ gives
$\E(T_l\mid\varepsilon_l,\mathcal H_w)=\omega\lambda t_l\varepsilon_l$.
The odd-position $\varepsilon_l$ are independent; they and the
target mask are independent of all background blocks. At even
positions, $T_l$ is known and $D_l$ is an independent common
coordinate.

Write $\mathcal O_w=(\mathcal H_w,S,(D_l)_{l=1}^{n_w})$. At each odd receiver,
define $q_l,\delta_l,\nu_l$ by \eqref{low:channel-definition} from
the conditional mass function of $Z_l$. Give $Q_w$ the true $\mathcal H_w$
marginal and, conditional on $\mathcal H_w$, the product of the fair-sum law
of $S$, the odd $q_l$ laws, and the true common even-coordinate laws.
Apply Lemma~\ref{low:projection} conditional on $\mathcal H_w$, taking
$\mathcal M$ constant and the $t_l$ fixed. Then $a_B^2=1$ and
\eqref{low:projection-chi} becomes
\[
 \chi^2(P_\omega^{\mathcal O_w}\Vert Q_w)
 =\E_{\mathcal H_w}\sum_{k=1}^N
       \binom Nk^{-1}\sum_{|B|=k}\prod_{l\in B}\nu_l.
\]
Here $N$ is the number of odd target positions. The background flag
collections are independent across receivers and independent of
the entire target reveal. Lemma~\ref{low:scalar} with
\eqref{low:within-flag-count} gives $\E\nu_l\le z_w$, hence
\begin{equation}\label{low:within-chi}
 \chi^2(P_\omega^{\mathcal O_w}\Vert Q_w)
 \le\sum_{k\ge1}z_w^k=\frac{z_w}{1-z_w}.
\end{equation}

For completeness, the oracle determines every degree by
\begin{equation}\label{low:within-reconstruction}
 \begin{aligned}
 \Delta_{\mathrm{tar}}&=S+\sum_{l:W_l\text{ even}}\Delta_l,\\
 d_u&=E_0+\tfrac12\left(\sum_lW_l+\Delta_{\mathrm{tar}}\right),&
 d_v&=E_0+\tfrac12\left(\sum_lW_l-\Delta_{\mathrm{tar}}\right),\\
 B_l^{\mathrm{deg}}&=O_l^++O_l^-+I_l+2E_l+W_l,&
 J_l^{\mathrm{deg}}&=O_l^+-O_l^-+D_l,\\
 d_{j_l^+}&=(B_l^{\mathrm{deg}}+J_l^{\mathrm{deg}})/2,&
 d_{j_l^-}&=(B_l^{\mathrm{deg}}-J_l^{\mathrm{deg}})/2.
 \end{aligned}
\end{equation}
The pair degree sum receives outgoing totals, incoming totals, two
copies of the pair's internal edge, and target total $W_l$. The
internal edge cancels in the degree difference. The target degrees
receive $E_0$ once each. All terms are observed by the oracle and
the reconstruction is common to both orientations. Equations
\eqref{low:hellinger-reference} and \eqref{low:within-chi} therefore
give the full within-group bound
\begin{equation}\label{low:within-hellinger}
 h_w^2\le\frac{4z_w}{1-z_w}.
\end{equation}
This construction and bound include $m=4$, where there are no
background noise signs. From $\lambda\le1/16$, $\pi\ge1/2$, and
$m/2-1\ge1$, we have $z_w\le1/120$.

\subsection{Many orientation bits and labeled probability loss}

Set
\begin{equation}\label{low:affinity-constants}
 \beta_w=1-\frac{2z_w}{1-z_w},\qquad
 \beta_c=\max\{0,1-D_c(m,w)/2\},\qquad
 A_{K,m}=\beta_w\beta_c^{K-1},
\end{equation}
with the cross factor interpreted as one when $K=1$.
The bound $z_w\le1/120$ ensures $\beta_w>0$.

At any fixed parameter, the observed array $Y$ is a rearrangement
of the $K$ within-group degree vectors and the $\binom K2$
cross-group margin vectors. Their underlying unordered edges are
disjoint, so these experiments are independent. A swap in one
group affects its within-group block and its $K-1$ cross-group
blocks only. The latter are covered by \eqref{low:cross-hellinger},
transposing a block if the changed nodes are its columns. Since
affinities multiply under independent products,
\begin{equation}\label{low:full-affinity}
 \aff(P_+^Y,P_-^Y)\ge A_{K,m}.
\end{equation}

In each group reserve $m/2$ anchor nodes, assigning half each fixed
sign. Divide its other $m/2$ nodes into $m/4$ labeled pairs and
assign each pair $(+1,-1)$ or $(-1,+1)$ independently. The resulting
hypercube has $d=r/4$ orientation bits. Every sign vector in it is
balanced within each group and has $\alpha=0$, so belongs to
\eqref{eq:class}.

For a bit $b$, let $\mathcal E_b$ contain the ordered dyads between
its two nodes and all $r/2$ anchors, in both directions. These sets
are disjoint across bits. Each has $2r$ coordinates, and every one
changes by magnitude $2\kappa$ when its bit flips. Thus the two
probability subvectors on $\mathcal E_b$ have squared separation
\begin{equation}\label{low:anchor-separation}
 D_b^2=8r\kappa^2.
\end{equation}

To convert this separation into risk, consider arbitrary observation
laws $P_+,P_-$ with target vectors $\theta_+,\theta_-$ separated by
distance $D$. For any estimator $\widehat\theta$,
\[
 \|\widehat\theta-\theta_+\|^2+
 \|\widehat\theta-\theta_-\|^2\ge D^2/2.
\]
Integrating against the smaller density and dividing by two proves
\begin{equation}\label{low:two-point-risk}
 \frac12\left\{\E_+\|\widehat\theta-\theta_+\|^2+
                   \E_-\|\widehat\theta-\theta_-\|^2\right\}
 \ge\frac{D^2}{4}\{1-\TV(P_+,P_-)\}.
\end{equation}
This also holds for randomized estimators by integrating their
conditional output distribution. Cauchy--Schwarz applied to
$|p-q|=|\sqrt p-\sqrt q|(\sqrt p+\sqrt q)$ gives
\begin{equation}\label{low:tv-affinity}
 \TV(P,Q)\le\sqrt{1-\aff(P,Q)^2}.
\end{equation}

We use the Assouad reduction as in \citet[Lemma~2]{yu1997assouad}.
Give the hypercube the uniform prior. For each bit, condition on
all other bits and apply \eqref{low:two-point-risk} to the estimator
restricted to $\mathcal E_b$. The uniform affinity bound
\eqref{low:full-affinity} and \eqref{low:anchor-separation} make its
average unnormalized risk at least
\[
 2r\kappa^2\{1-\sqrt{1-A_{K,m}^2}\}.
\]
Sum over the $r/4$ disjoint coordinate sets and divide by $r(r-1)$.
The minimax risk dominates this Bayes risk on a subfamily, yielding
\begin{equation}\label{low:global-risk}
 \begin{split}
 \risk_r^Y(T)&\ge\frac{r\kappa^2}{2(r-1)}
       \left\{1-\sqrt{1-\beta_w^2\beta_c^{2(K-1)}}\right\}\\
 &\ge\frac{r\kappa^2}{4(r-1)}\,
                  \beta_w^2\beta_c^{2(K-1)}.
 \end{split}
\end{equation}
The last inequality uses $1-\sqrt{1-x}\ge x/2$ for $0\le x\le1$.

\section{A full-law comparison at fixed composition}
\label{sec:alignment-lower}

A composition-preserving swap remains difficult to distinguish when
the total alignment and the trait resolution are bounded. The following
comparison includes all counts in $Y$, with their shared-edge dependence.
Throughout this section the activity vector is zero.

\begin{proposition}[Composition-dependent affinity]
\label{prop:composition-affinity}
Fix $0<\tau\le\operatorname{arctanh}(1/16)$ and $B<\infty$ with
$B\ge0$. There is a finite constant $C_{B,\tau}$ such that the following
holds for every $K\ge1$ and $m\in4\mathbb N$. Let $s_+,s_-$ be globally
balanced sign vectors that differ only by swapping two opposite signs
in one trait cell. Suppose their common composition $b=T^\top s_\pm$
satisfies $\Lambda=\|b\|_2^2/m\le B$. If $P_\pm^Y$ are the corresponding
laws of the complete array $Y=AT$, then
\begin{equation}\label{alow:full-affinity}
 \aff(P_+^Y,P_-^Y)
 \ge \exp\left\{-C_{B,\tau}
       \left(\Lambda+\rho+\frac1m\right)\right\},
 \qquad \rho=\frac K{m^2}.
\end{equation}
The constant is independent of the composition and of all signs outside
the swapped pair.
\end{proposition}

We first prove the two block comparisons from which
\eqref{alow:full-affinity} follows. Write
\begin{equation}\label{alow:parameters}
 \begin{gathered}
 w=2\kappa=\tanh(\tau/2),\qquad \lambda=\tanh\tau,\qquad
 q_0=\frac{1-w^2}{4},\qquad \pi=\frac{1+w^2}{2},\\
 s_0=4q_0\pi,\qquad p_f=\frac{(1-\pi)^2}{4}.
 \end{gathered}
\end{equation}
Thus $0<\lambda\le1/16$. We use the scalar smoothing, fair-sum
projection, and weighted-path bounds of Lemmas~\ref{low:scalar},
\ref{low:projection}, and \ref{low:fan}.
Constants denoted by $C_{B,\tau}$ or $c_{B,\tau}$
may change from line to line, and are respectively finite positive upper
constants and strictly positive lower constants.

\subsection{A cross-block oracle with surplus columns}

Consider an $m\times m$ cross block $X$ with row signs $a_i$ and
column signs $c_j$. Its independent entries satisfy
\[
 \Prob(X_{ij}=1)=\frac{1+w a_i c_j}{2}.
\]
Only the two row signs $(a_u,a_v)=(\omega,-\omega)$ vary, with
$\omega\in\{-1,1\}$. All background signs are fixed. Put
$\beta=\sum_jc_j$, $d=|\beta|$, and $\sigma=\operatorname{sign}(\beta)$,
choosing either sign if $d=0$. Partition the columns into
$n=(m-d)/2$ opposite-sign pairs $(j_l^+,j_l^-)$ and a set $J$ of
$d$ surplus columns, all with sign $\sigma$.

For each paired column define
\[
 U_{il}=X_{i,j_l^+}+X_{i,j_l^-},\qquad
 V_{il}=X_{i,j_l^+}-X_{i,j_l^-}.
\]
The law of $U_{il}$ is $(q_0,\pi,q_0)$ on $(0,1,2)$, independently
of the row sign. Conditional on $U_{il}=1$, $V_{il}$ is a sign of mean
$\lambda a_i$; otherwise $V_{il}=0$. For the focal rows put
\[
 W_l=U_{ul}+U_{vl},\qquad
 \Delta_l=U_{ul}-U_{vl},\qquad T_l=V_{ul}+V_{vl}.
\]
Let $I=\{l:W_l\text{ is odd}\}$ and $N=|I|$. Then
$N\sim\operatorname{Bin}(n,s_0)$. At an odd position set
$t_l=2-W_l$ and $\varepsilon_l=\Delta_l$. The calculation in
\eqref{low:odd-triple} gives
\begin{equation}\label{alow:odd-law}
 \Prob_\omega(t_l=t,\varepsilon_l=e,T_l=z\mid l\in I)
 =\frac18(1+\omega\lambda tez),\qquad t,e,z\in\{-1,1\}.
\end{equation}
In particular $t_l,\varepsilon_l$ are independent fair signs. At every
even position reveal $(W_l,\Delta_l,T_l)$, whose law is common to the
two orientations.

For each background row $i\notin\{u,v\}$ reveal the flags
$F_{il}=\one_{\{U_{il}=1\}}$. If $F_{il}=0$, write
$U_{il}=1+\eta_{il}$, with $\eta_{il}$ fair and $V_{il}=0$.
Conditional on the flags, these fair signs and the remaining active
$V_{il}$ signs are independent. Retain the first-layer observation
\begin{equation}\label{alow:first-layer}
 \begin{split}
 R_i^0&=\sum_{l:F_{il}=0}\eta_{il},\qquad
 C_l^0=\sum_{i\notin\{u,v\}:F_{il}=0}\eta_{il}
                 -\one_{\{l\in I\}}t_l,\\
 T_0&=\sum_{l\in I}t_l,\qquad
 \mathcal M=(T_0,(R_i^0)_{i\notin\{u,v\}},(C_l^0)_{l=1}^n).
 \end{split}
\end{equation}
This is the same fair bipartite node-sum experiment as
\eqref{low:first-layer}, now with $n$ paired columns. For the paired
column differences retain
\begin{equation}\label{alow:paired-channels}
 D_l=C_{j_l^+}-C_{j_l^-}=Z_l+T_l,
 \qquad Z_l=\sum_{i\notin\{u,v\}:F_{il}=1}V_{il},
\end{equation}
where $C_j$ is the actual column total. Conditional on the flags,
the $Z_l$ are independent and independent of the fair first-layer
variables and the $\varepsilon_l$. Their numbers of summands are
independent $\operatorname{Bin}(m-2,\pi)$ variables before the reveal.
These statements do not require balance among the background row signs.

For each surplus column $h\in J$, reveal all its background edges and
only the focal sum
\[
 U_h=X_{uh}+X_{vh}.
\]
Its law is again $(q_0,\pi,q_0)$, common under $\omega$. When $U_h=1$,
the unobserved difference $E_h=X_{uh}-X_{vh}$ is a sign of mean
$\omega\sigma\lambda$; otherwise it is zero. Let
$D=\#\{h\in J:U_h=1\}\le d$. Combine these differences with the
paired focal differences in the single observation
\begin{equation}\label{alow:combined-sum}
 S=\sum_{l\in I}\varepsilon_l+\sum_{h\in J:U_h=1}E_h.
\end{equation}
The common ancillary $H$ consists of the background flags, the mask $I$,
the even-position reveals, the surplus background edges, and the $U_h$.
The oracle is
\[
 \mathcal O=(H,\mathcal M,S,(D_l)_{l=1}^n).
\]

Every row and column margin is a common deterministic function of
$\mathcal O$. To give the reconstruction explicitly, put
$Q_i=\sum_{h\in J}X_{ih}$ for a background row and
$Q_h=\sum_{i\notin\{u,v\}}X_{ih}$ for a surplus column. These are
revealed. The actual row totals and paired-column totals satisfy
\begin{equation}\label{alow:cross-reconstruction}
 \begin{aligned}
 R_i&=n+R_i^0+Q_i &&(i\notin\{u,v\}),\\
 W_*&=\sum_{l\notin I}W_l+2N-T_0+\sum_{h\in J}U_h,&
 \Delta_*&=\sum_{l\notin I}\Delta_l+S,\\
 R_u&=(W_*+\Delta_*)/2,&R_v&=(W_*-\Delta_*)/2,\\
 B_l^*&=\begin{cases}
             m+C_l^0,&l\in I,\\
             m-2+W_l+C_l^0,&l\notin I,
          \end{cases}\\
 C_{j_l^+}&=(B_l^*+D_l)/2,&C_{j_l^-}&=(B_l^*-D_l)/2,\\
 C_h&=Q_h+U_h &&(h\in J).
 \end{aligned}
\end{equation}
The formulas have the required parity on both oracle supports. In
particular, reconstruction uses the combined sum $S$ and never separates
an individual surplus focal difference.

\subsection{The augmented projection identity}

At each odd paired position, apply \eqref{low:channel-definition} to
the conditional mass function of $Z_l$, obtaining $q_l,\delta_l,\nu_l$.
Lemma~\ref{low:scalar} gives
\begin{equation}\label{alow:noise-bound}
 \E_{q_l}\delta_l=0,\qquad \nu_l\le\lambda^2,\qquad
 \E\nu_l\le z_c:=\frac{\lambda^2}{(1-\lambda)\pi(m-1)}.
\end{equation}
The last expectation averages the flags in column pair $l$. The
corresponding channel density is
$1+\omega t_l\varepsilon_l\delta_l(D_l)$.

Give a reference law $Q$ the true common marginal of $H$. Conditional
on $H$, give $\mathcal M$ its true common law and let $S$ have the
law of a sum of $M_0=N+D$ independent fair signs. Take $\mathcal M$, $S$
and all paired-channel coordinates mutually independent under this
conditional reference. The channel laws are $q_l$ at odd positions and
their true common laws at even positions. This references the $D$ active
surplus differences to fair signs as well. Write
\[
 a_A=\E\left(\prod_{l\in A}t_l\,\middle|\,\mathcal M,H\right),
 \qquad A\subseteq I.
\]
For a sum $S$ of $M_0$ fair signs, define $g_k^{(M_0)}(S)$ to be the
conditional expectation of the product of any $k$ distinct signs,
given $S$. Set $g_0^{(M_0)}=1$; when $M_0=0$ this is the only term.
The density of the oracle law is exactly
\begin{equation}\label{alow:likelihood}
 L_\omega:=\frac{dP_\omega^{\mathcal O}}{dQ}
 =\sum_{A\subseteq I}a_A\prod_{l\in A}\delta_l(D_l)
       \sum_{j=0}^D\binom Dj(\sigma\lambda)^j
          \omega^{|A|+j}g_{|A|+j}^{(N+D)}(S).
\end{equation}
Indeed, before retaining only the sum $S$, the density is the product
of the paired factors $1+\omega t_l\varepsilon_l\delta_l(D_l)$ and
the surplus factors $1+\omega\sigma\lambda\eta_h$, where the $\eta_h$
are reference fair signs. Expanding this product and conditioning on
$(H,\mathcal M,S,(D_l))$ gives \eqref{alow:likelihood}. Each subset of
$j$ surplus signs has the same projection, accounting for $\binom Dj$.

The Walsh calculation in Lemma~\ref{low:projection} also gives
\begin{equation}\label{alow:walsh}
 \E_Q\{g_k^{(M_0)}(S)g_{k'}^{(M_0)}(S)\mid H\}
 =\begin{cases}\binom{M_0}{k}^{-1},&k=k',\\0,&k\ne k'.\end{cases}
\end{equation}
For distinct subsets $A$, the centered independent channel factors
give orthogonality. For fixed $A$ and distinct $j$, use
\eqref{alow:walsh}. The term $(A,j)=(\varnothing,0)$ is one. Thus
\begin{equation}\label{alow:chi-identity}
 \chi^2(P_\omega^{\mathcal O}\Vert Q)
 =\E\sum_{\substack{A\subseteq I,\ 0\le j\le D\\
                    (|A|,j)\ne(0,0)}}
  \frac{a_A^2\prod_{l\in A}\nu_l\binom Dj^2\lambda^{2j}}
       {\binom{N+D}{|A|+j}}.
\end{equation}
The expectation includes $H$ and $\mathcal M$. In particular this
identity retains all likelihood orders.

\subsection{Summing the likelihood orders}

Suppose $d^2/m\le B$. For all sufficiently large $m$, depending only
on $B$, one has $d\le m/2$ and $n\ge m/4$. Since
$N\sim\operatorname{Bin}(n,s_0)$, there are constants $c_\tau,c'_\tau>0$
such that
\begin{equation}\label{alow:good-mask}
 G=\{N\ge c_\tau m\},\qquad \Prob(G^c)\le e^{-c'_\tau m}.
\end{equation}
The background flag collections are independent across paired columns
and independent of $I,D$. We may use the common bound
$z=z_c\le C_\tau/m$ for their averaged channel energies.

The terms $|A|=1,j=0$ retain the additional first-layer smoothing.
The proof of Lemma~\ref{low:fan} applies with $n$ paired columns:
its fan uses $N-1\le n-1\le m-2$ distinct background rows, and its
scalar bound permits arbitrary background signs. The new surplus
reveals are independent of this fair graph. Consequently, for $l\in I$,
\[
 \E(\nu_l a_{\{l\}}^2\mid I)
 \le \frac{2\lambda^2}
 {(1-\lambda)\pi p_f(m-1)N}.
\]
The denominator $N+D$ in \eqref{alow:chi-identity} is at least $N$.
Summing over $l$ and using \eqref{low:reciprocal-first} therefore bounds
the entire singleton contribution by
\begin{equation}\label{alow:singleton}
 \frac{C_\tau}{m}\E\{N^{-1};N>0\}
 \le \frac{C_\tau}{m(n+1)}=O_\tau(m^{-2}).
\end{equation}
As in Lemma~\ref{low:fan}, the noise count and the surviving-path
count may be dependent; their reciprocal second moments, rather than
an independence assumption, give the displayed weighted bound.

For all remaining terms first use $a_A^2\le1$, and then average the
independent flag columns. For $0\le k\le N$ and $0\le j\le D$, put
\[
 A_{kj}=\frac{\binom Nk\binom Dj^2}{\binom{N+D}{k+j}}.
\]
Their flag-averaged contributions are at most
$A_{kj}z^k\lambda^{2j}$. Falling factorials give
\begin{equation}\label{alow:coefficient}
 A_{kj}\le\binom{k+j}{j}\frac{D^{2j}}{j!(N-k+1)^j}.
\end{equation}
To see this, cancel $(N)_k$ in the numerator against the first $k$
factors of $(N+D)_{k+j}$. Each of its remaining $j$ factors is at
least $N-k+1$, since $j\le D$, and $(D)_j\le D^j$.
For $N>0$ and $k\le N/2$, the inequalities
$N-k+1\ge N/2$ and $\binom{k+j}{j}\le2^{k+j}$ imply
\begin{equation}\label{alow:exponential-sum}
 \sum_{j=1}^D A_{kj}\lambda^{2j}
 \le 2^k\left\{\exp\left(\frac{4\lambda^2D^2}{N}\right)-1\right\}.
\end{equation}
On $G$, $D^2/N\le B/c_\tau$. The braces in
\eqref{alow:exponential-sum} are at most $C_{B,\tau}d^2/m$.
Thus the $k=0,j\ge1$ terms contribute at most
$C_{B,\tau}d^2/m$, and the $k=1,j\ge1$ terms at most
$C_{B,\tau}d^2/m^2$. For $k\ge2,j\ge1$ with $k\le N/2$, summing
$(2z)^k$ gives $O_{B,\tau}(m^{-2})$. The $j=0,k\ge2$ terms contribute
at most $\sum_{k\ge2}z^k=z^2/(1-z)$.

For $k>N/2$, Vandermonde's identity implies
$\binom{N+D}{k+j}\ge\binom Nk\binom Dj$, and hence
$A_{kj}\le\binom Dj$. For sufficiently large $m$ the remaining tail
is bounded by
\begin{equation}\label{alow:large-order-tail}
 \sum_{k>N/2}\sum_{j=0}^D A_{kj}z^k\lambda^{2j}
 \le \frac{(1+\lambda^2)^D z^{N/2}}{1-z}.
\end{equation}
On $G$, its logarithm is at most
$O_{B,\tau}(\sqrt m)-c_\tau m\log(m)/2+O_\tau(m)$, so this is
$o(m^{-2})$ uniformly over the allowed compositions.

It remains to control $G^c$. Dropping the binomial denominator in
\eqref{alow:chi-identity}, and including its constant term, gives the
conditional flag-averaged bound
\begin{equation}\label{alow:rare-mask}
 \E_{H\mid N,D}\left[
  1+\chi^2\{P_\omega^{\mathcal O\mid H}\Vert Q^{\mathcal O\mid H}\}
 \right]
 \le (1+z)^N(1+\lambda)^{2D}.
\end{equation}
Here we used
$\sum_j\binom Dj^2\lambda^{2j}\le(1+\lambda)^{2D}$ and averaged
the independent background flag collections before replacing their
energies by $z$. This replacement is not a pointwise bound conditional
on every flag realization. Since $N\le m/2$ and $D\le\sqrt{Bm}$,
the right side of \eqref{alow:rare-mask} is
$\exp\{O_{B,\tau}(\sqrt m)\}$. Multiplication by the probability in
\eqref{alow:good-mask} makes the rare-mask contribution $o(m^{-2})$.
Together with \eqref{alow:singleton}, these estimates prove
\begin{equation}\label{alow:cross-chi}
 \chi^2(P_\omega^{\mathcal O}\Vert Q)
 \le C_{B,\tau}\left(\frac{d^2}{m}+\frac1{m^2}\right)
\end{equation}
for all sufficiently large $m$.

\subsection{Positive overlap for bounded alignment}

To obtain a positive affinity for every bounded value of $d^2/m$, keep
only the $A=\varnothing$ terms in \eqref{alow:likelihood}, and denote
the resulting density by $L_\omega^0$. This is a normalized probability
density: conditional on $H$, its $S$ coordinate is a sum of $N$ fair
signs and $D$ independent signs of mean $\omega\sigma\lambda$,
independent of the common first-layer and channel coordinates. Let
$P_\omega^0$ denote this law. Orthogonality in
\eqref{alow:chi-identity} shows that
$\|L_\omega-L_\omega^0\|_{L^2(Q)}^2$ equals precisely the sum over
$A\ne\varnothing$. The preceding estimates give
\begin{equation}\label{alow:scalar-approximation}
 \|L_\omega-L_\omega^0\|_{L^2(Q)}^2
 \le C_{B,\tau}m^{-1},\qquad
 \TV(P_\omega^{\mathcal O},P_\omega^0)
 \le C_{B,\tau}m^{-1/2}.
\end{equation}
In detail, the new $k=1,j\ge1$ contribution is
$O_{B,\tau}(d^2/m^2)=O_{B,\tau}(m^{-1})$; the $j=0$ singleton,
the higher orders, and the two tails have the bounds already proved.
The total variation inequality follows from Cauchy--Schwarz.

On $G$, consider the lattice points $x$ with $|x|\le\sqrt N$ and
parity $N+D$. Every possible surplus sum $t$ has parity $D$ and
$|t|\le D\le\sqrt{Bm}$. Thus $x-t$ has parity $N$ and
$|x-t|\le C_{B,\tau}\sqrt N$. If $V_N$ is a sum of $N$ fair signs,
then, uniformly over such points and all sufficiently large $N$,
\begin{equation}\label{alow:central-mass}
 \Prob(V_N=x-t)\ge c_{B,\tau}N^{-1/2}.
\end{equation}
For completeness, the central binomial mass is bounded below by a
constant times $N^{-1/2}$. The ratio of a mass at distance $2h$ to a
central mass is a product of consecutive binomial ratios; for
$h=O_{B,\tau}(\sqrt N)$ its logarithm is bounded below by
$-C_{B,\tau}h^2/N$. For odd $N$ use either adjacent central lattice
point. This proves \eqref{alow:central-mass} with the asserted parity.

Each of the two scalar endpoint laws is a mixture over $t$, so each
assigns at least $c_{B,\tau}N^{-1/2}$ to every one of these
order-$\sqrt N$ common lattice points. Their overlap, the sum of their
pointwise minimum masses, is therefore at least $c_{B,\tau}>0$.
All other coordinates of $P_\pm^0$ have the same conditional law.
Averaging over the common ancillary, and using
\eqref{alow:good-mask} and \eqref{alow:scalar-approximation}, proves
\begin{equation}\label{alow:positive-overlap}
 1-\TV(P_+^{\mathcal O},P_-^{\mathcal O})\ge c_{B,\tau}>0
\end{equation}
for sufficiently large $m$, uniformly over the background signs.
This step supplies positive overlap even when the bound in
\eqref{alow:cross-chi} is not small.

Put $x=d^2/m+m^{-2}$. For $x$ below a sufficiently small positive
constant, \eqref{low:hellinger-reference} and
\eqref{alow:cross-chi} give
$\aff(P_+^{\mathcal O},P_-^{\mathcal O})\ge1-C_{B,\tau}x
\ge\exp(-C'_{B,\tau}x)$. For $x$ above that constant,
\eqref{alow:positive-overlap} gives the same exponential form after
enlarging $C'_{B,\tau}$, since affinity dominates overlap.
Applying data processing to \eqref{alow:cross-reconstruction} yields
\begin{equation}\label{alow:cross-affinity}
 \aff(P_+^{\mathrm{cross}},P_-^{\mathrm{cross}})
 \ge\exp\left\{-C_{B,\tau}
                   \left(\frac{\beta^2}{m}+\frac1{m^2}\right)\right\},
 \qquad \frac{\beta^2}{m}\le B.
\end{equation}
This extends to every finite $m\in4\mathbb N$. Indeed, the two
full-adjacency laws differ on exactly $2m$ Bernoulli entries, each
contributing affinity $\sqrt{1-w^2}$. Their affinity is
$(1-w^2)^m$, which lower-bounds the affinity after taking margins.
For the finitely many $m$ below the preceding $B,\tau$-dependent
threshold, use $x\ge m^{-2}$ and enlarge $C_{B,\tau}$. The resulting
constant is uniform in all admissible signs and compositions.

\subsection{The within-cell degree experiment}

Let the changed cell have composition $\beta=b_g$. Remove the two
focal nodes $u,v$ and pair its remaining opposite signs. There are
$n_w=(m-2-d)/2$ pairs $G_l=(j_l^+,j_l^-)$ and $d=|\beta|$ surplus
nodes of sign $\sigma$. We first work at sufficiently large $m$
under $d^2/m\le B$, so that $n_w\ge2$ and $n_w-1$ is of order $m$.

Orient a near-regular tournament on the $n_w$ pair indices. For odd
$n_w$ use a regular cyclic tournament; for even $n_w$, remove one
vertex from a regular tournament on $n_w+1$ vertices. Its in-degrees
$d_l^-$ differ by at most one and satisfy
$d_l^-\ge\lfloor(n_w-1)/2\rfloor$. For each $k\to l$, regard the
four edges between $G_k$ and $G_l$ as a block with source rows $G_k$
and receiver columns $G_l$. Reveal its source-row sums
$U_{k+,l},U_{k-,l}$, leaving the source-row differences
$V_{k+,l},V_{k-,l}$ hidden. Also reveal
\[
 E_l=A_{j_l^+,j_l^-},\qquad E_0=A_{uv},
\]
all background edges touching a surplus node, and, for every surplus
node $h$, only the focal sum $U_h=A_{uh}+A_{vh}$. For the focal edges
to paired nodes, use $W_l,\Delta_l,T_l$ as before. Reveal every $W_l$
and also $\Delta_l,T_l$ at even positions.

Denote all these reveals by $H_w$. The source-row sums have law
$(q_0,\pi,q_0)$, and all remaining background edges have fixed signs
at both endpoints. The edges $E_l,E_0$ join opposite signs. The focal
reveals have the common laws already calculated. Thus $H_w$ has a
common distribution under the two orientations. Define
\begin{equation}\label{alow:tournament}
 \begin{split}
 O_l^+&=\sum_{h:l\to h}U_{l+,h},\qquad
 O_l^-=\sum_{h:l\to h}U_{l-,h},\\
 I_l&=\sum_{k:k\to l}(U_{k+,l}+U_{k-,l}),\qquad
 Z_l=\sum_{k:k\to l}(V_{k+,l}+V_{k-,l}).
 \end{split}
\end{equation}
Conditional on $H_w$, the receiver noises $Z_l$ are independent sums
of signs of means in $\{\lambda,-\lambda\}$. Before the reveal, their
counts are independent
\begin{equation}\label{alow:within-flags}
 F_l\sim\operatorname{Bin}(2d_l^-,\pi).
\end{equation}
Every unordered background block is assigned to exactly one receiver;
its contribution to the source degrees is already fixed by its two
revealed row sums. This accounts for conditional independence despite
the symmetry of the original adjacency matrix.

Let $Q_l^\pm$ be the revealed number of neighbors of $j_l^\pm$ among
the surplus nodes. For a surplus node $h$, let $Q_h$ be its revealed
number of background neighbors, excluding $u,v$ and itself. With
$d_j$ denoting the actual degree inside the cell, retain the channels
\begin{equation}\label{alow:within-channel}
 D_l=(d_{j_l^+}-d_{j_l^-})-(O_l^+-O_l^-)-(Q_l^+-Q_l^-)
     =Z_l+T_l.
\end{equation}
Set $I=\{l:W_l\text{ is odd}\}$, $N=|I|$, and
$D=\#\{h:U_h=1\}$. At odd positions $t_l=2-W_l$ is now known,
whereas $\varepsilon_l=\Delta_l$ is fair. Keep only the combined sum
$S$ in \eqref{alow:combined-sum}, with surplus differences
$E_h=A_{uh}-A_{vh}$. The within-cell oracle is
$\mathcal O_w=(H_w,S,(D_l)_{l=1}^{n_w})$.

All degrees reconstruct explicitly. Put
\[
 W_*^{w}=\sum_lW_l+\sum_hU_h,
 \qquad \Delta_*^{w}=S+\sum_{l\notin I}\Delta_l.
\]
Then
\begin{equation}\label{alow:within-reconstruction}
 \begin{aligned}
 d_u&=E_0+(W_*^{w}+\Delta_*^{w})/2,&
 d_v&=E_0+(W_*^{w}-\Delta_*^{w})/2,\\
 B_l^{w}&=O_l^++O_l^-+I_l+2E_l+W_l+Q_l^++Q_l^-,\\
 J_l^{w}&=O_l^+-O_l^-+D_l+Q_l^+-Q_l^-,\\
 d_{j_l^+}&=(B_l^{w}+J_l^{w})/2,&
 d_{j_l^-}&=(B_l^{w}-J_l^{w})/2,\\
 d_h&=Q_h+U_h &&\text{for every surplus node }h.
 \end{aligned}
\end{equation}
The internal edge $E_l$ contributes twice to the degree sum and cancels
from the difference; $E_0$ contributes once to each focal degree. These
formulas again use no orientation and no individual surplus difference.

At odd receivers define $q_l,\delta_l,\nu_l$ from the conditional law
of $Z_l$. Since $2d_l^-+1\ge n_w-1$, Lemma~\ref{low:scalar} and
\eqref{alow:within-flags} give
\[
 \E\nu_l\le
 \frac{\lambda^2}{(1-\lambda)\pi(2d_l^-+1)}
 \le\widetilde z_w:=\frac{\lambda^2}{(1-\lambda)\pi(n_w-1)}
 \le C_\tau/m
\]
for all sufficiently large $m$, uniformly under $d^2/m\le B$.
The receiver flag collections are independent and independent of the
focal reveals. Construct $Q_w$ as above, with the true $H_w$ marginal,
the fair-sum reference for $S$, and the common channel laws. Identity
\eqref{alow:chi-identity} holds with $\mathcal M$ absent and
$a_A=\prod_{l\in A}t_l$, so $a_A^2=1$.

Here $N\sim\operatorname{Bin}(n_w,s_0)$ and $n_w$ is of order $m$.
The good-mask, coefficient, large-order, and rare-mask estimates
\eqref{alow:good-mask}--\eqref{alow:rare-mask} therefore apply with
$z=\widetilde z_w$. Only the $k=1,j=0$ term changes: its flag-averaged contribution
is at most $\widetilde z_w$, since $N/(N+D)\le1$. The $k=0,j\ge1$ contribution
is $O_{B,\tau}(d^2/m)$, the $k=1,j\ge1$ contribution is
$O_{B,\tau}(d^2/m^2)$, and all $k\ge2$ contributions are
$O_{B,\tau}(m^{-2})$, apart from the negligible tails. Hence
\begin{equation}\label{alow:within-chi}
 \chi^2(P_\omega^{\mathcal O_w}\Vert Q_w)
 \le C_{B,\tau}\frac{1+d^2}{m}.
\end{equation}
Restricting to $A\ne\varnothing$ gives the $O_{B,\tau}(m^{-1})$
squared $L^2$ approximation in \eqref{alow:scalar-approximation}.
The common scalar mixtures are again $N$ fair signs plus $D$ biased
signs. The central-lattice proof of \eqref{alow:positive-overlap}
applies without change. Combining that overlap with
\eqref{alow:within-chi} for small $(1+d^2)/m$, and using
\eqref{alow:within-reconstruction}, proves
\begin{equation}\label{alow:within-affinity}
 \aff(P_+^{\mathrm{within}},P_-^{\mathrm{within}})
 \ge\exp\left\{-C_{B,\tau}\frac{1+\beta^2}{m}\right\}.
\end{equation}
For the finitely many smaller $m$, the affinity of the full within-cell
adjacency laws is $(1-w^2)^{m-2}$: exactly $2(m-2)$ edges change their
Bernoulli probability. This is a positive lower bound for the degree-law
affinity. Since $(1+\beta^2)/m\ge1/m$, enlarging $C_{B,\tau}$ extends
\eqref{alow:within-affinity} to every legal swap, including cases with
no paired background nodes.

\subsection{Combining the independent blocks}

\begin{proof}[Proof of Proposition~\ref{prop:composition-affinity}]
At a fixed sign vector, the complete array $Y$ is a rearrangement of
the within-cell degree vectors and the cross-cell row-and-column
margin vectors. These blocks use disjoint unordered edges and are
independent. A swap in cell $g$ changes only its within-cell block and
the $K-1$ cross blocks incident to it. For cross block $(g,k)$, use
\eqref{alow:cross-affinity} with column composition $b_k$, transposing
the block when needed. Each $b_k^2/m\le\Lambda\le B$, so the same
constant applies to every block. Affinities multiply under independent
products, and unaffected blocks contribute one. Equations
\eqref{alow:cross-affinity} and \eqref{alow:within-affinity} give
\[
 \begin{split}
 \aff(P_+^Y,P_-^Y)
 &\ge\exp\left\{-C_{B,\tau}\left[
       \frac{1+b_g^2}{m}
       +\sum_{k\ne g}\left(\frac{b_k^2}{m}+\frac1{m^2}\right)
                  \right]\right\}\\
 &\ge\exp\left\{-C_{B,\tau}
                  \left(\Lambda+\rho+\frac1m\right)\right\}.
 \end{split}
\]
For $K=1$ the cross-block product is empty and equals one. The finite-size
extensions of both local comparisons make the same argument valid for
all $m\in4\mathbb N$, proving the proposition.
\end{proof}

\section{Separation of count laws}\label{sec:recovery}

We prove \eqref{eq:all-candidate} for every pair of candidate sign
patterns, then combine the block affinities to compare candidate
likelihoods. Section~\ref{sec:exponential-rates} turns these comparisons
into exponential risk bounds.

\subsection{Separation of all candidate blocks}

\begin{proof}[Proof of \eqref{eq:all-candidate} in Proposition~\ref{prop:block}]
Write $D=S-S'$ and $N_D=\sum_{ij}D_{ij}^2=4\ell$.
The assertion is immediate if $N_D=0$. Throughout the proof,
$v=1/4-\kappa^2$. The signal restriction implies
$0<\kappa\le1/4$, and hence $3/16\le v\le1/4$.

We use the following moment comparison. If a statistic $F$ has means
$\mu_P,\mu_Q$ under two laws $P,Q$, put
$\Delta=\mu_P-\mu_Q$ and $c=(\mu_P+\mu_Q)/2$. Writing the mean
difference as
\[
 \Delta=\sum_y(\sqrt{p(y)}-\sqrt{q(y)})
                  (\sqrt{p(y)}+\sqrt{q(y)})(F(y)-c)
\]
and applying Cauchy--Schwarz gives
\[
 \Delta^2\le h^2(P,Q)
       \{2\Var_P(F)+2\Var_Q(F)+\Delta^2\}.
\]
Here we used $(\sqrt p+\sqrt q)^2\le2(p+q)$.
This proves \eqref{eq:moment-hellinger} directly for the joint margin
laws.

If any row or column mean differs between $S$ and $S'$, its absolute
difference is at least $2\kappa$, since a nonzero difference of sums
of signs is an even integer. The corresponding observed margin has
variance $mv\le m/4$ under each law. The moment comparison therefore
gives
\[
 h^2(Q_S,Q_{S'})\ge\frac{4\kappa^2}{m+4\kappa^2}
       \ge\frac{2\kappa^2}{m}
       \ge\frac{\kappa^2\ell}{100m^3},
\]
where $\ell\le m^2$ was used in the last inequality.

It remains to consider equal row and column means. In this case
$D\one=D^\top\one=0$. Center the observed row and column margins
at these common means and denote them by $U_i,V_j$. Thus
\[
 F_D=\sum_{ij}D_{ij}(U_i^2-mv)V_j
\]
is the same function of the observations under both laws. Under a
candidate pattern $H\in\{S,S'\}$, let
$Z_{ij}=B_{ij}-(1/2+\kappa H_{ij})$. Independence and centering give
\[
 \E_H(U_i^2V_j)=\E_H Z_{ij}^3=-2\kappa vH_{ij}.
\]
Indeed, every other term in the expansion contains an independent
centered edge to the first power. Since $\E_HV_j=0$, it follows that
\begin{equation}\label{up:mean-gap}
 \E_SF_D-\E_{S'}F_D=-2\kappa vN_D.
\end{equation}

We next bound the variance under either $H$. The centered Bernoulli
variables satisfy
\begin{equation}\label{up:bernoulli-reduction}
 Z_{ij}^2=v+q_{ij}Z_{ij},\qquad
 q_{ij}=-2\kappa H_{ij},\qquad q_{ij}^2=4\kappa^2\le1.
\end{equation}
Products indexed by distinct nonempty edge subsets are orthogonal in
$L^2$: in the product of two different such monomials, some independent
centered edge occurs exactly once. A monomial on $d$ distinct edges has
squared norm $v^d$.

Reduce $F_D$ to square-free monomials by using
\eqref{up:bernoulli-reduction} and
\begin{equation}\label{up:row-square}
 U_i^2-mv=\sum_kq_{ik}Z_{ik}
                   +2\sum_{k<l}Z_{ik}Z_{il}.
\end{equation}
Let $c_d$ be the coefficient vector of its degree-$d$ monomials,
for $d=1,2,3$; a constant term does not affect the variance.

For degree one, the repeated edge in the first term of
\eqref{up:row-square}, multiplied by $V_j$, contributes
$4\kappa^2D_{ij}Z_{ij}$. Repeated edges in the second term contribute
$2v\sum_{l\ne j}D_{il}Z_{ij}=-2vD_{ij}Z_{ij}$.
Consequently
\begin{equation}\label{up:degree-one}
 \|c_1\|_2^2=(4\kappa^2-2v)^2N_D\le N_D.
\end{equation}

For degree two, first consider the product of the linear term in
\eqref{up:row-square} and $V_j$. For distinct edges
$e=(i,k)$ and $f=(h,j)$, its ordered coefficient is
$D_{ij}q_{ik}$. The sum of its squared ordered coefficients is at
most $m^2N_D$. Combining the two orders of each edge pair increases
this bound by at most a factor of two. The cubic term in
\eqref{up:row-square} also produces degree-two monomials when the
edge from $V_j$ repeats one of its two row edges. For the unordered
pair $\{(i,k),(i,l)\}$ their coefficient is
$2(D_{ik}q_{ik}+D_{il}q_{il})$. Its squared coefficient sum is at
most $8(m-1)N_D$. Applying $(x+y)^2\le2x^2+2y^2$ to these two
contributions yields
\begin{equation}\label{up:degree-two}
 \|c_2\|_2^2\le\{4m^2+16(m-1)\}N_D\le20m^2N_D.
\end{equation}

For degree three, two edges in row $i$ and a third edge in a
different row $h$ have coefficient $2D_{ij}$, where $j$ is the
column of the third edge. Summing the squares over $k<l$, $h\ne i$
and $j$ gives $2m(m-1)^2N_D$. Three distinct edges in the same row,
with columns $j,k,l$, have coefficient
$2(D_{ij}+D_{ik}+D_{il})$. By the squared triangle inequality and
the fact that each column belongs to $\binom{m-1}{2}$ triples,
their squared coefficient sum is at most
$6(m-1)(m-2)N_D$. Therefore
\begin{equation}\label{up:degree-three}
 \|c_3\|_2^2\le8m^3N_D.
\end{equation}

Orthogonality and \eqref{up:degree-one}--\eqref{up:degree-three}
now imply, under both candidates,
\begin{equation}\label{up:cubic-variance}
 \Var_H(F_D)
 \le\{v+20v^2m^2+8v^3m^3\}N_D
 \le2m^3N_D.
\end{equation}
Substituting \eqref{up:mean-gap} and \eqref{up:cubic-variance}
into the moment comparison gives
\[
 h^2(Q_S,Q_{S'})
 \ge\frac{\kappa^2v^2N_D}{2m^3+\kappa^2v^2N_D}
 \ge\frac{\kappa^2v^2N_D}{3m^3}
 \ge\frac{3\kappa^2\ell}{64m^3}
 \ge\frac{\kappa^2\ell}{100m^3}.
\]
For the second inequality, use $N_D\le4m^2$ and
$\kappa^2v^2\le1/256$; the third uses $v\ge3/16$.
This proves \eqref{eq:all-candidate}. An opposite-pair swap changes
$2m$ block entries, so it also proves the lower estimate in
\eqref{eq:swap}.
\end{proof}

\subsection{Cross-group likelihood comparisons}

For each globally balanced sign vector $t$ with $t_1=1$, let
$\mathcal L_t$ be the probability mass of the observed cross-group
margins under the zero-activity model with signs $t$. Explicitly,
\begin{equation}\label{up:likelihood}
 \mathcal L_t=\prod_{a<b}q_{t_{G_a}t_{G_b}^\top}
                (Y_{G_a,b},Y_{G_b,a}),
\end{equation}
where $q_H$ is the joint row-and-column margin mass of a block with
pattern $H$. This is a finite, observation-based likelihood, even
though evaluating it may be expensive. For $K\ge4$, let $\widehat s^L$ maximize
\eqref{up:likelihood}, breaking ties by a fixed rule, and put
$\widehat P^L=P_{0,\widehat s^L}$. For $K<4$, use a fixed globally
balanced sign vector instead. The constraint
$t_1=1$ chooses one representative of each global sign reversal.
No balance within a trait group is imposed.

For candidates $s,t$, orient $t$ so its Hamming distance from $s$ is
$h=d_\pm(s,t)\le r/2$. The unordered products $s_is_j$ and
$t_it_j$ differ exactly when one endpoint belongs to the set of
$h$ changed signs and the other does not. There are $h(r-h)$ such
edges. At most $h(m-1)$ lie inside trait groups. If $r\ge4m$,
the number of changed cross-group edges is consequently at least
\begin{equation}\label{up:changed-edges}
 h(r-h)-h(m-1)\ge hr/4.
\end{equation}

Under $\Prob_{0,s}$, Markov's inequality for the square root of the
likelihood ratio gives
\begin{align}
 \Prob_{0,s}(\mathcal L_t\ge\mathcal L_s)
 &\le\aff(P_{0,s}^{\rm cross},P_{0,t}^{\rm cross})\notag\\
 &=\prod_{a<b}\aff(Q_{s_{G_a}s_{G_b}^\top},
                         Q_{t_{G_a}t_{G_b}^\top})\notag\\
 &\le\exp\left\{-\frac{\kappa^2hr}{800m^3}\right\}.
 \label{up:likelihood-comparison}
\end{align}
The equality uses independence of the unordered edge blocks. For
the last inequality, sum \eqref{eq:all-candidate} over the changed
block entries, use \eqref{up:changed-edges}, and use
$\aff=1-h^2/2\le\exp(-h^2/2)$.

\section{Exponential risk bounds}\label{sec:exponential-rates}

We prove Corollary~\ref{cor:exponential-rates}. Write $q_s$ for the
mass function of the entire array $Y$ at $(0,s)$, and $b_s=T^\top s$.
Let $\widehat s^{\rm full}$ maximize $q_t(Y)$ over the globally balanced
signs with $t_1=1$, with a fixed rule for ties, and put
$\widehat P^{\rm full}=P_{0,\widehat s^{\rm full}}$.
Both $q_t$ and $P_{0,t}$ are invariant under global sign reversal.
The search uses $Y,T,\tau$ and includes all feasible compositions.

\subsection{Affinity of two full count laws}

Fix two candidates $s,t$ at distance $h=d_\pm(s,t)\le r/2$.
By \eqref{aup:centered-counts}, the difference of their mean arrays is
\[
 D=\mu_s-\mu_t=\kappa(sb_s^\top-tb_t^\top),\qquad
 \mu_u=\E_{0,u}Y.
\]
Projection of the columns of $D$ onto $t^\perp$ gives
\begin{equation}\label{rate:mean-separation}
 \Delta^2:=\|D\|_F^2
 \ge\kappa^2\left(r-\frac{(s^\top t)^2}{r}\right)\|b_s\|_2^2
 =4\kappa^2h(1-h/r)m\Lambda(b_s)
 \ge2\kappa^2hm\Lambda(b_s).
\end{equation}
Set $Z=\langle D,Y-(\mu_s+\mu_t)/2\rangle_F$. Its means under $s$
and $t$ are $\Delta^2/2$ and $-\Delta^2/2$, respectively. Let $g(i)$
be the trait group of node $i$. Since $Y=AT$, edge $\{i,j\}$ has
coefficient $a_{ij}=D_{i,g(j)}+D_{j,g(i)}$ in $Z$. In particular,
\[
 \sum_{i<j}a_{ij}^2
 \le 2\sum_{i\ne j}D_{i,g(j)}^2\le 2m\Delta^2 .
\]
Under either candidate, $Z$ is a constant plus a sum of independent
Bernoulli variables with these coefficients. Hoeffding's lemma thus
bounds each of $\E_{0,s}e^{-Z/m}$ and $\E_{0,t}e^{Z/m}$ by
$\exp\{-\Delta^2/(2m)+2m\Delta^2/(8m^2)\}$. Multiplying the two
densities by the reciprocal weights $e^{-Z/m}$ and $e^{Z/m}$ inside
the affinity and applying Cauchy--Schwarz yields
\begin{equation}\label{rate:mean-affinity}
 \aff(q_s,q_t)\le
 \bigl(\E_{0,s}e^{-Z/m}\,\E_{0,t}e^{Z/m}\bigr)^{1/2}
 \le e^{-\Delta^2/(4m)}
 \le e^{-\kappa^2 h\Lambda(b_s)/2}.
\end{equation}

For $K\ge4$, projecting onto the cross-group margins can only increase
affinity. Their affinity in \eqref{up:likelihood-comparison} is at most
$\exp(-\kappa^2 h\rho/800)$. Taking the geometric mean of this bound
and \eqref{rate:mean-affinity} gives
\begin{equation}\label{rate:full-affinity}
 \aff(q_s,q_t)
 \le \exp\!\left\{-\kappa^2 h
       \left(\frac{\Lambda(b_s)}4+\frac{\rho}{1600}\right)\right\}.
\end{equation}
The mean bound \eqref{rate:mean-affinity} holds also for $K<4$.

\subsection{Likelihood comparisons at each Hamming distance}

Consider either likelihood, for all of $Y$ or for its cross-group
margins. Suppose its affinity with a candidate at distance $h$ from
$s$ is bounded by $e^{-hu}$, where $u\ge0$ can depend on $s$. A
candidate reaches or exceeds the true likelihood with probability
at most this affinity under $(0,s)$, by Markov's inequality for the
square root of the likelihood ratio. There are at most $2\binom rh$
candidates at distance $h$, allowing either orientation. Consequently
the distance $d$ of the likelihood maximizer from $s$ satisfies
\begin{equation}\label{rate:shell}
 \Prob_{0,s}(d=h)\le 2\binom rh e^{-hu}
 \le 2(x/h)^h,\qquad x=er e^{-u},\quad 1\le h\le r/2.
\end{equation}
In $\E_{0,s}d$, distances at most $2x$ contribute at most $2x$.
For $h>2x$, use $h(x/h)^h\le x2^{-(h-1)}$ and sum
\eqref{rate:shell} to obtain $\E_{0,s}(d/r)\le6e\,e^{-u}$.

To transfer this exponent to nonzero activity, note that two globally
balanced signs differ at equally many positive and negative entries
of the first sign. Their Hamming distance and its complement are even.
Thus $d>0$ implies $d\in\{2,4,\ldots\}$. By \eqref{up:change-measure},
\[
 \Prob_{\alpha,s}(d=h)\le \sqrt{2}\,e^{M^2/2}(x/h)^{h/2}.
\]
For $h>2x$ and $h\ge2$,
$h(x/h)^{h/2}=x(x/h)^{(h-2)/2}\le x2^{-(h-2)/2}$.
Summing over the positive even distances and again bounding the
contribution from $h\le2x$ by $2x$ gives
\begin{equation}\label{rate:sign}
 \E_{\alpha,s}(d/r)\le e(2+2\sqrt{2}\,e^{M^2/2})e^{-u}.
\end{equation}

For the cross-group likelihood, \eqref{up:likelihood-comparison} gives
$u=\kappa^2\rho/800$ when $K\ge4$. For $K<4$ we have
$\rho\le3/16$, so $d/r\le1/2$ gives the same bound up to an absolute
factor for the fixed balanced sign specified in Section~\ref{sec:recovery}.
For the full likelihood, \eqref{rate:full-affinity} allows
$u=\kappa^2(\Lambda(b_s)+\rho)/1600$ when $K\ge4$. For $K<4$ use
$u=\kappa^2\Lambda(b_s)/2$ in \eqref{rate:sign}. Since $\rho\le3/16$
and $\kappa<1/2$, the missing factor $\exp(\kappa^2\rho/1600)$ is
bounded by an absolute constant. Applying \eqref{aup:sign-to-matrix}
to the two estimators therefore yields, with $C_M$ depending only on $M$,
\begin{align}
 \E_{\alpha,s}L_r(\widehat P^L,P_{\alpha,s})
 &\le C_M\kappa^2 e^{-\kappa^2\rho/800}+\frac{M^2}{r(r-1)},\label{rate:cross-risk}\\
 \E_{\alpha,s}L_r(\widehat P^{\rm full},P_{\alpha,s})
 &\le C_M\kappa^2 e^{-\kappa^2(\Lambda(b_s)+\rho)/1600}
       +\frac{M^2}{r(r-1)} .\label{rate:full-risk}
\end{align}

\subsection{Lower bounds on the balanced composition}

Write $R_r(M)$ for the minimax risk from $Y$ on $\Theta_r(M;0)$,
including the zero-activity case $M=0$. The hypercube used in
\eqref{low:global-risk} has $\alpha=0$ and $b=0$ at every corner.
Its lower bound therefore applies to $R_r(0)$ even with $b$ given
to the estimator. We bound its two affinity factors uniformly in $m$.
The parameters in \eqref{low:parameters} satisfy
$\lambda\le1/16$, $\pi\ge1/2$, $p_f\ge1/20$ and $s_0\ge2/5$.
For $m\ge4$, \eqref{low:constants} gives $z_c\le1/360$ and
\[
 D_c(m,w)\le \frac{20}{27}+\frac{4}{360\cdot359}<1.
\]
The same formulas give $D_c(m,w)\le C\lambda^2/m^2$ and
$z_w\le C\lambda^2/m$, with $z_w\le1/120$. Hence both subtractions
in the factors $\beta_c=1-D_c/2$ and $\beta_w=1-2z_w/(1-z_w)$
in \eqref{low:affinity-constants} are at most $1/2$.
Using $-\log(1-v)\le2v$ on this interval and
$\lambda=4\kappa/(1+4\kappa^2)\le4\kappa$ in \eqref{low:global-risk}
yields, for an absolute $C$,
\begin{equation}\label{rate:balanced-lower}
 R_r(M)\ge R_r(0)\ge\frac{\kappa^2}{4}
       \exp\{-C\kappa^2(\rho+1/m)\}.
\end{equation}

For the activity contribution, fix any feasible $s$, give it to the
estimator, and observe all of $A$. Compare $\alpha=0$ with
$\alpha=(\delta/r)\one_r$, $\delta=\min\{M,1\}$. The latter has energy
$\binom r2\delta^2/r^2\le M^2$. Each edge logit increases by
$v=2\delta/r$, so the two probabilities differ by at least
$v\ell_\tau$, where $\ell_\tau=\min_{|z|\le\tau+1}\sigma'(z)>0$.
For a Bernoulli logit $\theta$, the divergence from $\theta$ to
$\theta+v$ is $f(\theta+v)-f(\theta)-vf'(\theta)\le v^2/8$,
where $f(z)=\log(1+e^z)$. By independence the full-adjacency
divergence is at most $\binom r2 v^2/8\le\delta^2/4$.
Pinsker's inequality bounds total variation by $\delta/\sqrt8<1/2$.
The squared separation in $L_r$ is at least $v^2\ell_\tau^2$, so
\eqref{low:two-point-risk} gives the risk lower bound
\begin{equation}\label{rate:activity-floor}
 \frac{\ell_\tau^2\min\{M^2,1\}}{2r^2}.
\end{equation}
Every estimator based on $Y$ is also available from $A$. Thus this floor
holds at every feasible composition for either observation, including
the balanced composition defining $R_r(M)$.

Finally, fix $\tau>0$ in the stated signal range and let $\rho\to\infty$.
At $b=0$ and $M=0$, the full-likelihood upper bound and
\eqref{rate:balanced-lower} give $-\log R_r(0)=\Theta_\tau(\rho)$.
For fixed $M>0$, the risk is bounded below by the larger of
\eqref{rate:balanced-lower} and \eqref{rate:activity-floor}, and
above by $C_M\kappa^2e^{-\kappa^2\rho/1600}+M^2/[r(r-1)]$.
Since $r\to\infty$, taking logarithms gives
$-\log R_r(M)=\Theta_{\tau,M}(\min\{\rho,\log r\})$.

\section{Recovery criteria and the adjacency comparison}\label{sec:alignment-upper}

The full-likelihood bound and the composition-dependent swap comparison
now yield the consistency criterion. We then compare with full adjacency
and take the worst case over compositions.

\subsection{The composition-wise consistency criterion}

\begin{proof}[Proof of Theorem~\ref{thm:composition-boundary}]
For every feasible deterministic composition $b$, the full-likelihood
estimator in Corollary~\ref{cor:exponential-rates} gives
\begin{equation}\label{aup:composition-upper}
 \risk_{r,b}^Y(T)
 \le C_M\kappa^2e^{-\kappa^2(\Lambda(b)+\rho)/1600}
       +\frac{M^2}{r(r-1)}.
\end{equation}
This estimator uses $Y,T,\tau$ and ignores $b$, so it is admissible when
the infimum defining $\risk_{r,b}^Y(T)$ permits knowledge of $b$.
If $\Lambda(b)+\rho\to\infty$, the displayed bound tends to zero.

For the lower bound, fix $B\ge0$ and consider any feasible composition
with $\Lambda(b)\le B$ and $m\ge4B$. Then
\[
 \max_k|b_k|\le\sqrt{Bm}\le m/2.
\]
In each cell reserve $m/2$ nodes as
anchors, assigning $m/4+b_k/2$ of them sign $+1$ and $m/4-b_k/2$ sign
$-1$. These are nonnegative integers because $m\in4\N$ and $b$ is
feasible. Partition the remaining $m/2$ nodes into $m/4$ pairs and assign
each pair either $(+1,-1)$ or $(-1,+1)$. Set $\alpha=0$. All resulting
parameters have composition $b$, and the orientations form a hypercube
with $q=r/4$ bits.

For each orientation bit $j$, let $\mathcal E_j$ contain the ordered
dyads between its two nodes and all $r/2$ anchors, in both directions.
These coordinate sets are disjoint for distinct bits. Each contains
$2r$ coordinates, and flipping the bit changes every corresponding
probability by magnitude $2\kappa$. The squared separation on
$\mathcal E_j$ is therefore
\begin{equation}\label{aup:anchor-separation}
 D_j^2=8r\kappa^2.
\end{equation}
The neighboring sign vectors differ by an opposite-sign swap within
one cell and preserve $b$. Proposition~\ref{prop:composition-affinity}
therefore bounds the affinity of their full observed laws below by
\begin{equation}\label{aup:cube-affinity}
 A_r=\exp\{-C_{B,\tau}(\Lambda(b)+\rho+1/m)\}.
\end{equation}

Give this hypercube the uniform prior. For an arbitrary estimator,
condition on all bits except $j$ and restrict its squared loss to
$\mathcal E_j$. The two-point inequality \eqref{low:two-point-risk},
the affinity bound \eqref{low:tv-affinity}, and
\eqref{aup:anchor-separation} show that its average unnormalized loss
on these coordinates is at least
\[
 2r\kappa^2\{1-\sqrt{1-A_r^2}\}.
\]
Summing over the $r/4$ disjoint coordinate sets and normalizing yields
\begin{equation}\label{aup:composition-lower}
 \risk_{r,b}^Y(T)
 \ge\frac{r\kappa^2}{2(r-1)}\{1-\sqrt{1-A_r^2}\}
 \ge\frac{r\kappa^2}{4(r-1)}
       \exp\{-2C_{B,\tau}(\Lambda(b)+\rho+1/m)\}.
\end{equation}
Here the minimax risk dominates the Bayes risk on the displayed
zero-activity subfamily, and $1-\sqrt{1-x}\ge x/2$ for $0\le x\le1$.
The argument applies to randomized estimators as well. Knowledge of
$b$ does not change the comparison, since every corner has the same
composition. Thus \eqref{aup:composition-lower} holds uniformly over
all such compositions and group sizes.

On any sequence with $\Lambda(b)+\rho$ bounded and $r=\rho m^3\to\infty$,
necessarily $m\to\infty$. The preceding construction therefore applies
eventually, and \eqref{aup:composition-lower} stays bounded away from zero.

Finally, if $\Lambda(b)+\rho$ does not tend to infinity, it has an
infinite subsequence bounded by some finite $B$. The preceding lower
bound prevents the risk from tending to zero on that subsequence,
and hence on the full sequence. Together with
\eqref{aup:composition-upper}, this proves the stated equivalence
for arbitrary sequences of feasible compositions.
\end{proof}

\subsection{Recovery from adjacency}\label{sec:adjacency}

We construct an estimator with the following bound over the same activity
class whenever $r-1\ge M^2/\kappa^2$:
\begin{equation}\label{eq:adjacency-bound}
 \sup_{\Theta_r(M)}\E_{\alpha,s}L_r(\widetilde P,P_{\alpha,s})
 \le\frac{M^2}{2r(r-1)}+2\big\{(1+e^{-\kappa^2r/4})^r-1\big\}.
\end{equation}
For a balanced sign vector $t$, define
\[
 Q_A(t)=\sum_{i<j}(A_{ij}-1/2)t_it_j.
\]
Let $\widetilde s$ be a maximizer over balanced vectors with
$t_1=1$, using a fixed rule for ties. One of $s,-s$ is feasible,
and both have the same score. Set
$\widetilde p_{ij}=1/2+\kappa\widetilde s_i\widetilde s_j$.

Write
\[
 p_{ij}=1/2+\kappa s_is_j+\varepsilon_{ij}.
\]
Lipschitz continuity and the energy constraint give
\begin{equation}\label{up:adjacency-perturbation}
 \sum_{i<j}\varepsilon_{ij}^2
 \le\frac1{16}\sum_{i<j}(\alpha_i+\alpha_j)^2
 \le M^2/4.
\end{equation}
For a competitor $t\notin\{s,-s\}$, let
$h=d_\pm(s,t)$, so $1\le h\le r/2$. The set $D$ of changed
unordered edge products has size $d=h(r-h)\ge r-1$, and
\begin{equation}\label{up:adjacency-score-gap}
 Q_A(s)-Q_A(t)
 =2\sum_{\{i,j\}\in D}s_is_j(A_{ij}-1/2).
\end{equation}
By \eqref{up:adjacency-perturbation} and Cauchy--Schwarz,
\[
 \E_{\alpha,s}\{Q_A(s)-Q_A(t)\}
 \ge2\kappa d-M\sqrt d\ge\kappa d,
\]
provided $r-1\ge M^2/\kappa^2$. The summands in
\eqref{up:adjacency-score-gap} are independent and each has range
length two. Hoeffding's inequality
\citep{hoeffding1963probability} therefore gives
\begin{equation}\label{up:adjacency-competitor}
 \Prob_{\alpha,s}\{Q_A(t)\ge Q_A(s)\}
 \le\exp(-\kappa^2d/2).
\end{equation}
There are at most $2\binom rh$ vectors at distance $h$ from
$\{s,-s\}$; this also bounds the number satisfying the balance
and gauge constraints. A union bound including score ties gives
\begin{align}
 \Prob_{\alpha,s}
       \{\widetilde s\widetilde s^\top\ne ss^\top\}
 &\le2\sum_{h=1}^{r/2}\binom rh
                       e^{-\kappa^2h(r-h)/2}\notag\\
 &\le2\sum_{h=1}^{r}\binom rh e^{-\kappa^2hr/4}
 =2\{(1+e^{-\kappa^2r/4})^r-1\}.
 \label{up:adjacency-error}
\end{align}

On exact sign-product recovery, \eqref{up:adjacency-perturbation}
implies
\[
 L_r(\widetilde P,P_{\alpha,s})
 =\frac{2}{r(r-1)}\sum_{i<j}\varepsilon_{ij}^2
 \le\frac{M^2}{2r(r-1)}.
\]
The loss is at most one on failure. Equation~\eqref{up:adjacency-error}
therefore proves \eqref{eq:adjacency-bound}. Finally,
\[
 0\le(1+e^{-\kappa^2r/4})^r-1
 \le\exp\{r e^{-\kappa^2r/4}\}-1\longrightarrow0.
\]
Thus full adjacency permits uniform probability-matrix consistency
along every admissible design sequence.

\subsection{Worst-case resolution}

\begin{theorem}[Worst-case resolution]\label{thm:boundary}
For every sequence of equal-cell partitions with $r=Km\to\infty$
and $m\in4\N$,
\begin{equation}\label{eq:boundary}
 \risk_r^Y(T)\to0\quad\Longleftrightarrow\quad \rho\to\infty .
\end{equation}
If $\rho$ stays bounded, $\liminf_r\risk_r^Y(T)>0$.
The full-adjacency risk $\risk_r^A(T)$ tends to zero on every sequence.
\end{theorem}

\begin{proof}
Theorem~\ref{thm:composition-boundary} gives a positive lower risk on
any bounded-$\rho$ subsequence when $b=0$, even if the composition is
given to the estimator. The full-likelihood bound \eqref{rate:full-risk}
is uniform over all compositions and tends to zero when $\rho\to\infty$.
Equation~\eqref{eq:adjacency-bound} gives the adjacency assertion.
\end{proof}

\section{Polynomial-time estimation from count fluctuations}\label{sec:polynomial-rates}

\paragraph{Proof of Theorem~\ref{thm:poly-recovery}.}
We first establish the cubic initialization bound, then amplify independent
panel votes at exact balance. For general compositions we combine the
Fourier panel readout with held-out mean scores. The final step transfers
the comparison errors to probability-matrix loss and records the
computational and numerical-precision bounds.

\subsection{A spectral estimator from cubic margins}
\label{sec:cubic-initialization}

For $g<h$, let $U=Y_{G_g,h}-(m/2)\one_m$ and
$V=Y_{G_h,g}-(m/2)\one_m$. Write $H=I_m-\one_m\one_m^\top/m$.
With entrywise squares, form the symmetric matrix $C$ with blocks
\begin{equation}\label{poly:cubic-block}
 C_{gh}=-\frac{H\{U^{\circ2}V^\top+U(V^{\circ2})^\top\}H}
                  {4\kappa v},\qquad
 C_{hg}=C_{gh}^\top,\qquad C_{gg}=0,\qquad v=\tfrac14-\kappa^2 .
\end{equation}
Let $\widehat s^C$ be the balanced rounding of a unit eigenvector for
the largest algebraic eigenvalue of $C$, with the tie rule of
Section~\ref{sec:spectral-rounding}.

\paragraph{Population matrix.}
Work under $\Prob_{0,s}$. For $i\in G_g$, $j\in G_h$, the two
margins have respective means $\mu_i=\kappa s_i b_h$ and
$\nu_j=\kappa s_j b_g$. Each has variance $mv$. The margins share
one edge, whose second and third centered moments are $v$ and
$-2\kappa v s_i s_j$. Independence of the remaining edges gives
\[
 \E(U_i^2 V_j)=(mv+\mu_i^2)\nu_j+2v\mu_i-2\kappa v s_i s_j.
\]
Since $\mu_i^2=\kappa^2 b_h^2$, multiplication by $H$ on both sides
removes the first two terms. Interchanging the two margins yields the
same projected sign product. Thus $\E C_{gh}=t_g t_h^\top$, where
$t_g=Hs_{G_g}$. The vector $t$ obtained by stacking the $t_g$ satisfies
$\|s-t\|_2^2=\Lambda$ and $\|t\|_2\le\sqrt r$. The omitted diagonal
blocks of $tt^\top$ have operator norm at most $m$, so
\begin{equation}\label{poly:cubic-population}
 \|\E C-ss^\top\|_{\mathrm{op}}\le m+2\sqrt{r\Lambda}.
\end{equation}

\paragraph{Block moments.}
Put $L_h=1+b_h^2/m$, $F_h=H(U^{\circ2})$ and $X_h=HU$.
The subscripts denote the target group of each margin, so $F_g,X_g$
are formed from $V$. For a scalar random variable $Z$ write
$\|Z\|_{L^p}=(\E|Z|^p)^{1/p}$. For every unit vector $x\in\R^m$,
\begin{equation}\label{poly:directional-moments}
 \begin{aligned}
 \|x^\top F_h\|_{L^4}&\le c m\sqrt{L_h},&
 \big\|\|F_h\|_2\big\|_{L^4}&\le c m^{3/2}\sqrt{L_h},\\
 \|x^\top X_h\|_{L^4}&\le c\sqrt m(1+|b_h|),&
 \big\|\|X_h\|_2\big\|_{L^4}&\le c m\sqrt{L_h}.
 \end{aligned}
\end{equation}
Here and below numerical constants are uniform in \eqref{eq:signal}
unless their dependence is indicated. Indeed, the coordinates of
$\xi=U-\mu$ are independent centered row sums with
$\|\xi_i\|_{L^8}\le c\sqrt m$ by Hoeffding's inequality. The vector
$F_h$ is $H$ applied to $\xi^{\circ2}-mv\one_m+2\mu\circ\xi$,
whose independent centered coordinates have fourth norms at most
$cm\sqrt{L_h}$. Expanding the fourth moment of a weighted sum gives
the first directional bound, since $\|Hx\|_2\le1$. The expansion
applied to $\xi$ and the bound $\|\mu\|_2\le\kappa|b_h|\sqrt m$
give the second. The vector bounds follow from contraction by $H$ and
$\big\|\|z\|_2\big\|_{L^4}^2\le\sum_i\|z_i\|_{L^4}^2$.

Using Cauchy--Schwarz for each outer product in \eqref{poly:cubic-block}
and $v\ge3/16$, we obtain
\begin{align}
 \sup_{\|x\|_2=1}\E\|C_{gh}^\top x\|_2^2
   &\le c\kappa^{-2}m^4 L_g(1+b_h^2),\label{poly:block-variance}\\
 \E\|C_{gh}\|_F^2&\le c\kappa^{-2}m^5 L_g L_h .
 \label{poly:block-frobenius}
\end{align}
The bound for $\E\|C_{gh}x\|_2^2$ interchanges $g,h$. These
estimates permit dependence between the two margins of one block.

\paragraph{Concentration.}
Suppose $K\ge m^2$. With $S=\kappa^{-1}(1+\Lambda)m^2\sqrt K$, we
claim
\begin{equation}\label{poly:cubic-concentration}
 \Prob_{0,s}\{\|C-\E C\|_{\mathrm{op}}>cS\}\le c r^{-8}.
\end{equation}
Embed each $C_{gh}$ and its transpose as one self-adjoint summand.
Different unordered group pairs use disjoint edges. For each pair
truncate its summand to the event that all $2m$ margins differ from
their respective means by at most $c_1\sqrt{m\log(2r)}$. A sufficiently
large numerical $c_1$ bounds each discard probability by $r^{-16}$.
Center the truncated summands and denote their sum by $Z$.

Write $\sigma^2=\|\E Z^2\|_{\mathrm{op}}$, let $w^2$ be the operator norm
of $\Cov(\operatorname{vec}Z)$, and let $R$ bound each centered
summand in operator norm. Put $\ell=\log(2r)$. We have
\[
 \sigma\le cS,\qquad w\le cS\sqrt{m/K},\qquad
 R\le cS\sqrt{m/K}\,\ell^{3/2}.
\]
For the first bound, truncation and centering can only decrease each
summand's matrix second moment in positive-semidefinite order.
The square of an embedded summand is block diagonal, so its contributions
to a fixed diagonal block are bounded in total by
$c\kappa^{-2}m^4 L_g(K+m\Lambda)$ from \eqref{poly:block-variance}.
Here $L_g\le1+\Lambda$ and $m\le K$. For $w$, order entries by
unordered group pair. The covariance is a direct sum, and the norm
of each part is at most $2\E\|C_{gh}\|_F^2$, counting the transpose
in the same summand. Apply \eqref{poly:block-frobenius}. For $R$, the
truncation event gives
$\|X_h\|_2\le cm\sqrt{L_h\ell}$ and
$\|F_h\|_2\le cm^{3/2}\sqrt{L_h}\,\ell$ by the expansion of $F_h$
above. Using both outer products and allowing a factor two for
centering gives the assertion.

Apply \citet[Corollary~2.17]{brailovskaya2024universality} to $Z$
at tail parameter $u=11\ell$. The leading term is $2\sigma$.
The four remainder terms are bounded by a numerical constant times
\[
 \sqrt{w\sigma}\,\ell^{3/4}
    +w\sqrt u+R^{1/3}\sigma^{2/3}u^{2/3}+Ru,
\]
since their weak variance parameter $\sigma_*$ is at most $w$ by
testing the entry covariance on a vectorized unit outer product.
Each term is $O(S)$, as $\sqrt{m/K}\le K^{-1/4}$ and $r\le K^{3/2}$.
The failure probability is at most $4re^{-u}$. With probability at
least $1-K^2r^{-16}$ no block is discarded. The deterministic bound
$\|C_{gh}\|_{\mathrm{op}}\le c\kappa^{-1}m^4$ puts the difference of
the untruncated and truncated expectations below
$c\kappa^{-1}K^2m^4r^{-16}$. This proves
\eqref{poly:cubic-concentration} for the observed $C$.

\paragraph{Rounding.}
Let $e=\|C-ss^\top\|_{\mathrm{op}}$. If $e<r/2$, the leading eigenvalue
of $C$ is at least $r-e$. Projecting the corresponding eigenvector
equation onto $s^\perp$ bounds the sine of its angle with $s$ by
$e/(r-e)$. Balanced rounding, as in \eqref{aup:spectral-sign-error},
then gives $d_\pm(\widehat s^C,s)/r\le c(e/r)^2$.
For $e\ge r/2$ the same inequality follows from $d_\pm/r\le1/2$.
By \eqref{poly:cubic-population} and \eqref{poly:cubic-concentration},
whenever $\rho\ge1$, outside an event of $\Prob_{0,s}$-probability
at most $cr^{-8}$,
\begin{equation}\label{poly:cubic-weak}
 \frac{d_\pm(\widehat s^C,s)}r
 \le c\left\{\frac{(1+\Lambda)^2}{\kappa^2\rho}
                   +\frac{1}{K^2}+\frac{\Lambda}{r}\right\}
 \le \frac{c'(1+\Lambda)^2}{\kappa^2\rho}.
\end{equation}
We used $r=m^3\rho$. In particular both the exceptional probability
and the constant in the balanced bound $c'/(\kappa^2\rho)$ are uniform
in the signal range \eqref{eq:signal}.
For fixed model constants, an operator-norm backward error of $r^{-2}$
in the eigendecomposition is absorbed by these perturbation bounds.

\subsection{Independent votes at exact cell balance}\label{sec:balanced-panels}

Assume first that every cell is balanced and that $\alpha=0$. Fix a
reference node $a$. A panel will vote on whether $i$ and $a$ have the
same sign. All panels use the same reference, so the $r-1$ comparisons
give a single labeling up to global reversal.

Take $q=2\lceil Dm^2/\kappa^2\rceil$ for a sufficiently large numerical
constant $D$. In a fixed order, partition the $K$ cells into
$\lfloor K/q\rfloor$ panels of size $q$, dropping the remainder.
Inside each panel $B$, form \eqref{poly:cubic-block} using only its
internal cross-group margins and obtain a rounded sign vector $z_B$.
Equation~\eqref{poly:cubic-weak} with $K=q$ and $r=qm$ gives
\begin{equation}\label{poly:balanced-overlap}
 \Prob_{0,s}\{|\langle z_B,s_B\rangle|<3qm/4\}\le1/16 .
\end{equation}
Indeed, increasing $D$ makes the good-event fractional error at most
$1/8$ and its exceptional probability at most $1/16$, uniformly in $m$
and the stated signal range. The arbitrary orientation chosen in each
eigendecomposition will cancel in the vote.

For a held-out node $i\in G_g$ use panels omitting $g$. Within one such
panel, let $z_h$ denote the entries of $z_B$ on $G_h$, and set
\begin{equation}\label{poly:balanced-score}
 S_i^B=-\frac{1}{2\kappa v}\sum_{h\in B}
    (U_{ih}^2-mv)\langle z_h,V_{hg}\rangle,\qquad
 U_{ih}=Y_{i,h}-m/2,\quad V_{hg}=Y_{G_h,g}-(m/2)\one_m .
\end{equation}
Conditional on the panel fit, these $q$ summands use independent group
pairs. Within a pair, both margins have zero mean. Their shared edge
gives $\E[(U_{ih}^2-mv)(V_{hg})_j]=-2\kappa v s_i s_j$.
Consequently
\[
 \E(S_i^B\mid z_B)=s_i\langle z_B,s_B\rangle,\qquad
 \Var(S_i^B\mid z_B)\le C\kappa^{-2} q m^4 .
\]
For the variance, the fourth norm of $U_{ih}^2-mv$ is at most $Cm$ by
the eighth-moment bound for a centered sum of $m$ bounded variables.
The linear form in $V_{hg}$ is a signed sum of $m^2$ independent
centered edges, so its fourth norm is at most $Cm$. Cauchy--Schwarz
allows the dependence between these two factors. On the overlap event
in \eqref{poly:balanced-overlap}, Chebyshev's inequality gives conditional
probability at most $C m^2/(\kappa^2 q)\le1/16$ that $S_i^B$ is zero
or has the wrong sign relative to the orientation of $z_B$.

Use panels omitting both $g(i)$ and $g(a)$. Each casts the vote
$\operatorname{sign}(S_i^B S_a^B)$ for $s_i s_a$, assigning $+1$ at
zero. By a union bound each vote is correct with probability at least
$1-3/16$, including when $i$ and $a$ lie in the same trait group.
For this fixed pair the votes are independent: a panel uses only its
internal edges and its edges to the two omitted groups. At most two
panels are lost. Hoeffding's inequality for the majority therefore gives
\begin{equation}\label{poly:balanced-vote}
 \Prob_{0,s}\{\widehat u_i\ne s_a s_i\}
       \le C\exp(-c\kappa^2\rho).
\end{equation}
Here $\widehat u_a=1$, and ties and empty majorities receive a fixed
value. The bound follows by enlarging $C$ when $K<6q$ since then
$\kappa^2\rho$ is bounded. For $K\ge6q$ at least $K/(2q)$ panels
are available. The votes for different $i$ may be dependent.

Apply \eqref{up:change-measure} per comparison, and sum the resulting
probabilities over the nodes. Rounding $\widehat u$ to a nearest
globally balanced sign vector increases its Hamming distance to $s_as$
by at most a factor of two. Equation~\eqref{aup:sign-to-matrix} now
gives \eqref{eq:panel-risk}, with a numerical change in $c$. All panels
together require at most $O(r^3)$ arithmetic operations for their
spectral fits. Their held-out scores and the votes can be computed in
$O(r^2)$ operations by reusing each group projection. The two logarithmic
minimax rates follow by applying the lower bounds
\eqref{rate:balanced-lower} and \eqref{rate:activity-floor} to this
polynomial-time upper bound.

\subsection{Group panels for the resolution contribution}
\label{sec:fourier-panels}

We construct a ballot for each pair of nodes $i,j$ at zero activity.
Its error is exponentially small in $\rho$ whenever
$\Lambda\le\delta_\tau\rho$, for a fixed $\delta_\tau>0$.
Panel labels inferred up to a common reversal suffice to compare the
two nodes. The following readout has a signal bounded below uniformly
in the composition of a target node's group.

\paragraph{A directed block score.}
Fix a target $i\in G_g$ and a different group $h$, with a sign vector
$z_h\in\{-1,1\}^m$ obtained independently of the edges between $g$ and
$h$. Condition on the estimated labels $z_h$, keeping the true assignment $s$
fixed. Use $U_i=Y_{ih}-m/2$, $V=Y_{G_h,g}-(m/2)\one_m$,
$w_h=Hz_h$ and $t_h=Hs_{G_h}$. For $x>0$, set
\begin{gather*}
 \theta=2\arctan x,\quad
 \phi=\arctan(2\kappa x),\quad
 d^2=\frac{1+4\kappa^2x^2}{1+x^2},\\
 a_x=\frac{2vx}{1+4\kappa^2x^2},\quad
 c_x=\frac{4\kappa v x^2}{1+4\kappa^2x^2}.
\end{gather*}
Here $d$ is positive. For $\ell\in G_h$, the edge $i\ell$ with
parameter $1/2+\kappa s_i s_\ell$ has characteristic function
$\cos(\theta/2)+2\mathrm{i}\kappa s_i s_\ell\sin(\theta/2)
     =d\exp(\mathrm{i}s_i s_\ell\phi)$ after subtracting $1/2$.
Consequently
\[
 \E e^{\mathrm{i}\theta V_\ell}=d^m e^{\mathrm{i}s_\ell b_g\phi},
 \qquad
 \E\{U_i e^{\mathrm{i}\theta V_\ell}\}
 =\E e^{\mathrm{i}\theta V_\ell}
       \{\mu_i+\mathrm{i}a_x+c_x s_i s_\ell\},
 \qquad \mu_i=\kappa s_i b_h.
\]
Indeed, the edge shared by the two margins gives
$\E[(A_{i\ell}-P_{i\ell})e^{\mathrm{i}\theta(A_{i\ell}-1/2)}]
=2\mathrm{i}v\sin(\theta/2)$. Dividing by the edge characteristic function gives
$\mathrm{i}a_x+c_x s_i s_\ell$. Define
\[
 F_{ih}(x)=m\operatorname{Re}\left[
 (U_i-\mathrm{i}a_x)\sum_{\ell\in G_h}
          w_{h,\ell}e^{\mathrm{i}\theta V_\ell}\right],\qquad
 J_{gh}(x)=\frac1m\sum_{\ell\in G_h}\cos(\theta V_\ell).
\]
Because $\sum_\ell w_{h,\ell}=0$, the real part of the term
containing $\mu_i$ vanishes after summation. Hence
\begin{align*}
 \E F_{ih}(x)
   &=m d^m c_x\cos(b_g\phi)\,s_i\langle z_h,t_h\rangle,\\
 \E J_{gh}(x)&=d^m\cos(b_g\phi).
\end{align*}

Use $x_1=m^{-1/2}$ and
$x_2=2x_1/(1-4\kappa^2 x_1^2)$. Both lie in
$[m^{-1/2},3m^{-1/2}]$ and $\phi_2=2\phi_1$ by the tangent
double-angle formula. In the stated signal range,
$d_\nu^{2m}\ge(1+9/m)^{-m}\ge e^{-9}$ and
$c\kappa\le m c_{x_\nu}\le C\kappa$ for $\nu=1,2$, with numerical
$c,C>0$. Since $\cos^2 u+\cos^2(2u)\ge7/16$, the quantity
\[
 \gamma_g=\sum_{\nu=1}^2 m d_\nu^{2m} c_{x_\nu}
                         \cos^2(b_g\phi_\nu)
 \quad\text{satisfies}\quad c'\kappa\le\gamma_g\le C'\kappa
\]
uniformly in $g$. For two distinct groups $h,k$ outside $g$, the
two edge blocks $(g,h)$ and $(g,k)$ are independent of each other.
When the labels are obtained independently of these blocks,
the sum over both frequencies of
$F_{ih}(x_\nu)J_{gk}(x_\nu)+F_{ik}(x_\nu)J_{gh}(x_\nu)$ therefore
has conditional mean
$\gamma_g s_i(\langle z_h,t_h\rangle+\langle z_k,t_k\rangle)$.
Using both directions gives the same coefficient for every panel group.

For either frequency, $\E\cos(\theta V_\ell)=d^m\cos(b_g\phi)$
is constant in $\ell$. The $V_\ell$ are independent column margins.
Thus the fourth norm of $\sum_\ell w_{h,\ell}\cos(\theta V_\ell)$
is at most $C\sqrt m$, by the fourth-moment expansion for centered
bounded independent variables and $\|w_h\|_2\le\sqrt m$.
Also $\|U_i\|_{L^4}\le C\sqrt{m L_h}$ with $L_h=1+b_h^2/m$.
The sine term in $F_{ih}(x)/m$ has magnitude at most
$|a_x|\,\|w_h\|_1\le C\sqrt m$.
Cauchy--Schwarz gives
\[
 \E F_{ih}(x)^2\le C m^4 L_h,\qquad |J_{gh}(x)|\le1 .
\]
This bound permits dependence of the row and column margins in $F$
and holds even for a homogeneous target group $g$.

\paragraph{A single panel.}
Set $A=C_0\kappa^{-2}$ for a sufficiently large numerical $C_0$
and $q=2\lceil A m^2/2\rceil$. Partition the first
$q\lfloor K/q\rfloor$ groups into panels of $q$ consecutive groups.
Infer labels on each panel $B$ from its internal blocks by
\eqref{poly:cubic-block}, using coordinate signs of a leading unit
eigenvector (fixed ties). Let $z_B=(z_h)_{h\in B}$ be the concatenated
estimate and $t_B=(t_h)_{h\in B}$. Write $n=qm$ and
$\Lambda_B=\sum_{h\in B}b_h^2/m$. If $\Lambda_B\le1$, then
with probability at least $31/32$ there is an orientation
$\eta\in\{-1,1\}$ for which
\begin{equation}\label{poly:panel-overlap}
 \eta\langle z_B,t_B\rangle\ge n/2.
\end{equation}
Indeed, the population and concentration arguments
in \eqref{poly:cubic-population} and \eqref{poly:cubic-concentration}
use only individual cross-group blocks and require no balance of the
signs within the panel. With probability at least $1-Cn^{-8}$ the
relative perturbation from $s_Bs_B^\top$ is bounded by
$1/q+2/\sqrt n+C/(\kappa\sqrt A)$. Increasing $C_0$ makes this
quantity small enough that the eigenvector has squared Euclidean
distance at most $1/16$ from $\eta s_B/\sqrt n$ for some $\eta$,
by the eigenvector projection argument of
Section~\ref{sec:cubic-initialization}. Each wrongly signed
coordinate contributes at least $1/n$ to that squared distance, so
$\eta\langle z_B,s_B\rangle\ge7n/8$.
Finally $|\langle z_B,s_B-t_B\rangle|\le\sqrt{n\Lambda_B}$.
The same choice of $C_0$ gives \eqref{poly:panel-overlap} and
$Cn^{-8}\le1/32$.

For a target outside $B$, pair consecutive groups within $B$ and
let $S_{iB}$ sum the two-frequency, two-direction score above over
these $q/2$ pairs. Conditional on $z_B$, the pairs
use independent blocks. Their second moments are bounded by
$C m^4(L_h+L_k)$, by the square-of-a-sum inequality. Therefore
\[
 \E(S_{iB}\mid z_B)=\gamma_g s_i\langle z_B,t_B\rangle,\qquad
 \Var(S_{iB}\mid z_B)\le C m^4(q+\Lambda_B).
\]
On \eqref{poly:panel-overlap}, the signed mean in direction
$\eta s_i$ is at least $c\kappa qm$. Chebyshev's inequality bounds
$\Prob(\eta s_i S_{iB}\le0\mid z_B)$ by $C m^2/(\kappa^2 q)\le1/32$
after enlarging $C_0$, whenever $\Lambda_B\le1$. For two target nodes
$i,j$ external to $B$, its ballot is $+1$ when $S_{iB}S_{jB}\ge0$
and $-1$ otherwise. Both nodes use the same $\eta$ on
\eqref{poly:panel-overlap}. A union bound, including the probability
of failed initialization, shows that the ballot equals $s_i s_j$
with probability at least $1-3/32$, including when the targets belong
to the same group.

\paragraph{Panel majority.}
For each dyad discard any panel containing either target group.
Return the majority of the retained ballots, with ties and an empty
vote assigned $+1$, and denote the result by $\widehat d_{ij}^{\,R}$.
For this fixed dyad, distinct retained panels use disjoint edge
blocks: the internal group pairs of that panel and the pairs joining
its groups to the target groups. Thus the ballots are independent.
The number $N$ of retained panels is at least
$\rho/(A+1)-3$, since $q\le (A+1)m^2$.
For $\rho\ge8(A+1)$ this gives $N\ge\rho/[2(A+1)]$.
Set $\delta_\tau=1/[32(A+1)]$. When $\Lambda\le\delta_\tau\rho$,
at most $\Lambda\le N/16$ panels have $\Lambda_B>1$.
The expected number of correct ballots is then at least
$(15/16)(1-3/32)N>3N/4$. Hoeffding's inequality bounds the
probability of at most $N/2$ correct votes by $e^{-N/8}$.
Enlarging the prefactor to cover $\rho<8(A+1)$ yields, for every
ordered off-diagonal pair,
\begin{equation}\label{poly:resolution-vote}
 \Prob_{0,s}\{\widehat d_{ij}^{\,R}\ne s_i s_j\}
 \le C\exp\{-c\rho/(A+1)\},
 \qquad \Lambda\le\delta_\tau\rho .
\end{equation}

\subsection{A held-out mean score and selection between channels}
\label{sec:polynomial-mean-score}

We work at zero activity. For each distinct pair $i,j$ we construct a
sign $D^{\rm mean}_{ij}$ from the centered counts $W$ in
\eqref{aup:centered-counts}. There are absolute constants $C,c>0$ such
that, at every feasible composition,
\begin{equation}\label{poly:mean-vote}
 \Prob_{0,s}\{D^{\rm mean}_{ij}\ne s_i s_j\}
 \le C\exp\{-c\kappa^2\Lambda\}.
\end{equation}
The construction retains the contributions of the cells containing the
two endpoints, even when those cells account for all of $\|b\|_2^2$.

Let $U=\{g(i),g(j)\}$, let $H$ be its complement among the cells, and
let $I_H$ be the union of the cells in $H$. Include the coordinate
vectors $e_u$, $u\in U$, in a set $\mathcal D$ of candidate directions
in $\R^K$. If $H$ is nonempty, add a unit leading right singular vector
of $W_{I_H,H}$, extended by zero on $U$, and denote it by $d_H$.
Use fixed rules for ties throughout. For $k=i,j$, put
$X_k(d)=\langle W_{k,*},d\rangle$. Choose $\widehat d\in\mathcal D$
to maximize $\min\{|X_i(d)|,|X_j(d)|\}$ and return
$D^{\rm mean}_{ij}=\operatorname{sign}(X_i(\widehat d))
\operatorname{sign}(X_j(\widehat d))$, with $\operatorname{sign}(0)=1$.

For each coordinate candidate, $X_k(d)-\kappa s_k\langle b,d\rangle$
is a sum of independent centered edge variables with squared range
lengths summing to at most $m$. The deterministic diagonal correction
in $W$ gives the same centering when $k$ belongs to the selected cell.
The corresponding tail bound holds for $d_H$ conditional on
$W_{I_H,H}$. This matrix uses only edges internal to $I_H$, whereas
$X_k(d_H)$ uses edges joining $k$ to $I_H$. Thus Hoeffding's inequality
and a union bound give
\[
 \Prob_{0,s}\!\left\{
 \max_{d\in\mathcal D,\ k\in\{i,j\}}
 |X_k(d)-\kappa s_k\langle b,d\rangle|>t\right\}
 \le 12\exp(-2t^2/m).
\]

Write $B=\|b\|_2=\sqrt{m\Lambda}$. If
$\|b_U\|_2^2\ge B^2/2$, one of the at most two coordinate candidates
satisfies $|\langle b,e_u\rangle|\ge B/2$. Otherwise
$\|b_H\|_2>B/\sqrt 2$. In this case $n=|I_H|=m|H|$ satisfies
$n\ge\|b_H\|_2^2/m>\Lambda/2$. The net bound in
\eqref{aup:noise-event} applied to the adjacency matrix on $I_H$
gives
\[
 W_{I_H,H}=\kappa s_{I_H}b_H^\top+Z_H,\qquad
 \|Z_H\|_{\rm op}\le 4\sqrt{nm}
\]
with probability at least $1-2e^{-c_0 n}$, where $c_0=4-\log 9$.
On this event the variational characterization of $d_H$, tested
against $b_H/\|b_H\|_2$, gives
\[
 |\langle b,d_H\rangle|
 \ge \|b_H\|_2-\frac{2\|Z_H\|_{\rm op}}{\kappa\sqrt n}
 \ge B/\sqrt 2-8\sqrt m/\kappa .
\]
When $\kappa\sqrt\Lambda\ge32$, this is at least $B/4$.
In either case we have a candidate with absolute mean score at least
$\kappa B/4$ for both endpoints, except on an event of probability at
most $2e^{-c_0\Lambda/2}$. Take $t=\kappa B/16$. If all score errors
are at most $t$, the selected direction has both absolute scores at
least $3t$. Their signs then agree with the signs of their respective
means, whose product is $s_i s_j$. This proves \eqref{poly:mean-vote}
when $\kappa\sqrt\Lambda\ge32$. Increasing $C$ covers the remaining
range by the bound one.

Use $D^{\rm res}_{ij}=\widehat d_{ij}^{\,R}$ and $\delta_\tau$ from
\eqref{poly:resolution-vote}.
Set $\rho_0=256/(\kappa^2\delta_\tau)$. For $\rho<\rho_0$, use
$D^{\rm mean}_{ij}$. For $\rho\ge\rho_0$, use $D^{\rm mean}_{ij}$
if
\[
 \frac{\sigma_1(W)}{\sqrt{rm}}\ge
        \tfrac12\kappa\sqrt{\delta_\tau\rho},
\]
and use $D^{\rm res}_{ij}$ otherwise. Denote the resulting vote by
$D_{ij}$. The noise event \eqref{aup:noise-event} at zero activity
implies
\[
 \left|\frac{\sigma_1(W)}{\sqrt{rm}}-\kappa\sqrt\Lambda\right|\le4
\]
and fails with probability at most $2e^{-c_0r}$. On this event, for
$\rho\ge\rho_0$, selecting the mean score implies
$\Lambda\ge\delta_\tau\rho/16$, while selecting the second score
implies $\Lambda\le9\delta_\tau\rho/16$. In either permitted range
the corresponding unconditional error bound is at most
$C_\tau e^{-c_\tau(\Lambda+\rho)}$ after adjusting constants. Taking
the union of the two error events restricted to their respective
parameter ranges, and adding $2e^{-c_0r}$, gives
\begin{equation}\label{poly:adaptive-vote}
 \Prob_{0,s}\{D_{ij}\ne s_i s_j\}
 \le C_\tau e^{-c_\tau(\Lambda+\rho)} .
\end{equation}
Here $\Lambda,\rho\le r$, and for $\rho<\rho_0$ the bounded factor
$e^{c_\tau\rho_0}$ is absorbed into $C_\tau$.

\paragraph{From pairwise comparisons to the probability matrix.}
For the general-composition estimator set $\widehat u_a=1$ and
$\widehat u_i=D_{ia}$ for $i\ne a$, using the same fixed reference $a$.
At zero activity \eqref{poly:adaptive-vote} bounds each probability of
disagreement with $s_i s_a$. Applying \eqref{up:change-measure} to
each event replaces its bound by
$C_{\tau,M}\exp\{-c_\tau(\Lambda+\rho)/2\}$. By linearity of
expectation this also bounds the expected fraction of disagreements.
Nearest balanced rounding and \eqref{aup:sign-to-matrix} yield
\eqref{eq:poly-adaptive-risk} after changing constants.

There are at most $K$ distinct training complements for comparisons to
one reference node. Computing their singular vectors and the panel fits
uses at most $O(r^4)$ arithmetic operations and $O(r^2)$ working space
when complements are processed successively. The two frequencies require
only the integer powers of $e^{\mathrm{i}\theta}=(1+\mathrm{i}x)/(1-\mathrm{i}x)$
for margin values between $-m/2$ and $m/2$. They can be tabulated once.
For fixed model constants, the score bounds tolerate inverse-polynomial
numerical error. In particular,
the panel Chebyshev bound also holds with the threshold at half the
signed mean by increasing $C_0$, and the good mean scores have a
margin of order $\kappa\sqrt{m\Lambda}$.

The norm comparison in Section~\ref{sec:polynomial-mean-score} also
tolerates approximation without requiring the same branch at every input.
Suppose a deterministic nonnegative approximation $\widetilde\sigma$
satisfies
\[
 \frac{|\widetilde\sigma-\sigma_1(W)|}{\sqrt{rm}}\le1
\]
for every input, and replace $\sigma_1(W)$ by $\widetilde\sigma$ only in
that comparison. On the same zero-activity noise event, its normalized
deviation from $\kappa\sqrt\Lambda$ is at most $5$. For
$\rho\ge\rho_0$, put $t_\rho=\kappa\sqrt{\delta_\tau\rho}\ge16$.
Selection of the mean branch then implies
$\kappa\sqrt\Lambda\ge t_\rho/2-5\ge3t_\rho/16$, whereas selection of
the panel branch implies
$\kappa\sqrt\Lambda\le t_\rho/2+5\le13t_\rho/16$. Hence the two
selection regions satisfy, respectively,
\[
 \Lambda\ge\frac{9}{256}\delta_\tau\rho,
 \qquad
 \Lambda\le\frac{169}{256}\delta_\tau\rho<\delta_\tau\rho.
\]
The same unconditional error bounds therefore give
\eqref{poly:adaptive-vote}, and then \eqref{eq:poly-adaptive-risk},
after changing constants. This condition includes inverse-polynomial
error in the norm. The operation counts use the arithmetic model.
The output stores the balanced
sign vector together with the known parameter $\kappa$, which describes
both possible off-diagonal probabilities exactly.

\section{Recovery under uniformly allocated labels}\label{sec:average-recovery}

\subsection{Uniform allocation over the whole roster}

We now take $\alpha=0$ and put the uniform distribution on
\[
 \mathcal S=\{s\in\{-1,1\}^r:\one^\top s=0\},
\]
independently of the fixed partition $T$. Conditional on $s$, the edges
retain their independent Bernoulli laws. The averaged matrix risk is
\begin{equation}\label{tar:bayes-risk}
 \mathcal B_r^Y(T)=\inf_{\widehat P(Y,T)}
       \E_s\E_{A\mid s}L_r(\widehat P,P_{0,s}).
\end{equation}
\begin{theorem}[Uniformly allocated labels]\label{thm:average-recovery}
Let $r=Km\to\infty$, with equal cells of size $m\in4\N$.
Let $s$ be uniform on all globally balanced assignments independently
of $T$ and set $\alpha=0$. For $O\in\{Y,A\}$ write $\mathcal B_r^O(T)$
for the infimum of $\E_s\E_{A\mid s} L_r(\widehat P,P_{0,s})$ over estimators using
$O,T$. Then $\mathcal B_r^Y(T)$ tends to zero if and only if $K\to\infty$.
For bounded $K$ the risk has positive lower limit. With full adjacency
the averaged risk tends to zero throughout.
\end{theorem}

The random composition $b=T^\top s$ determines how much mean separation
is present in a particular draw. Its first two moments give a useful
description of this information.

\begin{lemma}[Moments of the composition energy]\label{tar:balanced-moments}
Let $s$ be uniform on $\mathcal S$. Its composition energy
$\Lambda=\|T^\top s\|_2^2/m$ satisfies
\begin{align}
 \E\Lambda&=\frac{r(K-1)}{r-1},\label{tar:lambda-mean}\\
 \Var(\Lambda)&=
 \frac{2r^2(m-1)(r-m)(r-2)}{m^2(r-1)^2(r-3)}.
 \label{tar:lambda-variance}
\end{align}
Consequently, if $K\to\infty$, then
$\Lambda/(K-1)\to1$ in $L^2$, uniformly over the admissible cell sizes.
\end{lemma}

\begin{proof}
Exchangeability and $\sum_i s_i=0$ give, for distinct indices,
\begin{equation}\label{tar:sign-moments}
 \E(s_i s_j)=-\frac1{r-1},\qquad
 \E(s_i s_j s_k s_l)=\frac3{(r-1)(r-3)}.
\end{equation}
For completeness, multiply $\sum_i s_i=0$ first by $s_1$ and then
by $s_1s_2s_3$, and take expectations. The resulting equations are
$1+(r-1)\E(s_1s_2)=0$ and
$3\E(s_1s_2)+(r-3)\E(s_1s_2s_3s_4)=0$.

Let $\mathcal E_T$ be the set of unordered pairs lying in the same
trait cell, and put
\[
 S_T=\sum_{\{i,j\}\in\mathcal E_T}s_i s_j,
 \qquad N=|\mathcal E_T|=\frac{r(m-1)}2.
\]
Expanding the cell squares yields $\Lambda=K+2S_T/m$. There are
$J=r\binom{m-1}{2}$ unordered pairs of distinct elements of
$\mathcal E_T$ that share one vertex. Writing
$\mu_2=-1/(r-1)$ and $\mu_4=3/\{(r-1)(r-3)\}$, we therefore have
\[
 \E S_T=N\mu_2,\qquad
 \E S_T^2=N+2J\mu_2+2\left\{\binom N2-J\right\}\mu_4.
\]
The first identity gives \eqref{tar:lambda-mean}. Substitution into
$\Var(\Lambda)=4\{\E S_T^2-N^2\mu_2^2\}/m^2$ and simplification
give \eqref{tar:lambda-variance}. For $K>1$, use $r-m=m(K-1)$ to obtain
\[
 \frac{\Var(\Lambda)}{(K-1)^2}
 =\frac{2r^2(m-1)(r-2)}{m(r-1)^2(r-3)(K-1)}
 \le\frac8{K-1}.
\]
Since $\E\Lambda/(K-1)=r/(r-1)$, this proves the $L^2$ assertion.
When $K=1$, $\Lambda=0$ identically, in agreement with both formulas.
\end{proof}

The conditional prior given $b$ is uniform on
$\mathcal S_b=\{s\in\mathcal S:T^\top s=b\}$. This prior has an exact
decision-theoretic relation to the fixed-composition problem.

\begin{lemma}[Conditional Bayes and minimax risks]\label{tar:conditional-symmetry}
For every feasible composition $b$,
\begin{align}
 \mathcal R_b^0
 &:=\inf_\delta\sup_{s\in\mathcal S_b}
          \E_{0,s}L_r(\delta(Y,T,b),P_{0,s})\notag\\
 &=\inf_\delta\frac1{|\mathcal S_b|}
       \sum_{s\in\mathcal S_b}
          \E_{0,s}L_r(\delta(Y,T,b),P_{0,s}).
 \label{tar:conditional-risk-identity}
\end{align}
The equality holds with either randomized or nonrandomized estimators.
\end{lemma}

\begin{proof}
Let $\mathcal G=\prod_{k=1}^K\operatorname{Sym}(G_k)$, represented by
permutation matrices. For $Q\in\mathcal G$, we have $QT=T$. The maps
\[
 A\mapsto QAQ^\top,\qquad Y\mapsto QY,\qquad
 s\mapsto Qs,\qquad P_{0,s}\mapsto QP_{0,s}Q^\top
\]
preserve the experiment and the loss. Moreover, $\mathcal G$ acts
transitively on $\mathcal S_b$: in each cell, the prescribed number of
positive signs can be moved to any set of that size.

For any estimator $\delta$, independently sample $Q$ uniformly from
$\mathcal G$ and use $Q^\top\delta(QY,T,b)Q$. Its risk at $s$ is
\[
 \frac1{|\mathcal G|}\sum_{Q\in\mathcal G}
       \E_{0,Qs}L_r(\delta(Y,T,b),P_{0,Qs}),
\]
which equals the uniform average over $\mathcal S_b$, since every point
of the orbit has the same number of preimages. Thus the minimax risk
is at most the optimal uniform-prior risk. The reverse inequality follows
by bounding an average risk by the maximum risk. Finally, conditional
averaging of the randomized estimator cannot increase squared loss, so
randomization does not change either infimum. The argument uses
transitivity on signs and does not require that distinct signs encode
distinct probability matrices.
\end{proof}

\begin{lemma}[A uniform conditional risk floor]\label{tar:bounded-composition-floor}
Fix finite $B,R\ge0$. For all sufficiently large $m$, uniformly over
feasible compositions with $\Lambda(b)\le B$ and $\rho=K/m^2\le R$,
\begin{equation}\label{tar:conditional-floor}
 \mathcal R_b^0\ge
 \frac{\kappa^2}{4}
       \exp\{-2C_{B,\tau}(B+R+1)\}>0.
\end{equation}
\end{lemma}

\begin{proof}
For $m\ge\max\{4,4B\}$, the uniform lower bound
\eqref{aup:composition-lower} applies with $M=0$ to every composition
with $\Lambda(b)\le B$. Its left side is $\mathcal R_b^0$, and its
right side is at least
\[
 \frac{\kappa^2}{4}
       \exp\{-2C_{B,\tau}(B+R+1)\},
\]
since $\rho\le R$, $1/m\le1$ and $r/(r-1)\ge1$.
Lemma~\ref{tar:conditional-symmetry} gives the same bound for the
conditional Bayes risk.
\end{proof}

\begin{proof}[Proof of Theorem~\ref{thm:average-recovery}]
\phantomsection
\label{proof:average-recovery}
First suppose $K\to\infty$. Lemma~\ref{tar:balanced-moments} implies
$\Lambda\to\infty$ in probability. The full-likelihood estimator in
Corollary~\ref{cor:exponential-rates} ignores $b$ and, at $M=0$, satisfies
conditionally on each $s$
\[
 \E_{0,s}L_r(\widehat P^{\rm full},P_{0,s})
 \le C\kappa^2e^{-\kappa^2(\Lambda+\rho)/1600}.
\]
The right side is bounded and tends to zero in probability under the
prior. Taking its expectation proves $\mathcal B_r^Y(T)\to0$.

Conversely, suppose $K\le L$ along an infinite subsequence, where
$L<\infty$. Then $m\to\infty$ and $\rho\to0$ on that subsequence.
By \eqref{tar:lambda-mean},
$\E\Lambda\le2(L-1)$, so Markov's inequality with $B=4L$ gives
$\Prob(\Lambda\le B)\ge1/2$. This also holds when $K=1$, since
$\Lambda=0$. Revealing $b$ to an estimator can only reduce its Bayes
risk. Conditional on this revelation, its risk is bounded below by the
optimal uniform-prior risk on $\mathcal S_b$. Hence
\begin{align*}
 \mathcal B_r^Y(T)
 &\ge \E_b\mathcal R_b^0\\
 &\ge \frac{\kappa^2}{8}
       \exp\{-2C_{4L,\tau}(4L+2)\}
\end{align*}
for all sufficiently large indices in the subsequence, by
\eqref{tar:conditional-floor} with $R=1$. If $K$ fails to diverge,
such a bounded subsequence exists, proving necessity for arbitrary
sequences.

For the adjacency comparison, the uniform bound
\eqref{eq:adjacency-bound} applies in particular at $\alpha=0$ and
has no dependence on $T$. Averaging that bound over $s$ proves
vanishing averaged adjacency risk along every admissible sequence.
\end{proof}

\subsection{What the counts save at exact group balance}\label{sec:no-count-risk}

\begin{proof}[Proof of Corollary~\ref{cor:no-count-risk}]
Give the assignments with $b=0$ the uniform prior. By
Lemma~\ref{tar:conditional-symmetry} the Bayes risk for the full count
array is $R_r(0)$. Fix distinct nodes $i,j$ and put $S=s_i s_j$, so
$p_{ij}=1/2+\kappa S$. Without counts the optimal estimate under the
prior is the mean of $p_{ij}$. In different groups the two labels are
independent and $\E S=0$. In the same group $\E S=-1/(m-1)$ by
sampling without replacement. The fraction of ordered pairs in the
same group is $(m-1)/(r-1)$. Thus the average of the prior variances
$\kappa^2\Var(S)$ is $B_{K,m}$ in the corollary. The matrix of prior
means is invariant under permutations within groups, so its loss is
constant on the orbit of assignments. It attains this risk also in
the minimax problem without counts.

Let $f_+,f_-$ be the probability mass functions of $Y$ conditional
on $S=+1,-1$, respectively. Global sign reversal preserves both $S$
and the count law, so we may orient the labels by $s_i=+1$. Draw an
assignment uniformly subject to $s_i=s_j=+1$, choose a negative node
uniformly in the group of $j$, and exchange its label with that of $j$.
There are $m/2$ choices from each starting assignment. In the reverse
direction every assignment with $s_i=+1,s_j=-1$ has $m/2$ choices if
the nodes belong to different groups. It has $m/2-1$ choices if they
belong to the same group, since $i$ must keep its label. The number
of preimages is therefore constant, giving the uniform law at each
end of the coupling. Each coupled pair has full-count affinity at
least $A_{K,m}$ by \eqref{low:full-affinity}. Cauchy--Schwarz applied
to the common mixing weights at each count array $y$ gives
$\aff(f_+,f_-)\ge A_{K,m}$ after summation over $y$.

Write $\pi=\Prob(S=+1)$ and $g=\pi f_++(1-\pi)f_-$. The posterior
probability of $S=+1$ at $y$ is $\pi f_+(y)/g(y)$ for $g(y)>0$. Hence
\[
 \frac{\E\Var(S\mid Y)}{\Var(S)}
 =\sum_{y:g(y)>0}\frac{f_+(y)f_-(y)}{g(y)}
 \ \ge\ \left(\sum_y\sqrt{f_+(y)f_-(y)}\right)^2
 \ \ge\ A_{K,m}^2 ,
\]
where the first inequality uses Cauchy--Schwarz and $\sum_y g(y)=1$.
The posterior means of all the $p_{ij}$ together attain the matrix
Bayes risk under squared loss. Multiplying by $\kappa^2\Var(S)$ and
averaging over pairs gives
\[
 A_{K,m}^2 B_{K,m}\ \le\ R_r(0)\ \le\ B_{K,m}.
\]
The upper bound follows by ignoring the counts. In
\eqref{low:affinity-constants}, $A_{K,m}=\beta_w\beta_c^{K-1}$ with
$1-\beta_w=O_\tau(m^{-1})$ and $1-\beta_c=O_\tau(m^{-2})$ by
\eqref{low:constants}. Since $K/m^2\to0$ implies $m\to\infty$,
$1-A_{K,m}\le (1-\beta_w)+(K-1)(1-\beta_c)\to0$.
\end{proof}

\section{Expected contact with a specified population}\label{sec:exposure}

Let $z\in[0,1]^r$ be deterministic and known. We study
\begin{equation}\label{tar:target-loss}
 \psi_z(P)=\frac{Pz}{r},\qquad
 L_z(\widehat\psi,\psi)=\frac1r\|\widehat\psi-\psi\|_2^2.
\end{equation}
For binary $z$, the $i$th coordinate is the expected number of contacts
from $i$ to the specified population, divided by the total population
size. Define its distance from the queried span by
\begin{equation}\label{tar:distance}
 d=d(T,z)=\frac1r\inf_{a\in\R^K}\|z-Ta\|_1.
\end{equation}
For a feasible composition $b$, write
\begin{equation}\label{tar:exposure-risk}
 \mathcal R_{r,z}^Y(T;b)
 =\inf_{\widehat\psi(Y,T,z,b)}
    \sup_{(\alpha,s)\in\Theta_r(M;b)}
       \E_{\alpha,s}L_z(\widehat\psi,\psi_z(P_{\alpha,s})).
\end{equation}

\begin{theorem}[Recovery of a fixed contact target]\label{thm:exposure-recovery}
Fix $M\ge0$ and $\tau$ satisfying \eqref{eq:signal}. For every sequence
of equal-cell partitions, deterministic vectors $z\in[0,1]^r$, and
feasible compositions $b$, with $r=Km\to\infty$,
\begin{equation}\label{tar:exposure-criterion}
 \mathcal R_{r,z}^Y(T;b)\longrightarrow0
 \quad\Longleftrightarrow\quad
 \frac{d(T,z)}{1+\Lambda(b)+\rho}\longrightarrow0.
\end{equation}
One estimator receiving only $Y,T,z,\tau$ uses $O(r^4)$ arithmetic
operations and attains the sufficient direction simultaneously over all
compositions. More precisely, it satisfies
\begin{equation}\label{tar:exposure-upper}
 \sup_{\Theta_r(M;b)}
    \E L_z(\widehat\psi,\psi_z(P))
 \le C_{\tau,M}\min\left\{d^2,\frac1{1+\Lambda(b)+\rho}\right\}
       +\frac1{4r}+\frac{M^2}{r(r-1)}+C_Me^{-cr}.
\end{equation}
\end{theorem}

\subsection{Target geometry and direct estimation}

At zero activity the hollow probability matrix is exactly
\[
 P_{0,s}=\frac12\one\one^\top+\kappa ss^\top
                       -\left(\frac12+\kappa\right)I.
\]
Put $\bar z=r^{-1}\sum_i z_i$ and $c_z(s)=r^{-1}s^\top z$. Then
\begin{equation}\label{tar:exact-target}
 \psi_z(P_{0,s})_i=\beta_i(z)+\kappa c_z(s)s_i,
 \qquad
 \beta_i(z)=\frac{\bar z}{2}
                   -\left(\frac12+\kappa\right)\frac{z_i}{r}.
\end{equation}

Choose a median $a_k\in[0,1]$ of the $z$-values in $G_k$, with a
fixed rule for ties, and set $u=Ta$. These medians minimize
$\|z-Ta\|_1$. Sorting the values in each cell gives
\begin{equation}\label{tar:median-duality}
 rd=\sum_{k=1}^K\left\{
       \sum_{j\in\operatorname{upper}(G_k)}z_j
       -\sum_{j\in\operatorname{lower}(G_k)}z_j\right\}
   =\max_{T^\top t=0,\ t\in\{-1,1\}^r}t^\top z,
\end{equation}
where each half has $m/2$ elements. To see the first equality, subtract
any median from the upper values and subtract the lower values from
that median; its contributions cancel. The second follows by assigning
the positive signs to the largest half in each cell. In particular,
$0\le d\le1/2$. For binary $z$,
\[
 d=\frac1r\sum_{k=1}^K\min\{n_{k,0},n_{k,1}\},
\]
where $n_{k,a}$ counts nodes in $G_k$ with $z_i=a$.

The direct estimator $\widehat\psi^D=Ya/r=Au/r$ has expectation
$Pu/r$. Since $0\le p_{ij}\le1$,
$|(P(u-z))_i|/r\le\|u-z\|_1/r=d$. Independence of the edges
incident to any fixed node gives
\begin{align}
 \E L_z(\widehat\psi^D,Pz/r)
 &=\frac1{r^3}\sum_{i\ne j}u_j^2p_{ij}(1-p_{ij})
       +\frac1{r^3}\|P(u-z)\|_2^2\notag\\
 &\le d^2+\frac1{4r}.
 \label{tar:direct-risk}
\end{align}
This calculation uses no independence between different rows. With
full adjacency, the estimator $Az/r$ has risk at most $1/(4r)$.

For any hollow matrix estimate $\widetilde P$, Cauchy--Schwarz gives
the deterministic contraction
\begin{equation}\label{tar:matrix-contraction}
 L_z(\widetilde Pz/r,Pz/r)
 \le\frac{(r-1)\|z\|_2^2}{r^2}L_r(\widetilde P,P)
 \le L_r(\widetilde P,P).
\end{equation}
Thus a matrix estimate provides a second route to the target, even when
$z$ varies substantially inside the queried cells.

\subsection{An observation-based selector}

Let $\widehat P$ be the hollow probability-matrix estimator of
Theorem~\ref{thm:poly-recovery}. Write
\[
 W=Y-\frac m2\one_r\one_K^\top
                  +\left(\frac12+\kappa\right)T,
\]
and define
\begin{equation}\label{tar:observed-information}
 H=1+\rho+\Lambda(b),\qquad
 \widehat H=1+\rho+\frac{\sigma_1(W)^2}{\kappa^2rm}.
\end{equation}
Use the estimator
\begin{equation}\label{tar:target-selector}
 \widehat\psi=
 \begin{cases}
   Ya/r,&d^2\le\widehat H^{-1},\\
   \widehat Pz/r,&d^2>\widehat H^{-1}.
 \end{cases}
\end{equation}
This rule requires only $Y,T,z,\tau$. Computing the group medians,
$\widehat H$, and the two candidate vectors adds at most $O(r^3)$
arithmetic operations to the $O(r^4)$ matrix estimator.

We verify the risk bound, including dependence between selection and
estimation. Write $W=\kappa sb^\top+Z$ as in
\eqref{aup:centered-counts} and put $D_0=4+M$ for the analysis.
The activity and operator bounds in Section~\ref{sec:recovery-tools}
give the event
\[
 \mathcal G=\{\|Z\|_{\mathrm{op}}\le D_0\sqrt{rm}\}
\]
with complement probability at most $2e^{-cr}$ uniformly over the
parameter class. On this event, singular-value perturbation gives
\begin{equation}\label{tar:singular-comparison}
 \left|\frac{\sigma_1(W)}{\sqrt{rm}}-\kappa\sqrt{\Lambda(b)}\right|
       \le D_0.
\end{equation}
Applying $(x+y)^2\le2x^2+2y^2$ in both directions shows that
\begin{equation}\label{tar:information-comparison}
 C_0^{-1}H\le\widehat H\le C_0H
 \quad\hbox{on }\mathcal G,
 \qquad C_0=2+2D_0^2/\kappa^2.
\end{equation}

On $\mathcal G$, selection of the direct branch is possible only if
$d^2\le C_0/H$. Its contribution to expected loss is therefore at most
\[
 \left(d^2+\frac1{4r}\right)\one_{\{d^2\le C_0/H\}}
 \le C_0\min\{d^2,H^{-1}\}+\frac1{4r},
\]
using its unconditional risk bound \eqref{tar:direct-risk}.
Since $\sup_{x\ge0}(1+x)\kappa^2e^{-c_\tau x}<\infty$,
Theorem~\ref{thm:poly-recovery} gives, after changing $C_{\tau,M}$,
\[
 \E_{\alpha,s}L_r(\widehat P,P_{\alpha,s})
 \le C_{\tau,M}H^{-1}+\frac{M^2}{r(r-1)}.
\]
Selection of the matrix branch on $\mathcal G$ is possible only if
$H^{-1}<C_0d^2$. By \eqref{tar:matrix-contraction}, its contribution
is therefore at most
\[
 C_{\tau,M}\min\{d^2,H^{-1}\}
       +\frac{M^2}{r(r-1)}.
\]
Both candidate vectors have coordinates in $[0,1]$, as does the target,
so the loss on $\mathcal G^c$ is at most one. Adding the three
contributions proves \eqref{tar:exposure-upper}.

\paragraph{Numerical precision of the selector.}
The same proof permits a deterministic nonnegative norm approximation
$\widetilde\sigma$ with
$|\widetilde\sigma-\sigma_1(W)|/\sqrt{rm}\le1$ for every input.
Define
\[
 \widetilde H=1+\rho+\frac{\widetilde\sigma^2}{\kappa^2rm}
\]
and use $\widetilde H$ in \eqref{tar:target-selector}.
On $\mathcal G$, the deviation in \eqref{tar:singular-comparison}
is then at most $D_0+1$, so
\[
 (C_0')^{-1}H\le\widetilde H\le C_0'H,
 \qquad C_0'=2+2(D_0+1)^2/\kappa^2.
\]
Thus both selection arguments apply with $C_0'$ in place of $C_0$,
even if numerical error changes the chosen branch. The bound
\eqref{tar:exposure-upper} holds with an adjusted $C_{\tau,M}$ and
coefficient one on each of $1/(4r)$ and $M^2/[r(r-1)]$.

\subsection{Separation at a prescribed composition}

We next construct a separated family for a fixed $z$ and $b$. It
suffices for the necessity argument to work along a subsequence on which
\begin{equation}\label{tar:hard-subsequence}
 d\ge\epsilon>0,\qquad \Lambda(b)\le B,\qquad\rho\le R,
\end{equation}
for fixed finite $B,R$. These bounds and $r=Km\to\infty$ imply
$m\to\infty$. Set
\[
 e=\frac{\|b\|_1}{r}\le\sqrt{\frac{\Lambda(b)}r}=o(1).
\]
In cell $G_k$, assign positive signs to the largest $(m+b_k)/2$
values of $z$, and negative signs to the rest. Denote this sign vector
by $s^*$ and put $c^*=r^{-1}(s^*)^\top z$. Starting from the maximizing
cell-balanced signs in \eqref{tar:median-duality} requires changing
$|b_k|/2$ signs in cell $k$. Since $z_i\in[0,1]$, this proves
\begin{equation}\label{tar:prescribed-maximizer}
 c^*\ge d-e.
\end{equation}

Pair each minority-sign node with a distinct majority-sign node in its
cell. Order every pair as $(u_j,v_j)$ with $s^*_{u_j}=1$ and
$s^*_{v_j}=-1$. The sorting construction ensures the nonnegative gaps
$g_j=z_{u_j}-z_{v_j}\ge0$. There are
\[
 N_p=\frac{r-\|b\|_1}{2}
\]
pairs, with $|b_k|$ unpaired majority-sign nodes in cell $k$. Let
$\xi$ be their total signed $z$-sum divided by $r$. Then
$|\xi|\le e$ and
$\sum_{j=1}^{N_p}g_j=r(c^*-\xi)$.

Select $q=\lfloor N_p/4\rfloor$ pairs having the smallest gaps. Their
gap sum is at most $r(c^*-\xi)/4$. Vary only the orientations of these
pairs, retaining all other signs from $s^*$, and set $\alpha=0$.
Every resulting sign vector $s(\theta)$ has the same composition $b$.
Its target correlation obeys
\begin{equation}\label{tar:correlation-floor}
 c_z(s(\theta))
 \ge c^*-\frac{c^*-\xi}{2}
 =\frac{c^*+\xi}{2}\ge\frac d2-e\ge\frac d4
\end{equation}
eventually, by \eqref{tar:hard-subsequence}. Also $q/r\to1/8$, so
$q/r\ge1/16$ eventually.

Fix all orientation bits except that of a selected pair $(u_j,v_j)$.
Let $c_+,c_-$ be the correlations at the two neighboring corners.
The exact diagonal terms in \eqref{tar:exact-target} cancel, and on
these two coordinates the target difference is
\[
 \bigl(\psi_{u_j}^+-\psi_{u_j}^-,\psi_{v_j}^+-\psi_{v_j}^-\bigr)
       =\kappa(c_++c_-)(1,-1).
\]
Its squared norm is at least $\kappa^2d^2/2$ by
\eqref{tar:correlation-floor}. The full $Y$ laws of these corners
have affinity at least
$\exp\{-C_{B,\tau}(B+R+1)\}$ by
Proposition~\ref{prop:composition-affinity}. Apply the two-point
squared-error inequality \eqref{low:two-point-risk}
to this pair of coordinates, average over the other bits, and sum over
the disjoint selected pairs. Since $1-\TV\ge\aff^2/2$, we obtain
\begin{equation}\label{tar:target-floor}
 \mathcal R_{r,z}^Y(T;b)
 \ge\frac{\kappa^2d^2}{256}
        \exp\{-2C_{B,\tau}(B+R+1)\}
\end{equation}
eventually. The known quantities $T,z,b$ are identical at every corner.

\begin{proof}[Proof of Theorem~\ref{thm:exposure-recovery}]
The bound \eqref{tar:exposure-upper} was proved for the estimator
\eqref{tar:target-selector}. For $H=1+\Lambda(b)+\rho$,
\[
 \min\{d^2,H^{-1}\}\le(d/H)^{2/3}.
\]
Indeed, the minimum of two nonnegative numbers is at most their weighted
geometric mean, here with weights $1/3$ and $2/3$. Thus $d/H\to0$
implies consistency, including sequences on which the two branches
alternate.

If $d/H$ fails to tend to zero, take a subsequence on which
$d/H\ge\epsilon>0$. Because $0\le d\le1/2$, on this subsequence
$d\ge\epsilon$ and $\Lambda(b)+\rho\le1/(2\epsilon)-1$.
Consequently \eqref{tar:hard-subsequence} holds with fixed finite
$B,R$, and \eqref{tar:target-floor} bounds the risk away from zero.
This proves necessity.
\end{proof}

\subsection{A fixed target under uniform random labels}

The uniform-label experiment of Section~\ref{sec:average-recovery}
has a simpler behavior for a fixed contact target. Write
$v_z=r^{-1}\sum_i(z_i-\bar z)^2$.

\begin{proposition}[Consistency without count data]\label{tar:random-exposure}
Suppose $s$ is uniform on $\mathcal S$, independently of deterministic
$z\in[0,1]^r$. At zero activity, the known vector $\beta(z)$ in
\eqref{tar:exact-target} satisfies
\begin{equation}\label{tar:no-data-exact}
 \E_s L_z(\beta(z),\psi_z(P_{0,s}))
       =\frac{\kappa^2v_z}{r-1}\le\frac{\kappa^2}{4(r-1)}.
\end{equation}
For any choice of admissible activity $\alpha(s)$ at each sign vector,
\begin{equation}\label{tar:no-data-activity}
 \E_s L_z(\beta(z),\psi_z(P_{\alpha(s),s}))
 \le\frac{2\kappa^2v_z}{r-1}+\frac{M^2}{r(r-1)}.
\end{equation}
\end{proposition}

\begin{proof}
Symmetry gives $\E c_z(s)=0$. Using \eqref{tar:sign-moments},
\begin{align*}
 \E c_z(s)^2
 &=\frac1{r^2}\left\{\sum_i z_i^2
          -\frac1{r-1}\sum_{i\ne j}z_i z_j\right\}\\
 &=\frac1{r-1}\left\{\frac1r\sum_i z_i^2-\bar z^2\right\}
 =\frac{v_z}{r-1}.
\end{align*}
Equation~\eqref{tar:exact-target} gives
$L_z(\beta(z),\psi_z(P_{0,s}))=\kappa^2c_z(s)^2$ exactly.
Since $z_i\in[0,1]$, $v_z\le1/4$, proving
\eqref{tar:no-data-exact}. Finally,
\eqref{up:target-transfer} and \eqref{tar:matrix-contraction}
give, pointwise in $s$,
\[
 L_z(\psi_z(P_{\alpha(s),s}),\psi_z(P_{0,s}))
 \le L_r(P_{\alpha(s),s},P_{0,s})
 \le\frac{M^2}{2r(r-1)}.
\]
Apply $\|x+y\|_2^2\le2\|x\|_2^2+2\|y\|_2^2$ and average
over $s$ to obtain \eqref{tar:no-data-activity}.
\end{proof}

\section*{Statement on AI use}
An AI assistant accessed through OpenAI Codex contributed to formulating
the mathematical results and developing the proofs, including the
lower-bound argument and the polynomial-time recovery construction. It
drafted and revised complete sections of the manuscript and wrote code
for finite checks of likelihood identities, count reconstructions and
panel scores. Claude and Gemini, the latter accessed through Antigravity,
also contributed manuscript readings and checks of arguments and
exposition. AI tools assisted in locating references.

\bibliographystyle{plainnat}
\bibliography{references}
\end{document}